\documentclass[11pt,a4paper]{article}
\usepackage[utf8]{inputenc}
\usepackage{microtype}
\usepackage{amsmath,amsthm,amssymb}
\usepackage{mathrsfs} 
\usepackage[T1]{fontenc}
\usepackage{ae,aecompl}
\usepackage{slashed}
\usepackage{times}
\usepackage[british]{babel}
\usepackage{aliascnt}
\usepackage[colorlinks,pagebackref,hypertexnames=false]{hyperref} 
\usepackage{url}
\usepackage[numeric,backrefs]{amsrefs}
\usepackage{nth}
\usepackage{float}
\usepackage[all]{xy}
\usepackage{tikz-cd} 
\usepackage{bookmark}
\usepackage[margin=2cm]{geometry}

\usepackage{enumerate}

\newcommand{\R}{\mathbb{R}}
\newcommand{\C}{\mathbb{C}}

\newcommand{\SO}{{\rm SO}}
\newcommand{\Sp}{{\rm Sp}}

\newcommand{\SU}{{\rm SU}}

\newcommand{\Rm}{{\rm Rm}}

\renewcommand{\epsilon}{\varepsilon}

\newcommand{\Ric}{{\rm Ric}}

\newcommand{\Scal}{\rm Scal}

\newcommand{\vol}{\mathrm{vol}}

\def\<{\mathopen{}\left<}
\def\>{\right>\mathclose{}}
\def\({\mathopen{}\left(}
\def\){\right)\mathclose{}}

\usepackage{multicol, color}

\definecolor{gold}{rgb}{0.85,.66,0}
\definecolor{cherry}{rgb}{0.9,.1,.2}
\definecolor{burgundy}{rgb}{0.8,.2,.2}
\definecolor{orangered}{rgb}{0.85,.3,0}
\definecolor{orange}{rgb}{0.85,.4,0}
\definecolor{olive}{rgb}{.45,.4,0}
\definecolor{lime}{rgb}{.6,.9,0}
\definecolor{green}{rgb}{.2,.7,0}
\definecolor{grey}{rgb}{.4,.4,.2}
\definecolor{brown}{rgb}{.4,.3,.1}

\newtheorem{theorem}{Theorem}
\newtheorem{proposition}[theorem]{Proposition}
\newtheorem{corollary}[theorem]{Corollary}
\newtheorem{lemma}[theorem]{Lemma}

\theoremstyle{remark} 
\newtheorem{remark}{Remark}

\theoremstyle{definition}
\newtheorem{definition}[theorem]{Definition}
\newtheorem{example}[theorem]{Example}

\def\rmG{\mathrm{G}}
\def\al{\alpha}

\def\w{\wedge}
\def\R{\mathbb{R}}
\def\C{\mathbb{C}}

\def\Lm{\Lambda}
\def\lm{\lambda}
\def\om{\omega}
\def\Om{\Omega}
\def\vp{\varphi}
\def\ip{\raise1pt\hbox{\large$\lrcorner\ \!$}} 

\usepackage{enumitem}
\usepackage{mathtools}

\usepackage[textwidth=1.7cm,textsize=scriptsize]{todonotes}
\usepackage{comment}

\numberwithin{theorem}{section}
\numberwithin{remark}{section}

\allowdisplaybreaks
\date{}

\usepackage{titling}
\usepackage{changepage}

\usepackage{authblk}

\begin{document}

\title{Some remarks on strong $\rmG_2$-structures with torsion}

\author{Anna Fino, \qquad \qquad 
Udhav Fowdar}

\affil{Dipartimento di Matematica ``G. Peano'' \\ Universit\`a degli Studi di Torino\\ Via Carlo Alberto 10, 10123 Torino, Italy\\ annamaria.fino@unito.it; \ udhav.fowdar@unito.it}





\maketitle
	
\vspace{-0.5cm}	

\begin{abstract}
A $\rmG_2$-structure on a $7$-manifold $M$ is called a $\rmG_2T$-structure if $M$ admits a $\rmG_2$-connection $\nabla^T$ with totally skew-symmetric torsion $T_\vp$. 
If furthermore, $T_\vp$ is closed then it is called a strong $\rmG_2T$-structure.
In this paper we investigate the geometry of (strong) $\rmG_2T$-manifolds in relation to its curvature, $S^1$ action and almost Hermitian structures. 
In particular, we study the Ricci flatness condition of $\nabla^T$ and give an equivalent characterisation in terms of geometric properties of the $\rmG_2$ Lee form. Analogous results are also obtained for almost Hermitian $6$-manifolds with skew-symmetric Nijenhuis tensor. Moreover, by considering the $S^1$ reduction by the dual of the $\rmG_2$ Lee form, we show that Ricci-flat strong $\rmG_2T$-structures correspond to solutions of the $\SU(3)$ heterotic system on certain almost Hermitian half-flat $6$-manifolds. Many explicit examples are described and in particular, we construct the first examples of strong $\rmG_2T$-structures with $\nabla^T$ not Ricci flat. Lastly, we classify $\rmG_2$-flows inducing gauge-fixed solutions to the generalised Ricci flow akin to the pluriclosed flow in complex geometry. The approach is this paper is based on the representation theoretic methods due to Bryant in \cite{Bryant06someremarks}.

\end{abstract}	
\begin{adjustwidth}{0.95cm}{0.95cm}
    \tableofcontents
\end{adjustwidth}

\section{Introduction}

Riemannian manifolds admitting  metric connections with totally skew-symmetric torsion and special holonomy have received considerable interest in both mathematics and theoretical physics, particularly in the study of supersymmetric string theories and supergravity, see for instance \cites{Agricola06, FriedrichIvanov02, FriedrichIvanov03, IvanovPapadopoulos01}.
The key link is given by the Hull-Strominger system, introduced in \cites{Hull1986, STROMINGER1986}, which describes the supersymmetric background in heterotic string theories. In dimension seven, this system is also known as the $\rmG_2$-Hull-Strominger system (or heterotic $\rmG_2$ system) and has been investigated extensively in \cites{Xenia15, Xenia18, LotaySaEarp, Martelli04, IvanovIvanov2005, Gemmer2013, Clarke2022, Fernandez2015, Fernandez11}. More recently, it was generalised in \cite{Lotay2024}, motivated by advances in theoretical physics and generalised geometry.

Our perspective in this paper is primarily motivated by the analogy to SKT, or pluriclosed, geometry. A $\rmG_2$-structure $\varphi$ admits
a metric connection $\nabla^T$ with torsion $3$-form $T_\vp$ preserving the $\rmG_2$-structure if and only if $d *_\vp \varphi= \theta  \wedge *_\vp \varphi,$ where $\theta$ denotes the $\rmG_2$ Lee form. Such types of $\rmG_2$-structures are called $\rmG_2$-structures with torsion (shortly $\rmG_2T$) or sometimes `integrable' $\rmG_2$-structures. The $\rmG_2$-connection $\nabla^T$ is unique and is often called the characteristic connection \cite{FriedrichIvanov02}. If the torsion form $T_\vp$ is closed then $\vp$ is said to be a \textit{strong} $\rmG_2T$-structure. A strong K\"ahler manifold with torsion (shortly SKT) is a Hermitian manifold admitting a Hermitian connection $\nabla^B$ with closed torsion $3$-form $T_\om$; this connection is also unique and is often called the Bismut connection \cite{Bismut1989}. Given the close analogy between these two definitions, it is natural to ask what properties these two classes of special geometric structures have in common. In contrast to SKT geometry, however, fairly little is known about strong $\rmG_2T$ manifolds. For instance, to date the only known non-trivial examples of strong $\rmG_2T$ manifolds are $S^3 \times S^3 \times S^1$ and  $S^3 \times N^4$, where $N^4$ denotes a hyperK\"ahler $4$-manifold; 
on the other hand, there are plenty of examples of SKT manifolds cf. \cite{FinoTomassini}. Strong $\rmG_2T$-structures can be viewed as a limit of the aforementioned heterotic $\rmG_2$ system when the `string scale' goes to zero cf. \cite{Mario2024}. In \cite{StreetsTian10} Streets-Tian introduced the pluriclosed flow, an evolution equation for Hermitian metrics preserving the SKT condition. Furthermore, they showed that the pluriclosed flow can be viewed as a gauge-fixed solution to the generalised Ricci flow, see also \cite{FernandezStreetsBook}. Thus, it is natural to ask if such an analogous flow also exists in the context of $\rmG_2$ geometry. In this paper we attempt to give a reasonable answer to all these questions. \smallskip


\textbf{Main results.}
By deriving new explicit expressions for various $\rmG_2$ and $\SU(3)$ invariant quantities (Riemannian and characteristic curvatures, derives of intrinsic torsion, Lie derivatives by the Lee form,...) in terms of first and second invariants we are able to obtain several new relations between these quantities. A key advantage of our approach, which is based on \cite{Bryant06someremarks}, is that most of our results hold locally i.e. no compactness or completeness assumptions are needed; this method also circumvents tedious index calculations and is coordinate free. As most of our formulae do not require the strong condition, the expressions derived here should be of interest to readers interested broadly in $\rmG_2$ and $\SU(3)$ geometry. Moreover, since we express each quantity in terms of its irreducible $\rmG_2$ and $\SU(3)$ components, our proofs illustrate clearly where each hypothesis is needed. Thus, we are able to improve on several previous results in the literature by minimising certain assumptions. 
We summarise below some of our main results:
\begin{itemize}
    \item In Theorem \ref{thm: main without strong}, we show that for $\rmG_2T$ manifolds of constant type, the $\rmG_2$ Lee form $\theta$ is $\nabla^T$-parallel if and only if $T_\vp$ is co-closed and $\theta^\sharp$ preserves $\vp$. By further imposing the strong condition, we show in Theorem \ref{cor: main corollary} that this is equivalent to the vanishing of the characteristic Ricci tensor. Some important consequences are given in \ref{strong ricci flat constant tau1}-\ref{strong ricci flat g2t scal}.
    
    \item For almost Hermitian $6$-manifolds with skew-symmetric Nijenhuis tensor, we prove analogues of the aforementioned results. More precisely, in Theorem \ref{thm: main thm in su3 case without strong}, we show that if the Bismut-Ricci form $\rho^B$ and $dT_\om$ are $J$-invariant then the Lee form $\nu_1$ is $\nabla^B$-parallel if and only if $T_\om$ is co-closed and $\nu_1^\sharp$ preserves the Hermitian structure. If furthermore, $\rho^B = dT_\om=0$ then we show in Theorem \ref{thm: main thm in su3 case} this is also equivalent to the vanishing of the characteristic Ricci tensor.  Some important consequences are given in \ref{consequence 1 aSKT}-\ref{consequence 3 aSKT}. This generalises some known results in the compact complex case cf. \cite{ApostolovLeeStreets2024}.
    
    \item In Theorem \ref{thm: gibbons hawking ansatz} we derive a Gibbons-Hawking ansatz type result for $S^1$-invariant $\rmG_2T$ manifolds. Under certain further assumptions, we show how one can relate strong $\rmG_2T$-structures with almost SKT $6$-manifolds.
    Since by Theorem \ref{cor: main corollary} the dual of the $\rmG_2$ Lee form $\theta^\sharp$ preserves $\vp$ for a characteristic Ricci flat strong $\rmG_2T$-structure, in Theorem \ref{thm: gibbons hawking strong ricciT=0} we obtain a Gibbons-Hawking type theorem for this action. We show that a characteristic Ricci-flat strong $\rmG_2T$-structure with non-vanishing Lee form corresponds to an abelian Hermitian Yang-Mills connection over a half-flat $\SU(3)$-structure solving the `$\SU(3)$ heterotic Bianchi identity' (\ref{equ: heterotic}). Our result shows that the quotient structure is in fact never complex but instead has non-vanishing skew-symmetric Nijenhuis tensor. Some important consequences are given in \ref{consequence 1}-\ref{consequence 4}. 

    \item All known compact examples of strong $\rmG_2T$ manifolds are characteristic Ricci-flat. 
    In \S \ref{sec: explicit example} we show that on $S^3 \times S^3 \times S^1$ there are two isometric but distinct strong $\rmG_2 T$-structures. While both have non-vanishing $\rmG_2$ Lee form, for only one of them it is closed, as such these structures are not diffeomorphic. In particular, these two examples arise from different initial data in Theorem \ref{thm: gibbons hawking strong ricciT=0}. We should point out that the product of $S^3$ and a K3 surface admits a characteristic Ricci-flat strong $\rmG_2T$-structure with vanishing $\rmG_2$ Lee form and provides the only known compact example for which $\nabla^T$ is non-flat, see Example \ref{example: N4 x S3}. 
    In \S \ref{sec: lots of examples} we show that there are local inhomogeneous examples whereby $T_\vp$ is both closed and co-closed but $\Ric^T \neq 0$. In particular, we exhibit an explicit example with $\SU(2)^3$-symmetry on an open set of the spinor bundle of $S^3$. Our result also shows that there are no complete examples with this symmetry group aside from the torsion free Bryant-Salamon solution \cite{Bryant1989}; thus, this leaves open the problem of finding \textit{complete} examples with smaller symmetry group. {It is worth noting that all the previously known examples of strong $\rmG_2T$ metrics have product metrics whereas our local examples have holonomy group equal to $\SO(7)$.} We also construct an example whereby $T_\vp$ closed but not co-closed.

    \item In Section \ref{sec: flows of g2} we consider two classes of $\rmG_2$-flows. First we consider flows of co-closed $\rmG_2$-structures (which in particular, are $\rmG_2T$-structures) and extend the result of Grigorian in \cite{Grigorian2013short}. More precisely, we show in Theorem \ref{thm: short time existence generalisation} that there is a 1-parameter family deformation of the modified Laplacian co-flow which has short time existence; the critical points of the flow depends on this parameter and are given in Proposition \ref{prop: critical point of coclosed flows}. We expect that these flows might be supplemented by a flow of connections to construct an anomaly type flow cf. \cites{Phong2024, Ashmore2024}.
    Secondly, akin to the pluriclosed flow, we consider flows of $\rmG_2$-structures which induce the generalised Ricci flow (up to gauge-fixing by the Lee form). By appealing to the recent result in \cite{Dwivedi2023} we establish short time existence and uniqueness for a family of such $\rmG_2$-flows, see Corollary \ref{cor: main result flow}.
\end{itemize}

\smallskip

\textbf{Acknowledgements}. 
The authors would like to thank Simon Salamon for several helpful comments {and also Alessandro Tomasiello for bringing \cite{Passias2021} to our attention. 
The authors are partially supported  by GNSAGA (Indam). Anna Fino is also  supported by Project PRIN 2022 \lq \lq Geometry and Holomorphic Dynamics”  and by a grant from the Simons Foundation (\#944448). 

\smallskip

\section{Preliminaries on 
\texorpdfstring{$\SU(3)$}{}- and 
\texorpdfstring{$\rmG_2$}{}-structures}\label{sec: preliminaries}
The purpose of this section is to gather some basic facts about $\SU(3)$- and $\rmG_2$-structures, and to fix the notation that will be used throughout this article.

\subsection{Background on \texorpdfstring{$\SU(3)$}{}-structures} Let $(P,\om,\Om)$ denote a $6$-manifold endowed with an $\SU(3)$-structure determined by the real $(1,1)$-form $\om$ and a complex $(3,0)$-form $\Om = \Om^+ + i\Om^-$. We denote the induced metric by $g_{\om}$, the almost complex structure by $J$ and the Hodge star operator by $*_\om$. As $\SU(3)$ modules, the space of differential forms $\Lm^\bullet(P)$ decompose into irreducible representations:
\begin{align*}
\Lm^1(P) &= \Lm^1_6,\\ 
\Lm^2(P) &= \langle\om\rangle \oplus\Lm^2_6 \oplus \Lm^2_{8},\\
\Lm^3(P) &= \langle\Om^+\rangle\oplus \langle\Om^-\rangle \oplus \Lm^3_6 \oplus \Lm^3_{12}.
\end{align*}
A more concrete description of these irreducible modules is given by
\begin{align*}
\Lm^2_6&=\{\al\in \Lm^2\ |\ *_\om(\al \w \om)=+\al\},\\
\Lm^2_{8}&=\{\al\in \Lm^2\ |\ *_\om(\al \w \om)=-\al\}, \\
\Lm^3_6&=\{\al\w \om\ |\ \al \in \Lm^1_6\},\\
\Lm^3_{12}&=\{\al\in \Lm^3\ |\ \al \w \om=0,\ \al \w \Om^{\pm}=0 \}. 
\end{align*}
\noindent Given a $k$-form $\alpha$ we shall denote by either $(\al)^k_p$ or $\pi^k_p(\alpha)$ its projection in the space $\Lm^k_p$. Following \cite{Bedulli2007}, see also \cite{ChiossiSalamonIntrinsicTorsion}, we define the $\SU(3)$ torsion forms $\sigma_0,\pi_0 \in C^\infty (P)$,  
$\nu_1,\pi_1 \in \Lm^1_6$, 
$\kappa_2,\pi_2 \in \Lm^2_8$, 
$\nu_3 \in \Lm^3_{12}$ by
\begin{align}
    d\om &= -\frac{3}{2}\sigma_0 \Om^+ + \frac{3}{2}\pi_0 \Om^- + \nu_1 \w \om + \nu_3,\label{equ: torsion su3 1} \\
    d\Om^+ &= \pi_0 \om^2 + \pi_1 \w \Om^+ - \pi_2 \w \om,\label{equ: torsion su3 2}  \\
    d\Om^- &= \sigma_0 \om^2 + J\pi_1 \w \Om^+ - \sigma_2 \w \om. \label{equ: torsion su3 3} 
\end{align}
Note that up to a constant factor $\nu_1$ corresponds to the \textbf{Lee form}. Various types of $\SU(3)$-structures can be defined by requiring that certain components of the torsion forms vanish. Of main relevance to us here are the following ones:
\begin{definition}
    The $\SU(3)$-structure determined by $(\om,\Om)$ is called
    \begin{enumerate}
        \item {\bf Hermitian} if $\pi_0=\sigma_0=0$ and $\pi_2=\sigma_2=0$, cf. \cite{ChiossiSalamonIntrinsicTorsion};
        \item {\bf nearly K\"ahler} if $\sigma_0=-2$ and all other torsion forms vanish, cf. \cite{Gray1976}; 
        \item {\bf half-flat} if $\pi_0=0$, $\nu_1=\pi_1=0$ and $\pi_2=0$, cf. \cite{ChiossiSalamonIntrinsicTorsion};
        \item {\bf co-coupled} or {\bf double half-flat} if in addition to half-flat $\sigma_2=0$, cf. \cite{MadsenSalamon}.
    \end{enumerate}
\end{definition}
\noindent The Nijenhuis tensor of the almost complex $J$ is defined by
\begin{equation*}
    N_{J}(X,Y) := [X,Y] + J[JX,Y] + J[X,JY] - [JX,JY],
\end{equation*}
where $X,Y$ denote vector fields on $P$. We can identify $N_J$ with a $(3,0)$-tensor by
\begin{equation}
    \hat{N}(X,Y,Z) :=  g_{\om}(N_{J}(X,Y),Z).
\end{equation}
In \cite{FriedrichIvanov02}*{Theorem 10.1}
it is shown that an almost Hermitian manifold admits a Hermitian connection  (i.e. a connection preserving both the metric and almost complex structure) with totally skew-symmetric torsion if and only if the Nijenhuis tensor $\hat{N}$ is a 3-form. Moreover, this connection is unique and its torsion $3$-form is given by
\begin{equation}
T_{\om}(X,Y,Z) = d\om (JX,JY,JZ) - \hat{N}(X,Y,Z). \label{equ: torsion of skew connection almost hermtian}
\end{equation}
\begin{remark}
    Expression (\ref{equ: torsion of skew connection almost hermtian}) is the same as given in \cite{FriedrichIvanov02}*{Theorem 10.1} but appears to differ by a minus sign owing to different conventions.
\end{remark}
We define $d^c\om(X,Y,Z):=J(d\om)(X,Y,Z)=d\om(JX,JY,JZ)$ and simply write
\begin{equation*}
T_{\om} = d^c \om - \hat{N}.
\end{equation*}
We shall denote this Hermitian connection by $\nabla^B$ and the Levi-Civita one by $\nabla$. 
\begin{definition}
    Let $(P,\om,J,g_\om)$ be an almost Hermitian manifold with skew-symmetric Nijenhuis tensor $\hat{N}$. Then we say $(P,\om,J,g_\om)$ is an {\bf almost strong K\"ahler} manifold with torsion (`almost SKT' for short) if $dT_{\om}=0$.
\end{definition}

\begin{remark}  
    When $(P,\om,J,g_\om)$ is Hermitian i.e. $\hat{N}=0$, $\nabla^B$ is called the Bismut  or Strominger connection \cites{Bismut1989, STROMINGER1986}. In this case, the strong condition is equivalent to $\partial\bar{\partial}\om=0$; this is also called pluriclosed or SKT in the literature.  
\end{remark}
We define the (first) Bismut-Ricci form $\rho^B$ by
\begin{equation}
    \rho^B(X,Y) :=\frac{1}{2}\sum_{i=1}^6 g_{\om}(\Rm^B(X,Y)Je_i,e_i),
\end{equation}
where $\Rm^B(X,Y)Z:=\nabla^B_X (\nabla^B_Y Z)- \nabla^B_Y (\nabla^B_X Z)- \nabla^B_{[X,Y]}Z$ and $\{e_i\}$ is a local orthonormal frame. 
\begin{definition}
    Assume that $\hat{N}$ is a $3$-form.
    \begin{itemize}
        \item If $\Rm^B =0$, we say $(P,\om,J,g_\om)$ is {\bf almost Bismut flat}.
        \item If $\rho^B=0$, we say $(P,\om,J,g_\om)$ is {\bf almost Calabi-Yau with torsion} (`almost CYT' for short).
    \end{itemize}
    When $\hat{N}=0$, we omit the qualifier `almost'.
\end{definition}
Note that there are different traces of $\Rm^B$ that one can consider in the almost Hermitian case. Of particular interest to us here will be the so-called \textbf{characteristic Ricci curvature} defined by
\begin{equation}
    \Ric^B(X,Y) := \sum_{i=1}^6g_\om(\mathrm{Rm}^B(e_i,X)Y,e_i).
\end{equation}
In Section \ref{sec: su3 curvature} we shall derive formulae expressing $\rho^B$ and $\Ric^B$ in terms of the $\SU(3)$ torsion forms. This allows one to efficiently compute these curvature tensors in explicit examples.

\subsection{Background on \texorpdfstring{$\rmG_2$}{}-structures} 
Let $(M,\vp,g_{\vp})$ denote a $7$-manifold endowed with a $\rmG_2$-structure determined by the $3$-form $\vp$. We recall that $\vp$ determines both the metric and an orientation by
\begin{equation}
    6 g_\vp(X,Y) \vol_\vp = (X \ip \vp) \w (Y \ip \vp) \w \vp,
\end{equation}
where $\vol_\vp$ denotes the volume form and $\ip$ denotes contraction. We shall write $*_\vp$ for the associated Hodge star operator. As $\rmG_2$ modules, the space of differential forms $\Lm^\bullet(M)$ decompose into irreducible representations:
\begin{align*}
\Lm^1(M) &= \Lm^1_7,\\
\Lm^2 (M) &= \Lm^2_7 \oplus \Lm^2_{14}\\
\Lm^3(M) &= \langle \vp \rangle \oplus \Lm^3_7 \oplus \Lm^3_{27},
\end{align*}
where the subscript denotes the dimension of the irreducible module. Using the Hodge star operator we get the corresponding splitting for $\Lm^4$, $\Lm^5$ and $\Lm^6.$ A more concrete description of these modules is given by:
\begin{align*}
\Lm^2_7&=\{\al\in \Lm^2\ |\ *_\vp(\al \w \vp)=2\al\},\\
\Lm^2_{14}&=\{\al\in \Lm^2\ |\ *_\vp(\al \w \vp)=-\al\},\\
\Lm^3_7&=\{*_\vp(\al\w \vp)\ |\ \al \in \Lm^1_7\},\\
\Lm^3_{27}&=\{\al\in \Lm^3\ |\ \al \w \vp=0\ \text{\ and \ }\al \w *_\vp\vp=0 \}.
\end{align*}
\noindent 
As before, we shall denote the projection of a $k$-form $\al$ in $\Lm^k_p$ by either $(\al)^k_p$ or $\pi^k_p(\al)$; this should not cause any confusion with the $\SU(3)$ case since the pair $k,p$ are always different. Following \cite{Bryant06someremarks} the $\rmG_2$ torsion forms $\tau_i$ are defined by
\begin{align}
    d\vp &= \tau_0 *_\vp\vp + 3 \tau_1 \w \vp + *_\vp \tau_3, \label{equ: g2 torsion 1}\\
    d*_\vp\vp &= 4 \tau_1 \w *_\vp\vp + \tau_2 \w \vp,\label{equ: g2 torsion 2}
\end{align}
where $\tau_0\in C^\infty(M)$, $\tau_1 \in \Lm^1_7$, $\tau_2 \in \Lm^2_{14}$ and $\tau_3 \in \Lm^3_{27}$. We refer the reader to \cites{Salamon1989, Fernandez1982} for more details on the intrinsic  torsion of $\rmG_2$-structures.
By analogy to the almost Hermitian case, the $1$-form $\theta:=4\tau_1$ is sometimes called the {\bf $\rmG_2$ Lee form}.

\begin{definition}
    If $\tau_2=0$ then we call the $\rmG_2$-structure defined by $\vp$ a  {\bf $\rmG_2$-structure with torsion} (`$\rmG_2T$' for short). 
\end{definition}
\begin{remark}
    $\rmG_2T$-structures are sometimes also called `integrable' $\rmG_2$-structures in the physics literature, see for instance \cites{FriedrichIvanov02, FriedrichIvanov03, Ivanov2023}. The terminology is motivated to make the analogy with \textit{integrable} almost complex structures. However, this is rather misleading since generally an integrable $\rmG$-structure is one with vanishing intrinsic torsion. Thus, we shall not employ this terminology to avoid confusion. It is also worth pointing out that the condition $\tau_2=0$ is necessary to define a Dolbeault type complex for $\rmG_2$-structures, which plays an important role in the 
    deformation theory of $\rmG_2T$ manifolds and and heterotic compactifications, see \cite{Xenia18}*{\S 2.1.3} and \cite{CARRION19981}*{\S 5.2}.
\end{remark}

It was shown in \cite{FriedrichIvanov02}*{Theorem 4.7} that $\tau_2=0$ if and only if there exists a $\rmG_2$ connection, i.e a connection preserving $\vp$, with totally skew-symmetric torsion $T_{\vp}$; moreover, this connection is unique. The torsion $3$-form is then explicitly given by
\begin{align}
    T_{\vp} &= \frac{1}{6} *_\vp (d\vp \w \vp) \vp - *_\vp d\vp + *_\vp(4\tau_1 \w \varphi) \label{equ: definition of T}\\
    &= \frac{1}{6} \tau_0 \vp + *_\vp (\tau_1 \w \varphi) - \tau_3. \label{equ: definition of T 2}
\end{align}

\begin{definition}
    If $\tau_2=0$ and $dT_{\vp}=0$ then we call the $\rmG_2$-structure defined by $\vp$ a {\bf strong $\rmG_2T$-structure}.
\end{definition}
\begin{remark}
    The terminology strong $\rmG_2T$-structure is motivated to make the analogy with SKT structure. 
\end{remark}
We shall denote this $\rmG_2$ connection by $\nabla^T$ and use $\nabla$ for the Levi-Civita connection of $g_\vp$ (this should not cause any confusion with the Levi-Civita connection of the $6$-manifold $P$). One can show that $\nabla^T$ is explicitly given by 
\begin{equation*}
    g_\vp( \nabla^T_XY - \nabla_XY,Z ) =  Z \ip Y \ip A(X)
\end{equation*}
where the tensor $A\in \Lm^1 \otimes \Lm^2$ is defined by
\begin{equation}
    A(X) := \frac{1}{12}\tau_0 (X \ip \vp) + \pi^2_7(X^\flat \w \tau_1) -\frac{1}{2}\pi^2_{14}(X^\flat \w \tau_1) 
    - \frac{1}{2} (X \ip \tau_3),
\end{equation}
see \cite{FriedrichIvanov02}*{Theorem 4.8} or more generally \cite{Agricola06}*{\S 2}.
We shall call $\nabla^T$ the {\bf characteristic $\rmG_2$ connection}. We define the associated {\bf characteristic Ricci tensor} by
\begin{equation}
    \Ric^T(X,Y) := \sum_{i=1}^7g_\vp(\mathrm{Rm}^T(e_i,X)Y,e_i),
\end{equation}
where 
$\mathrm{Rm}^T(X,Y)Z := \nabla^T_X(\nabla^T_Y Z) - \nabla^T_Y(\nabla^T_X Z) - \nabla^T_{[X,Y]} Z $ and $\{e_i\}_{i=1}^7$ denotes a local orthonormal framing.  Note that unlike the Riemannian Ricci tensor $\Ric(g_{\vp})$, the characteristic Ricci tensor $\Ric^T$ is not in symmetric in general. In Section \ref{sec: g2 curvature} we shall derive formulae expressing $\Ric^T$ in terms of the $\rmG_2$ torsion forms. This allows one to efficiently compute $\Ric^T$ in explicit examples.


\section{Curvature of strong \texorpdfstring{$\rmG_2$}{}T manifolds}\label{sec: g2 curvature}

In this section we derive expressions for certain components of the curvature tensor of (strong) $\rmG_2T$-structures in terms of the $\rmG_2$ torsion forms. While a few of these formulae have already appeared in the literature (sometimes under further assumptions), see for instance \cite{Ivanov2023},
one key difference is that we here employ the representation theoretic approach of Bryant in \cite{Bryant06someremarks} to derive these formulae; this often leads to much simpler proofs and avoids using the language of spinors.

\subsection{Ricci and Weyl curvatures of strong \texorpdfstring{$\rmG_2$}{}T manifolds}
Let $(M,\vp,g_{\vp})$ denote a $7$-manifold with a general $\rmG_2$-structure and define the $3$-form $T_{\vp}$ by
\begin{align*}
    T_{\vp} := \frac{1}{6} \tau_0 \vp + *_\vp (\tau_1 \w \varphi) - \tau_3.
\end{align*}
We emphasise that we are  not  assuming that $\tau_2=0$ at this point. In terms of the irreducible $\rmG_2$ components, we write
\begin{equation*}
    dT_{\vp} = (dT_{\vp})_0 *_\vp\vp + (dT_{\vp})_1 \w \varphi + (dT_{\vp})_4 \in \Lm^4 \cong \langle *_\vp\vp\rangle \oplus \Lm^4_7 \oplus \Lm^4_{27}.
\end{equation*}
\begin{theorem}\label{thm: dTvp}
Given an arbitrary $\rmG_2$-structure $(M,\vp,g_\vp)$, the following holds:
\begin{gather}
    (dT_{\vp})_0 = \frac{4}{7}\delta \tau_1 + \frac{1}{6}\tau_0^2+\frac{12}{7}|\tau_1|^2-\frac{1}{7}|\tau_3|^2,\label{equ: g2T 1}\\
    (dT_{\vp})_1 = -\frac{7}{12}d \tau_0 - *_\vp(d\tau_1 \w *_\vp\vp) -\frac{1}{4} *_\vp(\tau_1\w \tau_2 \w \varphi)
    -\frac{1}{4} *_\vp(*_\vp(\tau_2\w \tau_3) \w *_\vp\varphi),\label{equ: g2T 2}\\
    (dT_{\vp})_4 =  
    \pi^4_{27}\Big(
    *_\vp d \big(*_\vp (\tau_1 \w *_\vp\vp)\big)
    -d\tau_3
    +*_\vp(\tau_1 \w *_\vp(\tau_1 \w *_\vp \vp))\nonumber\\
    \qquad +\frac{1}{6}\tau_0 *_\vp\tau_3 
    - *_\vp(\tau_1 \w \tau_2)
    - \tau_1 \w \tau_3
    \Big),\label{equ: g2T 3}
\end{gather}
where $\delta:=- *_\vp d *_\vp$ denotes the codifferential.
\end{theorem}
\begin{proof}
   Since $T_\vp$ is a second order $\rmG_2$-invariant, we can express it terms of the basis elements in Theorem \ref{thm: second order invariants} and \ref{thm: first order invariants} in the appendix. Thus, the proof reduces to determining the coefficient constants and this is easily done by verification in explicit examples.
\end{proof}

Let us now specialise to the strong $\rmG_2T$ case i.e. when $\tau_2=0$ and $dT_{\vp}=0$. From Theorem \ref{thm: dTvp} we immediately recover the following recent result in \cite{Ivanov2023}*{Theorem 7.1}: 

\begin{corollary}\label{cor: tau0 is constant} 
    If $(M,\vp,g_{\vp})$ is a strong $\rmG_2T$ manifold (not necessarily compact) then $\tau_0$ is constant.
\end{corollary}
\begin{proof}
    Since $\tau_2=0$, from (\ref{equ: g2 torsion 2}) we have that 
    \[
    0=d(d*_\vp\vp) = d(4\tau_1 \w *_\vp\vp) = 4 d\tau_1 \w *_\vp\vp
    \]
    i.e. $d\tau_1\in \Lm^2_{14}$. Furthermore, from the strong condition we have $dT_{\vp}=0$ and it follows immediately from (\ref{equ: g2T 2}) that $d\tau_0=0$. 
\end{proof}
$\rmG_2$-structures with constant torsion form $\tau_0$ are sometimes called `$\rmG_2$-structures of constant type' in the literature.

\begin{corollary}
    If $(M,\vp,g_{\vp})$ is a compact strong $\rmG_2T$ manifold and $\tau_3=0$ then it is torsion free. 
\end{corollary}
\begin{proof}
    Integrate (\ref{equ: g2T 1}) and use Stokes' theorem.
\end{proof}
\begin{corollary}
    If $(M,\vp,g_{\vp})$ is a strong $\rmG_2T$ manifold (not necessarily compact) and $\tau_0=0=\tau_1$ then it is torsion free.
\end{corollary}
\begin{proof}
    This follows immediately from (\ref{equ: g2T 1}).
\end{proof}

Recall from \cite{Bryant06someremarks}*{(4.28)} that the scalar curvature of $g_\vp$ is given by
\begin{equation}
    \mathrm{Scal}(g_\vp) = 12\delta \tau_1 + \frac{21}{8}\tau_0^2+30|\tau_1|^2-\frac{1}{2}|\tau_2|^2-\frac{1}{2}|\tau_3|^2,\label{equ: scal curvature general}
\end{equation}
and hence from (\ref{equ: g2T 1}) we deduce that
the scalar curvature of a strong $\rmG_2T$ manifold is given by 
\begin{equation}
    \mathrm{Scal}(g_\vp) =  -\frac{7}{8}\tau_0^2-6|\tau_1|^2+\frac{5}{2}|\tau_3|^2
    = 10 \delta\tau_1 + \frac{49}{24}\tau_0^2 + 24 |\tau_1|^2. \label{equ: scal of strong g2t}
\end{equation}
The above result was also shown in \cite{FinoG2T2023}*{Proposition 3.1} where it was used to demonstrate the inexistence of invariant strong $\rmG_2T$-structures on compact solvmanifolds, see also \cite{FriedrichIvanov03}*{Theorem 1.1}. In particular, observe from (\ref{equ: scal of strong g2t}) that the scalar curvature is expressible in terms of first order $\rmG_2$-invariants only. Furthermore, if $M$ is compact then
\begin{equation*}
    \int_M \mathrm{Scal}(g_\vp) \vol = \int_M \Big(\frac{49}{24}\tau_0^2+24|\tau_1|^2 \Big) \vol\geq 0.
\end{equation*}
Thus, if $(M,\vp, g_{\vp})$ is a compact homogeneous Riemannian manifold then it must have positive scalar curvature. If $\vp$ is co-closed then we have the following stronger implication:
\begin{corollary}
    If $(M,\vp,g_\vp)$ is a co-closed strong $\rmG_2T$ manifold (not necessarily compact) then it has constant positive scalar curvature. In particular, if $M$ is compact but not torsion free then $\hat{\mathcal{A}}(M)=0$.
\end{corollary}
\begin{proof}
    Using that $\tau_1=0$ together with (\ref{equ: scal of strong g2t}), we get 
    \[
    \Scal(g_\vp)=\frac{49}{24}\tau_0^2.
    \]
    The result now follows from Corollary \ref{cor: tau0 is constant} and \cite{LMspin}*{\S IV Theorem 4.1}.    
\end{proof}
To the best of our knowledge, the only known examples of co-closed strong $\rmG_2T$ manifolds are $S^3 \times N^4$, where $N^4$ is a hyperK\"ahler 4-manifold, see Example \ref{example: N4 x S3} below.

It is standard fact from the representation theory of $\rmG_2$ that the space of traceless symmetric $2$-tensors $S^2_0(M)$ is isomorphic to $\Lm^4_{27}$, see for instance \cite{Bryant1987}*{Section 2}. In what follows, we shall use the following explicit identification: 
\begin{align}
    S^2(M):= \langle g_\vp \rangle \oplus S^2_0(M) &\xrightarrow{\simeq} \langle *_\vp\vp \rangle \oplus \Lm^4_{27}\nonumber\\
    \alpha &\mapsto \alpha \diamond *_\vp\vp := \sum_{i=1}^7 \alpha_{ij}e^i \w (e_j \ip *_\vp\vp) \label{equ: identifying 2-tensors with 4-forms}
\end{align}
where $\{e_i\}_{i=1}^7$ denotes a local $\rmG_2$ frame and the operator $\diamond$ corresponds to the endomorphism action, see also \cite{Dwivedi2023}*{(2.19)}. In particular, we have that $g_\vp \diamond *_\vp\vp = 4 *_\vp\vp$. It is also worth noting that the above definition for $\diamond$ extends naturally to any $2$-tensor $\alpha=\al_{ij}e^i\otimes e^j$; this yields an isomorphism of the $\rmG_2$ modules $\Lm^2_{7}\cong \Lm^4_{7}$ while $ \mathfrak{g}_2 \cong \Lm^2_{14} = \mathrm{ker}(\cdot \diamond *_\vp\vp)$ since $\rmG_2$ preserves $*_\vp\vp$. 
\begin{remark}
    We should point out that in the literature, there is also another common $\rmG_2$-equivariant map $\textsf{i}_\vp: S^2(M)\to \langle \vp \rangle \oplus \Lm^3_{27}$, see for instance \cite{Bryant06someremarks}*{(2.15)}. Up to overall constant factors, this coincides with the map $\diamond \vp$ restricted to symmetric $2$-tensors.
\end{remark}

We recall from \cite{Bryant06someremarks}*{(4.30)} that the traceless Ricci tensor of $g_\vp$, viewed as a $4$-form in $\Lm^4_{27}$, is given by
\begin{align}
    \Ric_0(g_\vp) = -4\cdot \pi^4_{27}\Big(
    &-\frac{5}{4}*_\vp d \big(*_\vp (\tau_1 \w *_\vp\vp)\big)
    -\frac{1}{4} *_\vp d\tau_2
    +\frac{1}{4}d\tau_3
    +\frac{5}{2}*_\vp(\tau_1 \w *_\vp(\tau_1 \w *_\vp \vp))\nonumber\\
    &-\frac{1}{8}\tau_0 *_\vp\tau_3 
    +\frac{1}{4} *_\vp(\tau_1 \w \tau_2)
    +\frac{3}{4} \tau_1 \w \tau_3
    +\frac{1}{8} \tau_2 \w \tau_2
    +\frac{1}{64}*_\vp \textsf{Q}(\tau_3,\tau_3)
    \Big),\label{equ: Ricci curvature general}
\end{align}
where $\textsf{Q}$ is a bilinear map defined in Theorem \ref{thm: first order invariants}. Hence we deduce from (\ref{equ: g2T 3}) that the traceless Ricci curvature of a strong $\rmG_2$T manifold is given by
    \begin{equation}
    \Ric_0(g_\vp) = -4 \cdot \pi^4_{27}\Big(
    -d\tau_3
    +\frac{15}{4}*_\vp(\tau_1 \w *_\vp(\tau_1 \w *_\vp \vp))+\frac{1}{12}\tau_0 *_\vp\tau_3 
    -\frac{1}{2} \tau_1 \w \tau_3
    +\frac{1}{64}*_\vp \textup{\textsf{Q}}(\tau_3,\tau_3)
    \Big),\label{equ: traceless ric strong g2t}
\end{equation}
or equivalently, by 
    \begin{align}
    \Ric_0(g_\vp) = -4 \cdot \pi^4_{27}\Big(
    &-*_\vp d \big(*_\vp (\tau_1 \w *_\vp\vp)\big)
    +\frac{11}{4}*_\vp(\tau_1 \w *_\vp(\tau_1 \w *_\vp \vp))-\frac{1}{12}\tau_0 *_\vp\tau_3\label{equ: traceless ric strong g2t 2} \\
    &+\frac{1}{2} \tau_1 \w \tau_3
    +\frac{1}{64}*_\vp \textup{\textsf{Q}}(\tau_3,\tau_3)
    \Big).\nonumber
\end{align}
In particular, observe that if $\tau_1=0$ i.e. $\vp$ is co-closed then $\Ric_0(g_\vp)$ is given by only first order $\rmG_2$-invariants.

Lastly we consider the Weyl curvature. Recall that the space Weyl curvature $\mathcal{W}$ of a manifold with a $\rmG_2$-structure splits as a $\rmG_2$ module into
\begin{equation*}
    \mathcal{W} \cong \Lm^4_{27} \oplus V_{1,1} \oplus V_{0,2},
\end{equation*}
see \cite{Bryant06someremarks}*{(4.8)}. 
One can show, using similar method as in \cite{Bryant06someremarks},  that the $\Lm^4_{27}$-component of the Weyl curvature defined by 
\begin{equation*}
    W_{27} := \vp_{ijk} W_{ijab} e^a \w e^b \w (e_k \ip \vp)
\end{equation*}
can be expressed as 
\begin{align}
    W_{27} = \pi^4_{27}\Big(
    &-\frac{6}{5} *_\vp d\tau_2
    -\frac{4}{5}d\tau_3
    -\frac{3}{5}\tau_0 *_\vp\tau_3 
    +\frac{6}{5} *_\vp(\tau_1 \w \tau_2)
    +\frac{8}{5} \tau_1 \w \tau_3\nonumber\\
    &-\frac{3}{20} \tau_2 \w \tau_2
    +\frac{1}{80}*_\vp \textsf{Q}(\tau_3,\tau_3)
    -\frac{1}{2} \widehat{\textsf{Q}}(\tau_2,\tau_3)
    -\frac{1}{4} *_\vp\widehat{\textsf{Q}}(\tau_3,\tau_3)
    \Big)\label{equ: Weyl27 curvature general},
\end{align}
where $\widehat{\textsf{Q}}$ is a bilinear map defined in Theorem \ref{thm: first order invariants}.

\begin{remark}
Note that $S^2(V_1(\mathfrak{g}_2))$ has exactly $8$ copies of $\Lm^4_{27}$, see (\ref{equ: sym2 v1g2}); the generators of $6$ of them appear in the formula of Ricci curvature while the remaining two, corresponding to $\widehat{\textsf{Q}}(\tau_2,\tau_3)$ and  $*_\vp\widehat{\textsf{Q}}(\tau_3,\tau_3)$, only make an appearance in the formula for $W_{27}$. This shows that all the $8$ terms occur in the curvature formulae, compare with \cite{Bryant06someremarks}*{Remark 10}. 
\end{remark}

\begin{corollary}
    The $\Lm^4_{27}$-component of the Weyl curvature of a strong $\rmG_2T$ manifold is given by
    \begin{align}
    W_{27} = \pi^4_{27}\Big(
    -\frac{4}{5}d\tau_3
    -\frac{3}{5}\tau_0 *_\vp\tau_3 
    +\frac{8}{5} \tau_1 \w \tau_3
    +\frac{1}{80}*_\vp \textup{\textsf{Q}}(\tau_3,\tau_3)
    -\frac{1}{4} *_\vp\widehat{\textup{\textsf{Q}}}(\tau_3,\tau_3)
    \Big)\label{equ: Weyl27 curvature strong G2T},
\end{align}
or equivalently, by
\begin{align}
    W_{27} = \pi^4_{27}\Big(
    &-\frac{4}{5}*_\vp d \big(*_\vp (\tau_1 \w *_\vp\vp)\big)
    -\frac{4}{5}*_\vp(\tau_1 \w *_\vp(\tau_1 \w *_\vp \vp))    
    -\frac{11}{15}\tau_0 *_\vp\tau_3\nonumber\\ 
    &+\frac{12}{5} \tau_1 \w \tau_3
    +\frac{1}{80}*_\vp \textup{\textsf{Q}}(\tau_3,\tau_3)
    -\frac{1}{4} *_\vp\widehat{\textup{\textsf{Q}}}(\tau_3,\tau_3)
    \Big)\label{equ: Weyl27 curvature strong G2T 2}.
\end{align}
In particular, if $\tau_3=0$ then $W_{27}=0$. Moreover, similarly as for $\Ric_0(g_{\vp})$, if $\tau_1=0$, then $W_{27}$ is given by only first order $\rmG_2$-invariants.
\end{corollary}
Next we derive analogous formulae for the characteristic curvature tensor $\mathrm{Rm}^T$.

\subsection{Characteristic Ricci curvature of  \texorpdfstring{$\rmG_2$}{}T manifolds}

Let $(M,g)$ denote a general Riemannian manifold with a metric connection $\nabla^H$ and whose associated torsion tensor $H$ is a $3$-form. It is well-known that if $dH=0$ then the associated Ricci tensor $\Ric(\nabla^H)$ is given by
\begin{equation}
\Ric(\nabla^H) = \Ric(g) - \frac{1}{4} H^2 + \frac{1}{2}\delta H,\label{equ: ricciH}
\end{equation}
where $H^2:=(e_j \ip e_i \ip H) \otimes (e_j \ip e_i \ip H)$, see for instance \cite{FernandezStreetsBook}*{Proposition 3.18} and \cite{FriedrichIvanov02}. In particular, this applies to strong $\rmG_2T$-structures. By deriving explicit expressions for all the quantities involved in terms of the $\rmG_2$ torsion forms, we show below that the strong condition, $dT_\vp=0$, is in fact not necessary i.e. $\rmG_2 T$ is sufficient,  see Proposition \ref{prop: traceless ricT}, \ref{prop: scalT} and \ref{prop: ricci of bismut skew1}. We also show that part of the characteristic Ricci curvature is in fact determined by $\mathcal{L}_{\tau_1^\sharp}\vp$ (see Theorem \ref{thm: ricciT and Lie derivative}). This leads to a new characterisation of the characteristic Ricci flat condition (see Theorem \ref{cor: main corollary}).
We should emphasise that our result here is purely local i.e. we do not require compactness or completeness.

If we denote the traceless symmetric part of characteristic Ricci tensor by $S^2_0(\Ric^T)$ and identify it with a $4$-form in $\Lm^4_{27}$ via (\ref{equ: identifying 2-tensors with 4-forms}) as before, then we have the following:
\begin{proposition}\label{prop: traceless ricT}
Given a $\rmG_2T$ manifold $(M,\vp,g_{\vp})$ (not necessarily strong), we have 
    \begin{align}
    S^2_0(\Ric^T) = -4 \cdot \pi^4_{27}\Big(
    &-\frac{5}{4}*_\vp d \big(*_\vp (\tau_1 \w *_\vp\vp)\big)
    +\frac{1}{4}d\tau_3
    +\frac{11}{4}*_\vp(\tau_1 \w *_\vp(\tau_1 \w *_\vp \vp))\nonumber\\
    &-\frac{1}{24}\tau_0 *_\vp\tau_3 
    +\frac{5}{4} \tau_1 \w \tau_3\Big).\label{equ: Ricci curvature Skew connection}
\end{align}
In particular,
    \begin{equation*}
    \Ric_0(g_\vp) - S^2_0(\Ric^T) = -4 \cdot \pi^4_{27}\Big(
    -\frac{1}{4}*_\vp(\tau_1 \w *_\vp(\tau_1 \w *_\vp \vp))
    -\frac{1}{12}\tau_0 *_\vp\tau_3 
    -\frac{1}{2} \tau_1 \w \tau_3
    +\frac{1}{64}*_\vp \textup{\textsf{Q}}(\tau_3,\tau_3)\Big)
\end{equation*}
i.e. these two components of Ricci curvatures only differ by first order $\rmG_2$-invariants, namely, the traceless component of $T_\vp^2$.
\end{proposition}
\begin{proof}
    Use the same strategy as in the previous section, namely, we know that $S^2_0(\Ric^T)$ has to be a linear combination of the $S^2_0(\Lm^1_7) \cong \Lm^3_{27}\cong \Lm^4_{27}$ terms in Theorem \ref{thm: second order invariants} and \ref{thm: first order invariants}. The proof is just a matter of determining the constant coefficients. The second part follows from (\ref{equ: Ricci curvature general}).
\end{proof}
Taking the trace of $\Ric^T$ we can define the scalar curvature $\Scal^{T}$. Similarly as above, one finds that:
\begin{proposition}\label{prop: scalT}
Given a $\rmG_2T$ manifold $(M,\vp,g_{\vp})$ (not necessarily strong), we have 
    \begin{equation}
    \Scal^T = 
    12\delta \tau_1 
    + \frac{7}{3}\tau_0^2+24|\tau_1|^2
    -2|\tau_3|^2.\label{equ: scal curvature Skew connection}
\end{equation}
In particular, 
    \begin{equation*}
    \Scal(g_\vp) - \Scal^T = 
    \frac{7}{24}\tau_0^2+6|\tau_1|^2
    +\frac{3}{2}|\tau_3|^2 \geq 0.
\end{equation*}
\end{proposition}
\begin{proof}
    We know that $\Scal^T$ has to be a linear combination of the terms in the trivial module of Theorem \ref{thm: second order invariants} and \ref{thm: first order invariants}. The proof is just a matter of determining the constant coefficients. The second part follows from (\ref{equ: scal curvature general}).
\end{proof}
Next we compute the $\Lm^2_7$- and $\Lm^2_{14}$-components of $\Ric^T$ and $\delta T_\vp$: 
\begin{proposition}\label{prop: ricci of bismut skew1} \label{prop: ricci of bismut skew2} 
    Given a $\rmG_2T$ manifold $(M,\vp,g_{\vp})$ (not necessarily strong), we have 
    \begin{gather*}
    2(\Ric^T)^2_7 =(\delta T_\vp)^2_7 = \pi^2_7\Big(
    -\frac{7}{6}*_\vp(d\tau_0 \w *_\vp\vp)
    -\frac{2}{3}\tau_0 *_\vp(\tau_1 \w *_\vp\vp) + 4*_\vp(\tau_1 \w *_\vp\tau_3) \Big),\\
    2(\Ric^T)^2_{14}=(\delta T_\vp)^2_{14} = 
    \pi^2_{14}\Big( 
    4d\tau_1
    +4 *_\vp(\tau_1 \w *_\vp\tau_3)
    \Big).
\end{gather*}
In particular, $\delta T_\vp=2\mathrm{Skew}(\Ric^T)$.
\end{proposition}
\begin{proof}
    To derive the general expression for $(\delta T_\vp)^2_7$ and $(\Ric^T)^2_7$ we apply the same method as before, namely we know that these have to be a linear combinations of the $\Lm^1_7\cong \Lm^2_7 \cong \Lm^4_7$ terms in Theorem \ref{thm: second order invariants} and \ref{thm: first order invariants}. It is then just a matter of determining the constant coefficients, 
    A similar argument applies for $(\Ric^T)^2_{14}$ and $(\delta T_\vp)^2_{14}$ by considering the $\Lm^2_{14}$ terms in Theorem \ref{thm: second order invariants} and \ref{thm: first order invariants}.
\end{proof}

An immediate consequence of the latter and Corollary \ref{cor: tau0 is constant} is:
\begin{corollary}\label{prop: codifferential of T}\label{thm: T is harmonic and RicciT}
    For a $\rmG_2T$-structure with $\tau_0$ constant, we have
    \begin{equation*}
    \mathrm{Skew}(\Ric^T) = \frac{1}{2}\delta T_\vp = 2 d\tau_1+ 2*_\vp(\tau_1 \w *_\vp\tau_3) - \frac{1}{3} \tau_0 *_\vp(\tau_1 \w *_\vp\vp).
    \end{equation*}
    In particular, if $\vp$ is a co-closed strong $\rmG_2T$-structure then $T_\vp$ is closed and co-closed\footnote{Note here that we avoid saying that $T_\vp$ is harmonic since harmonic is not equivalent to closed and co-closed if the underlying manifold is not compact.}.
\end{corollary}

A similar computation as above gives:
\begin{theorem}\label{thm: ricciT and Lie derivative}
    For a strong $\rmG_2T$-structure, we have
    \begin{gather*}
        \Scal^T *_\vp\vp= -7 (\mathcal{L}_{\tau_1^\sharp}(*_\vp\vp))^4_1 = 4 \delta\tau_1 *_\vp\vp,\\
        S^2_0(\Ric^T) = -4 (\mathcal{L}_{\tau_1^\sharp}(*_\vp\vp))^4_{27},\\
        (\Ric^T)^4_7 = -4 (\mathcal{L}_{\tau_1^\sharp}(*_\vp\vp))^4_{7}.
    \end{gather*}
    In particular, the vector field $\tau_1^\sharp$ is Killing if and only if $\Ric^T \in \Lm^2$. Furthermore, the vector field $\tau_1^\sharp$ preserves $\vp$ if and only if $\Ric^T \in \Lm^2_{14}$. 
\end{theorem}
\begin{proof}    
    Using the same representation theoretic method as before one can show that for a general $\rmG_2$-structure the following hold:
    \begin{gather}(\mathcal{L}_{\tau_1^\sharp}(*_\vp\vp))^4_1 = -\frac{4}{7}\delta\tau_1 *_\vp\vp,\label{equ: lie derivative 1}\\(\mathcal{L}_{\tau_1^\sharp}(*_\vp\vp))^4_{27} = \pi^4_{27}\Big( 
    -*_\vp d*_\vp(\tau_1 \w *_\vp \vp) + \tau_1 \w \tau_3 +
    3*_\vp(\tau_1 \w *_\vp(\tau_1 \w *_\vp\vp))
    \Big),\label{equ: lie derivative 2}
    \\(\mathcal{L}_{\tau_1^\sharp}(*_\vp\vp))^4_{7} = \pi^4_{7}\Big( 
    -\frac{1}{2}*_\vp(d\tau_1 \w *_\vp \vp)\w \vp -\frac{1}{4} \tau_0 \tau_1 \w \vp -
    2\tau_1 \w \tau_3 + 
    \frac{1}{2}*_\vp(\tau_1 \w \tau_2 \w \vp) \w \vp
    \Big).\label{equ: lie derivative 3}
    \end{gather}
    The result now follows by comparing the latter expressions with the above formulae for $\Scal^T$ and $S^2_0(\Ric^T)$, and using the strong condition i.e. the vanishing of (\ref{equ: g2T 1})-(\ref{equ: g2T 3}). For instance, from (\ref{equ: g2T 1}) and (\ref{equ: scal curvature Skew connection}) we get
    \[
    \mathrm{Scal}^T = 4 \delta \tau_1 + 14 (dT)_0.
    \]
    Likewise, for the $\Lm^4_{27}$-component we have $S^2_0(\Ric^T)=(dT_\vp)_4-4(\mathcal{L}_{\tau_1^\sharp}(*_\vp\vp))^4_{27}$, and a 
    similar computation applies for the remaining $\rmG_2$ component in $\Lm^2_7\cong \Lm^4_7$.
\end{proof}
Note that in Theorem \ref{thm: ricciT and Lie derivative}  we do not need compactness or completeness, the result  is a local statement. We should also point out unlike in the previous derivations whereby we only needed that $\tau_2=0$, in the proof of Theorem \ref{thm: ricciT and Lie derivative} we actually do require the strong condition $dT_\vp=0$.

\begin{proposition}\label{prop: nablaT tau1}
For a $\rmG_2T$-structure (not necessarily strong), the following holds:
\begin{gather*}
    \big((\nabla^T \tau_1) \diamond *_\vp \vp\big)^4_1 = -\frac{4}{7}(\delta \tau_1) *_\vp\vp,\\
    \big((\nabla^T \tau_1) \diamond *_\vp \vp\big)^4_{27}=\pi^4_{27}\Big(\mathcal{L}_{\tau_1^\sharp}(*_\vp\vp)\Big),\\
    (\nabla^T \tau_1)^2_{7}=\pi^2_7\Big(
    \frac{1}{2}*_\vp(\tau_1 \w *_\vp\tau_3)-\frac{1}{12}\tau_0 *_\vp(\tau_1\w *_\vp\vp)
    \Big)= \pi^2_7\Big(\frac{1}{8}\delta T_\vp +\frac{7}{48}*_\vp(d\tau_0 \w *_\vp\vp) \Big),\\
    (\nabla^T \tau_1)^2_{14}=\pi^2_{14}\Big(
    \frac{1}{2} d\tau_1
    +\frac{1}{2} *_\vp(\tau_1\w *_\vp\tau_3)
    \Big)= \frac{1}{8}\pi^2_{14}\Big(\delta T_\vp\Big).
\end{gather*}
\end{proposition}
\begin{proof}
    Since $\nabla^T \tau_1$ is a second order $\rmG_2$-invariant and vanishes when $\tau_1=0$, we know that it can only involve terms containing $\tau_1$ in Theorem \ref{thm: second order invariants} and \ref{thm: first order invariants}. So it suffices to check which terms occur by verifying in examples. The expressions involving $\delta T_\vp$ and $\mathcal{L}_{\tau_1^\sharp}(*_\vp\vp)$ follow from Proposition \ref{prop: ricci of bismut skew1} and (\ref{equ: lie derivative 1})-(\ref{equ: lie derivative 3}).
\end{proof}

\begin{theorem}\label{thm: main without strong}
    Let $(M,\vp,g_\vp)$ denote a $\rmG_2T$-manifold (not necessarily strong) with $\tau_0$ constant and with Lee $1$-form $\theta:=4\tau_1$. Then the following are equivalent:
    \begin{enumerate}
        \item $T_\vp$ is co-closed and the vector field $\theta^\sharp$ preserves $\vp$.
        \item  $\nabla^T \theta=0$.
    \end{enumerate}
    Moreover, if we drop the hypothesis that $\tau_0$ is constant, then 2. still implies $\theta^\sharp$ is a Killing vector field.
\end{theorem}
\begin{proof}
    The result is immediate in view of Proposition \ref{prop: nablaT tau1}, together with expressions (\ref{equ: lie derivative 1})-(\ref{equ: lie derivative 3}). To see the vanishing of $(\mathcal{L}_{\tau_1^\sharp}(*_\vp\vp))^4_{7}$ it suffices to rewrite (\ref{equ: lie derivative 3}) as
    \[
    *_\vp(*_\vp(*_\vp(\mathcal{L}_{\tau_1^\sharp}(*_\vp\vp))^4_{7} \w \vp) \w *_\vp\vp) = \pi^2_7\Big( \tau_0 *_\vp(\tau_1 \w *_\vp\vp) - 6 *_\vp(\tau_1 \w *_\vp \tau_3) \Big)= \pi^2_7\Big( -\frac{3}{2}\delta T_\vp - \frac{7}{4}*_\vp(d\tau_0 \w *_\vp\vp)) \Big).
    \]
    Moreover, from the expressions in Proposition \ref{prop: ricci of bismut skew1} and \ref{prop: nablaT tau1}, we see that the condition $\tau_0$ is constant is needed precisely for the equivalence of the terms $(\mathcal{L}_{\tau_1}^\sharp*_\vp\vp)^4_7$, $(\nabla^T\tau_1)^2_7$ and $(\delta T_\vp)^2_7$.
\end{proof}

Combining the latter result and Theorem \ref{thm: ricciT and Lie derivative} we have:
\begin{theorem}\label{cor: main corollary}
    Let $(M,\vp,g_\vp)$ denote a strong $\rmG_2T$ manifold with Lee $1$-form $\theta:=4\tau_1$. Then the following are equivalent:
    \begin{itemize}
        \item[(a)] $\Ric^T$ vanishes.
        \item[(b)] $T_\vp$ is co-closed and the vector field $\theta^\sharp$ preserves $\vp$.
        \item[(c)]  $\nabla^T \theta=0$. 
    \end{itemize} 
    In particular, if $(M,\vp,g_\vp)$ is a compact strong $\rmG_2T$ manifold and $\Ric^T$ vanishes then $b_3(M)>0$. \end{theorem}
Note again that the first part of Theorem  \ref{cor: main corollary} is entirely local i.e. we do not need compactness or completeness. All known compact examples of strong $\rmG_2T$ manifolds have $\Ric^T = 0$. Below however, we exhibit a local example whereby $T_\vp$ is closed and co-closed but $\Ric^T$ does not vanish (i.e. $\theta^\sharp$ is not Killing). This shows that Proposition \ref{thm: T is harmonic and RicciT} is optimal.

From Theorem \ref{cor: main corollary} we have the following consequences of strong $\rmG_2T$ and $\Ric^T=0$:
\begin{enumerate}[label=(\roman*)]
    \item \label{strong ricci flat constant tau1}
    Since $\nabla^T$ is a metric connection and $\nabla^T\tau_1=0$, it follows that $\tau_1$ has constant norm. Thus, with loss of generality one can always assume that either $\tau_1=0$ or $|\tau_1|=1$; both situations do indeed occur, see \S \ref{sec: explicit example}.
    \item \label{strong ricci flat constant T} Since $\tau_1$ has constant norm and defines a Killing vector field, from equation (\ref{equ: g2T 1}) we get
    \begin{equation*}
    |\tau_3|^2=\frac{7}{6}\tau_0^2+12|\tau_1|^2.
    \end{equation*}
    In particular, $T_\vp$ has constant norm. 
    If $\tau_0=\tau_1=0$ then the $\rmG_2$-structure is torsion free. In \S \ref{sec: explicit example} we show that there are examples with 
    $\tau_0=0$ and $\tau_1\neq0$, and also examples with  $\tau_0\neq0$ and $\tau_1=0$. To the best of our knowledge, there are no known examples with both $\tau_0$ and $\tau_1$ non-zero. 
    \item \label{strong ricci flat g2t scal} From (\ref{equ: scal of strong g2t}) we deduce
    \begin{equation*}
    \Scal(g_\vp)=\frac{49}{24}\tau_0^2+24|\tau_1|^2\geq 0
    \end{equation*}
    and hence $\Scal(g_\vp)=0$ if and only if $T_\vp=0$. In particular, if $M$ is compact but not torsion free then it has constant positive scalar curvature and thus, $\hat{\mathcal{A}}(M)=0$ cf. \cite{LMspin}*{\S IV Theorem 4.1}.
    \item If $\tau_1 \neq 0$ then it can be viewed as a connection $1$-form for the action generated by the vector field $\tau_1^\sharp$. If this action is locally free then the quotient space inherits a natural $\SU(3)$-structure  (this is investigated in detail in Section \ref{sec: s1 reduction of ricci flat}). 
\end{enumerate}
Note that the equivalence of $(a)$ and $(c)$ was also recently shown in \cite{Ivanov2023}*{Theorem 7.1}; however, our proof here is by direct computation and in particular, we do not require any compactness assumption or any relation with generalised or spin geometry. Furthermore, observe that that from Theorem \ref{thm: main without strong} we see that (i)-(iv) in fact hold under the weaker assumption that $(M,\vp,g_\vp)$ is a $\rmG_2T$ manifold with $\nabla^T \tau_1=(dT_\vp)_0=\delta\tau_1=0$ and $\tau_0$ is constant. In order to investigate (iv) we next gather some results for $\SU(3)$-structures akin to what we did in this section for $\rmG_2$-structures.


\section{The strong condition for \texorpdfstring{$\SU(3)$}{}-manifolds with skew-symmetric \texorpdfstring{$\hat{N}$}{}} \label{sec: su3 curvature}

In this section we consider $6$-manifolds endowed with $\SU(3)$-structures with skew-symmetric Nijenhuis tensor. We derive expressions for the torsion form $T_\om$ and its exterior derivative. We also derive expressions for the associated Bismut Ricci form in terms of the $\SU(3)$ torsion forms. The computations here are adapted from the method used in \cite{Bedulli2007}, which is itself based on the representation theoretic ideas in \cite{Bryant06someremarks}.  
The reader might find it insightful to compare the results in this section with the previous one.

\begin{proposition}\label{prop: expression for NJ}
    Let $(P,\om,\Om)$ denote a $6$-manifold endowed with a general $\SU(3)$-structure, then the Nijenhuis tensor can be expressed as 
    \begin{equation}
        \hat{N} = -2 \pi_0 \Om^+ - 2\sigma_0 \Om^- +\widetilde{\pi}_2 
        -\widetilde{\sigma}_2,\label{equ: Nijenhuis formula}
    \end{equation}
where 
\[
\widetilde{\pi}_2 :=\sum_{i=1}^6 (e_i \ip \Om^-) \otimes (e_i \ip \pi_2) \quad 
{\mbox {and}} \quad 
\widetilde{\sigma}_2 :=\sum_{i=1}^6 (e_i \ip \Om^+) \otimes (e_i \ip \sigma_2).
\]
In particular, $\hat{N}$ is a 3-form if and only if $\pi_2=\sigma_2=0$.
\end{proposition}
\begin{proof}
    An expression of the form (\ref{equ: Nijenhuis formula}) was to be expected in view of \cite{ChiossiSalamonIntrinsicTorsion}*{Theorem 1.1}, and indeed this follows from a direct computation. The last statement follows from the fact that $\pi_2,\sigma_2 \in \Lm^2_8 \cong \mathfrak{su}(3)$ while $\Lm^3$ has no $\SU(3)$-module isomorphic to $\mathfrak{su}(3)$, therefore $\hat{N}$ is a 3-form if and only if $\widetilde{\pi}_2=\widetilde{\sigma}_2$. Since $X\ip \Om^+ = JX \ip \Om^-$, it follows that $\widetilde{\pi}_2=\widetilde{\sigma}_2$ if and only if $Je_i \ip \pi_2 = e_i \ip \sigma_2$ for each $i$, but then $e_i \ip Je_i \ip \pi_2 =0$ and $\pi_2\in \Lm^2_8 $ so we must have that $\pi_2=0$ and likewise for $\sigma_2$. 
\end{proof}

\begin{proposition}\label{prop: ddcom general}
    Let $(P,\om,\Om)$ denote a $6$-manifold endowed with a general $\SU(3)$-structure. If we write
    \begin{equation*}
        dd^c\om = (dd^c\om)_0 \om^2 + (dd^c\om)^4_6 + (dd^c\om)^4_8 \in \Lm^4 = \langle \om^2 \rangle \oplus \Lm^4_6 \oplus \Lm^4_8
    \end{equation*}
    then we have that
    \begin{align*}
        (dd^c\om)_0 =\ &\frac{1}{3}\delta \nu_1 +\frac{1}{3}|\nu_1|^2-\frac{1}{6}|\nu_3|^2-\frac{3}{2}\pi_0^2 -\frac{3}{2}\sigma_0^2,\\
        (dd^c\om)^4_6 =\ &\Big( 
        -3J(d\sigma_0) -3 d\pi_0 + 2 \pi_0 \nu_1
        -3 \pi_0 \pi_1 + 2 \sigma_0 J\nu_1 \\ &\ \ -3 \sigma_0 J\pi_1 - \frac{1}{2}J*_\om(\pi_2 \w \nu_3) +\frac{1}{2}*_\om(\sigma_2 \w \nu_3)\nonumber\\ 
        &\ \ +\frac{1}{2}*_\om(\pi_2 \w J\nu_1 \w \om)
        -\frac{1}{2}*_\om(\sigma_2 \w \nu_1 \w \om)
        \Big) \w \Om^+,\nonumber\\
        (dd^c\om)^4_8 =\ &\pi^2_8\Big( \delta \nu_3 - \delta(\nu_1 \w \om) +
        \frac{3}{2} \pi_0 \pi_2 +
        \frac{3}{2} \sigma_0 \sigma_2
        \Big) \w \om.
    \end{align*}
    In particular, if $N_J=0$ then $(dd^c\om)^4_6=0$ (as expected since $dd^c =\pm \partial \bar{\partial}$).
\end{proposition}
\begin{proof}
    Apply the same strategy as in the $\rmG_2$ case. The relevant first and second order $\SU(3)$-invariants were already computed in \cite{Bedulli2007}*{Section 3} so it is just a matter of determining the coefficient constants. 
\end{proof}
\begin{remark}[{Differential relations}]
    By taking the exterior derivative of (\ref{equ: torsion su3 1})-(\ref{equ: torsion su3 3}), we obtain the following identities relating the certain derivatives of the torsion forms:
    \begin{align}
        0=\ &d\nu_3 + d\nu_1 \w \om + \frac{3}{2}(-d\sigma_0 + Jd\pi_0 -\sigma_0 \pi_1 + \pi_0 J\pi_1 + \sigma_0 \nu_1 - \pi_0 J\nu_1)\w \Om^+\label{equ: Nijenhuis identity 0}\\ 
        &+\frac{3}{2}(\sigma_0 \pi_2 - \pi_0 \sigma_2)\w \om - \nu_1 \w \nu_3, \nonumber \\
        0=\ &\delta \pi_2 + 2J(d\pi_0) +4 \pi_0 J\nu_1 - 2\pi_0 J\pi_1 - *(\pi_1 \w \pi_2 \w \om) - *(d\pi_1 \w \Om^+),\\
        0=\ &\delta \sigma_2 + J(\delta \pi_2) + 2J(d\sigma_0)-2 d\pi_0 - 4 \pi_0 \nu_1 + 2 \pi_0 \pi_1 + 4 \sigma_0 J\nu_1 \label{equ: Nijenhuis identity}\\ 
        &-2 \sigma_0 J(\pi_1)- *(\pi_2 \w J\pi_1\w \om) - *(\sigma_2 \w \pi_1 \w \om).\nonumber
    \end{align}
    In particular, we note that when $N_J=0$ then the latter identity is trivial. These identities can be used to express certain quantities in different ways as we shall illustrate below.
\end{remark}

Henceforth we shall assume that $\pi_2=\sigma_2=0$, equivalently $\hat{N}$ is a 3-form (in view of \cite{FriedrichIvanov02}*{Theorem 10.1}), so that
\begin{align*}
    d\om &= -\frac{3}{2}\sigma_0 \Om^+ + \frac{3}{2}\pi_0 \Om^- + \nu_1 \w \om + \nu_3, \\
    d\Om^+ &= \pi_0 \om^2 + \pi_1 \w \Om^+ , \\
    d\Om^- &= \sigma_0 \om^2 + J\pi_1 \w \Om^+ .
\end{align*}
Then we have:
\begin{proposition}\label{prop: Nhat}
    Suppose that $\pi_2=\sigma_2=0$ i.e. $\hat{N}$ is a 3-form then
    \begin{equation*}
        d\hat{N} = -2 (\pi_0^2+\sigma_0^2)\om^2-2(d\pi_0+J(d\sigma_0)+\pi_0\pi_1+\sigma_0J\pi_1)\w \Om^+.
    \end{equation*}
    In particular, $\hat{N}$ is closed if and only if $\hat{N}=0$. 
\end{proposition}
\begin{proof}
This follows by direct computation using Proposition \ref{prop: expression for NJ}.
\end{proof}
Combining Proposition \ref{prop: ddcom general} and \ref{prop: Nhat} we get: 
\begin{theorem}\label{thm: dTom}
Let $(P,\om,\Om)$ denote a $6$-manifold endowed with an $\SU(3)$-structure with $\hat{N}$ totally skew-symmetric then
    \begin{align*}
        dT_{\om} =\ &\Big(\frac{1}{3}\delta \nu_1 +\frac{1}{3}|\nu_1|^2-\frac{1}{6}|\nu_3|^2+\frac{1}{2}\pi_0^2 +\frac{1}{2}\sigma_0^2\Big) \om^2 + \\
         &\Big( 
        -J(d\sigma_0) - d\pi_0 + 2 \pi_0 \nu_1
        - \pi_0 \pi_1 + 2 \sigma_0 J\nu_1  - \sigma_0 J\pi_1 
        \Big) \w \Om^+ \nonumber\\
        &+\pi^2_8\Big( \delta \nu_3 - \delta(\nu_1 \w \om) 
        \Big) \w \om,
    \end{align*}
    or equivalently, using (\ref{equ: Nijenhuis identity}),  
    \begin{align*}
        dT_{\om} =\ &\Big(\frac{1}{3}\delta \nu_1 +\frac{1}{3}|\nu_1|^2-\frac{1}{6}|\nu_3|^2+\frac{1}{2}\pi_0^2 +\frac{1}{2}\sigma_0^2\Big) \om^2 + \\
         &\Big( 
        -2 d\pi_0 + 4 \sigma_0 J\nu_1  - 2\sigma_0 J\pi_1 
        \Big) \w \Om^+ \nonumber\\
        &+\pi^2_8\Big( \delta \nu_3 - \delta(\nu_1 \w \om) 
        \Big) \w \om.
    \end{align*}
In particular, 
\begin{enumerate}
    \item if $(P,\om,\Om)$ is a compact almost SKT manifold and $\nu_3=0$ then it is K\"ahler.
    \item if $(P,\om,\Om)$ is a nearly K\"ahler manifold then $T_{\om}=-\Om^-$ and $dT_{\om}= 2 \om^2$.
\end{enumerate}
\end{theorem}
As a consequence:
\begin{corollary}\label{cor: scal curvature almost skt}
    The scalar curvature of an almost SKT $6$-manifold  $(P,\om,\Om)$ is given by
    \begin{equation}
    \mathrm{Scal}(g_{\om})= \frac{9}{2}\pi_0^2 +\frac{9}{2}\sigma_0^2 +2 \delta \pi_1 - 3|\nu_1|^2 + \frac{1}{2}|\nu_3|^2 + 4 g_{\om}(\nu_1,\pi_1),\label{equ: scal aSKT}
\end{equation}
and the traceless $J$-invariant component of the Ricci tensor is given by
\begin{equation*}
        (\Ric_0(g_{\om}))^2_8 = \pi^2_8\Big(  \delta(\nu_1 \w \om) + d(J\pi_1) - *_\om(\nu_1 \w J\nu_3) \Big),
    \end{equation*}
    or equivalently, by 
    \begin{equation*}
        (\Ric_0(g_{\om}))^2_8 = \pi^2_8\Big( \delta \nu_3 + d(J\pi_1) - *_\om(\nu_1 \w J\nu_3) \Big).
    \end{equation*}
In particular, if $P$ is compact and $\nu_1=0$ then $\int_P \mathrm{Scal}(g_{\om}) \vol \geq 0$. Observe also that the Nijenhuis tensor does not occur in $(\Ric_0(g_{\om}))^2_8$.
\end{corollary}
\begin{proof}
The result follows from expression \cite{Bedulli2007}*{(3.11)} and \cite{Bedulli2007}*{Theorem 3.6}, and using the condition $dT_{\om}=0$ from Theorem \ref{thm: dTom}.    
\end{proof}
The latter result is an analogue of  (\ref{equ: scal of strong g2t}) and (\ref{equ: traceless ric strong g2t}).
Next we consider the Bismut Ricci form. 

\begin{theorem}\label{thm: ricci form of bismut}
    Let $(P,\om,\Om)$ denote a $6$-manifold endowed with an $\SU(3)$-structure with $\hat{N}$ totally skew-symmetric. If we write
    \begin{equation*}
        \rho^B = (\rho^B)_0 \om + (\rho^B)^2_6 + (\rho^B)^2_8 \in \Lm^2 = \langle \om \rangle \oplus \Lm^2_6 \oplus \Lm^2_8
    \end{equation*}
    then we have
    \begin{gather*}
        (\rho^B)_0 = -\frac{2}{3} \delta \nu_1 + \frac{1}{3} \delta \pi_1 - \frac{4}{3}|\nu_1|^2 + \frac{2}{3} g_{\om}(\pi_1, \nu_1),\\
        (\rho^B)^2_6 \w \Om^+ = \Big( -*_\om(d\nu_1 \w \Om^+)+J(d\pi_0)+2\sigma_0 \nu_1-\sigma_0 \pi_1 \Big)\w \om^2,\\
        (\rho^B)^2_8 = \pi^2_8\Big( 2\delta(\nu_1 \w \om) + d(J\pi_1) -2 *_\om(\nu_1 \w J\nu_3)+2 \nu_1 \w J\nu_1 \Big).
    \end{gather*}
    In particular, when $\hat{N}=0$ we have
    \begin{equation*}
        \rho^B = \rho^C - d\delta \om,
    \end{equation*}
    where $\rho^C:=-\frac{1}{2}d\circ J \circ d(\log(|dz_1 \w dz_2 \w dz_3|^2))$ denotes the Chern-Ricci curvature and $\{z_i\}$ are local complex coordinates.
\end{theorem}
\begin{proof}
    Use the same argument as before.
\end{proof}

\begin{remark}
    When the manifold $P$ is Hermitian i.e. $N_J=0$, we can also define the Chern connection $\nabla^C$: this is the unique Hermitian connection whose torsion tensor is of type $(2,0)+(0,2)$.
    Thus, we can define the Chern-Ricci form $\rho^C$ as we did for the Bismut connection, and by analogy to the above theorem we have:
    $$
        (\rho^C)_0 =  \frac{1}{3} \delta \pi_1 + \frac{2}{3} g_{\sigma}(\pi_1, \nu_1),\quad 
        (\rho^C)^2_6  = 0, \quad
        (\rho^C)^2_8 = \pi^2_8\Big(  d(J\pi_1)  \Big),
        $$
which is of course equivalent to 
\begin{equation*}
    \rho^C=d(J \pi_1).
\end{equation*}
\end{remark}
An important consequence of the above is the following:
\begin{corollary}\label{cor: dnu1 is type 1,1}
Let $(P,\om,\Om)$ denote a $6$-manifold endowed with an $\SU(3)$-structure with $\hat{N}$ totally skew-symmetric. 
If $(\rho^B)^2_6=(dT_\om)^2_6=0$ then $d\nu_1 \in \Lm^{1,1} = \langle \om \rangle \oplus \Lm^2_8$. In particular, this applies to almost SKT and almost CYT manifolds.
\end{corollary}
\begin{proof}
    The result follows from Theorem \ref{thm: dTom}, \ref{thm: ricci form of bismut} and (\ref{equ: Nijenhuis identity}). The details are as follows. Since $\pi_2=\sigma_2=0$ from (\ref{equ: Nijenhuis identity}) we get 
    \[
    0= Jd\sigma_0 - d\pi_0 - 2\pi_0 \nu_1 + \pi_o \pi_1 + 2\sigma_0 J\nu_1 - \sigma_0 J\pi_1.
    \]
    Using that $(dT_\om)^2_6=0$ from Theorem \ref{thm: dTom}, we can rewrite the latter as
    \[
    0= - 2d\pi_0  + 4\sigma_0 J\nu_1 - 2\sigma_0 J\pi_1.
    \]
    The result now follows from the expression of $(\rho^B)^2_6$ from Theorem \ref{thm: ricci form of bismut}.
\end{proof}
One can view the above result as an analogue to the fact that for $\rmG_2T$-structures we have $d\tau_1 \in \Lm^2_{14}\cong \mathfrak{g}_2$. Unlike in the $\rmG_2$ case, however, in the $\SU(3)$ it appears that we also need the condition $(\rho^B)^2_6=(dT_\om)^2_6=0$ in order for $d\nu_1 \in \Lm^{1,1}\cong \mathfrak{u}(3)$.

From Theorem \ref{thm: ricci form of bismut} we can also deduce the following special class of almost CYT 6-manifolds:
\begin{corollary}
    Let $(P,\om,\Om)$ denote a $6$-manifold endowed with an $\SU(3)$-structure with $\hat{N}$ totally skew-symmetric.
    If $\pi_1=2\nu_1$ then $\rho^B=0$.
\end{corollary}
\begin{proof}
    This is again a direct computation using the expression for $\rho^B$ and the identities (\ref{equ: Nijenhuis identity 0})-(\ref{equ: Nijenhuis identity}).
\end{proof}
Before stating the next result, we recall the following basic fact about the representation of $\SU(3)$. As $\SU(3)$ modules one can identify the space of symmetric 2-tensors $S^2(P)$ with differential forms by:
\begin{align}
  S^2(P) &\xrightarrow{\simeq} (\langle \om \rangle \oplus \Lm^2_8) \oplus \Lm^3_{12},\nonumber\\
  \alpha &\mapsto (\alpha \diamond \om, \alpha \diamond \Om^+),
\end{align}
compare with (\ref{equ: identifying 2-tensors with 4-forms}) and see also \cite{Bedulli2007}*{\S 2.3}. In the next result we shall implicitly use this identification between symmetric tensors and differential forms akin to as we did in the $\rmG_2$ case.

\begin{lemma}\label{lemma: scary one}
    Let $(P,\om,\Om)$ denote a $6$-manifold endowed with an $\SU(3)$-structure with $\hat{N}$ totally skew-symmetric. Then 
    the second order quantities $\delta T_\om$, $\mathcal{L}_{\nu_1^\sharp}\om$, $\mathcal{L}_{\nu_1^\sharp}g_\om$, $\Ric^B$ and $\nabla^B \nu_1$ can be expressed in terms of the $\SU(3)$ torsion forms as follows:
    \begin{gather*} \Big(\mathcal{L}_{\nu_1^\sharp}g_\om\Big)_0 = \Big(\mathcal{L}_{\nu_1^\sharp}\om\Big)_0 = 2\Big(S^2(\nabla^B \nu_1)\Big)_0= -\frac{1}{3}\delta\nu_1 ,\\    \Big(\mathcal{L}_{\nu_1^\sharp}g_\om\Big)^2_8 = 2\Big(\mathcal{L}_{\nu_1^\sharp}\om\Big)^2_8
    = 2\Big(S^2(\nabla^B \nu_1)\Big)^2_8 = \pi^2_8\Big( 2 \delta(\nu_1 \w \om)+4 \nu_1 \w J\nu_1 \Big),\\    \Big(\mathcal{L}_{\nu_1^\sharp}g_\om\Big)^3_{12} = -\Big(S^2(\Ric^B)\Big)^3_{12} = 2\Big(S^2(\nabla^B \nu_1)\Big)^3_{12} = \Big( -2 *_\om d*_\om(\nu_1 \w \Om^+)+ 
    2*_\om(\nu_1 \w *_\om(\pi_1 \w \Om^+)) \Big)^3_{12},\\
    \Big(\mathcal{L}_{\nu_1^\sharp}\om\Big)^2_6 \w \Om^+=\frac{1}{2} \Big( \pi_0 J\nu_1 -\sigma_0 \nu_1 + *_\om (*_\om(\nu_1 \w \nu_3)\w\Om^+) - *_\om(d\nu_1 \w \Om^+) \Big) \w \om^2,\\
    (\delta T_\om)_0 = (\mathrm{Skew}(\Ric^B))_0 = (\mathrm{Skew}(\nabla^B \nu_1))_0=0,\\
    (\delta T_\om)^2_8 = 2 (\mathrm{Skew}(\Ric^B))^2_8 = 4(\mathrm{Skew}(\nabla^B \nu_1))^2_8 = 2 \pi^2_8\Big(d\nu_1 + *_\om(\nu_1 \w \nu_3)\Big),\\
    (\delta T_\om)^2_6 \w \Om^+ = 2 (\mathrm{Skew}(\Ric^B))^2_6 \w \Om^+ =  (Jd\sigma_0 + d\pi_0+\pi_0 \pi_1 + \sigma_0 J\pi_1+ J*_\om(*_\om(d\nu_3)\w \Om^+)
    )\w \om^2,\\
    S^2(\Ric^B)_0 = \frac{1}{3}\delta \nu_1 + \frac{1}{3} \delta \pi_1 + \pi_0^2 + \sigma_0^2 - \frac{2}{3} |\nu_1|^2 + \frac{2}{3}g_\om(\nu_1,\pi_1) - \frac{1}{3}|\nu_3|^2,\\
    (S^2(\Ric^B))^2_8 = \pi^2_8\Big( \delta \nu_3 + \delta(\nu_1\w \om) + 2 d(J\pi_1)-4*_\om(\nu_1 \w J\nu_3) \Big),\\
    (\mathrm{Skew}(\nabla^B \nu_1))^2_6 \w \Om^+ =\frac{1}{4}(-\pi_0 \nu_1 - \sigma_0 J\nu_1 + J*_\om (*_\om(\nu_1 \w \nu_3)\w\Om^+) + J *_\om(d\nu_1 \w \Om^+) ) \w \om^2.
    \end{gather*}
\end{lemma}
\begin{proof}
    The proof follows the same representation theoretic argument as before. The constant coefficients are again obtained by tedious verifications.
\end{proof}
The above result should be viewed as an analogue of the expressions we derived in the $\rmG_2$ case in Proposition \ref{prop: ricci of bismut skew1}, Theorem \ref{thm: ricciT and Lie derivative} and Proposition \ref{prop: nablaT tau1}. 
By analogy to Theorem \ref{thm: main without strong}, in the $\SU(3)$ case we have:
\begin{theorem}\label{thm: main thm in su3 case without strong}
Let $(P,\om,\Om)$ denote a $6$-manifold endowed with an $\SU(3)$-structure with $\hat{N}$ totally skew-symmetric and with Lee $1$-form $\nu_1$.
Suppose furthermore that $\rho^B$ and $dT_\om$ are $J$-invariant i.e.  $(\rho^B)^2_6=(dT_\om)^4_6=0$.
Then the following are equivalent:
\begin{enumerate}
    \item $T_\om$ is co-closed and the vector field $\nu_1^\sharp$ preserves the pair $(\om, g_\om)$.
    \item $\nabla^B \nu_1 =0$.
\end{enumerate}
Moreover, if we drop the hypothesis that $(\rho^B)^2_6=(dT_\om)^2_6=0$, then 2. still implies $\nu_1^\sharp$ is a Killing vector field.
\end{theorem}
\begin{proof}
 The result follows almost immediately from Lemma \ref{lemma: scary one}. The hypothesis $(\rho^B)^2_6=(dT_\om)^2_6=0$ is only needed for the equivalence of the terms $(\mathcal{L}_{\nu_1^\sharp}\om)^2_6$, $(\delta T_\om)^2_6$ and $(\mathrm{Skew}(\nabla^B \nu_1))^2_6$. More precisely, using the differential identities (\ref{equ: Nijenhuis identity 0})-(\ref{equ: Nijenhuis identity}) we can equivalently write
    \[
    \Big(\mathcal{L}_{\nu_1^\sharp}\om\Big)^2_6 \w \Om^+ = \Big( \frac{1}{4} *_\om (\delta T_\om \w \Om^+)+d\sigma_0-Jd\pi_0+2\pi_0 J\nu_1 - \pi_0 J\pi_1-2\sigma_0 \nu_1 +\sigma_0 \pi_1
    \Big) \w \om^2.
    \]
    Comparing the latter with the $\Lm^4_6$-component of $dT_\om$ from Theorem \ref{thm: dTom}, we see that $(\mathcal{L}_{\nu_1^\sharp}\om)^2_6$ is equivalent to $(\delta T_\om)^2_6$ when $(dT_\om)^4_6=0$. With the further assumption that $(\rho^B)^2_6=0$ from Corollary \ref{cor: dnu1 is type 1,1} we have that $d\nu_1 \w \Om^+=0$ and hence $(\mathcal{L}_{\nu_1^\sharp}\om)^2_6$ is also equivalent to $(\mathrm{Skew}(\nabla^B \nu_1))^2_6$ from the expressions in Lemma \ref{lemma: scary one}. This concludes the proof.
\end{proof}
By analogy to Theorem \ref{cor: main corollary}, in the $\SU(3)$ case we have:

\begin{theorem}\label{thm: main thm in su3 case}
Let $(P,\om,\Om)$ denote an almost SKT and almost CYT $6$-manifold with Lee $1$-form $\nu_1$. Then the following are equivalent:
\begin{enumerate}
    \item[(a)] $\Ric^B$ vanishes.
    \item[(b)] $T_\om$ is co-closed and the vector field $\nu_1^\sharp$ preserves the pair $(\om, g_\om)$.
    \item[(c)] $\nabla^B \nu_1 =0$.
\end{enumerate}
    Note again that the above result is purely local in nature i.e we do not require compactness or completeness.
\end{theorem}
\begin{proof}
    From Theorem \ref{thm: dTom} and \ref{thm: ricci form of bismut}, it is not hard to show that
    \begin{gather*}
        S^2(\Ric^B)_0=(\rho^B)_0+2(dT_\om)_0+\frac{1}{3}\delta \nu_1,\\
        (S^2(\Ric^B))^2_8 \w \om = \pi^2_8\Big( 2 \rho^B - \mathcal{L}_{\nu_1^\sharp}g_\om\Big) \w \om + \pi^4_8\Big(dT_\om\Big).
    \end{gather*}
    Also recall from Lemma \ref{lemma: scary one} and the proof of Theorem \ref{thm: main thm in su3 case without strong} that $(\mathcal{L}_{\nu_1^\sharp}\om)^2_6$, $(\delta T_\om)^2_6$, $(\mathrm{Skew}(\Ric^B))^2_6$ and $(\mathrm{Skew}(\nabla^B \nu_1))^2_6$ are all equivalent under the the hypothesis of the theorem. Thus, this proves the desired result.
\end{proof}

From Theorem \ref{thm: main thm in su3 case} we have the following consequences of almost SKT, almost CYT and $\Ric^B=0$: 
\begin{enumerate}[label=(\roman*)]
    \item\label{consequence 1 aSKT} Since $\nabla^B$ is a metric connection and $\nabla^B \nu_1=0$, it follows that $\nu_1$ has constant norm.
    \item\label{consequence 2 aSKT} Since $\nu_1$ has constant norm and defines a Killing vector field, from Theorem \ref{thm: dTom} we get 
    \[
    |\nu_3|^2=2 |\nu_1|^2+3(\pi_0^2+\sigma_0^2).
    \]
    Thus, $T_\om$ has constant norm if and only if $\hat{N}$ has constant norm; in particular, this applies to the Hermitian case. Observe also that in the Hermitian case if $\nu_1=0$ then $T_\om=0$, this is however not true in the almost Hermitian case; indeed there is an example on $S^3 \times S^3$ with $\hat{N}\neq 0$ and $\nu_1 =0$, see Example \ref{sec: almost skt s3xs3} below.  
    \item\label{consequence 3 aSKT} Using the fact that $(\rho^B)_0=(dT_\om)_0=0$, from (\ref{equ: scal aSKT}) we get
    \[
    \Scal(g_\om)= 6 (|\nu_1|^2+\pi_0^2+\sigma_0^2) \geq 0
    \]
    and hence $\Scal(g_\om)=0$ if and only if $T_\om=0$ (thus, $\hat{N}=0$).
    In particular, if $P$ is compact but not torsion free then it has positive scalar curvature and thus, $\hat{\mathcal{A}}(P)=0$ cf. \cite{LMspin}*{\S IV Theorem 4.1}.
\end{enumerate}

We should point out some of the above results have also recently appeared in \cite{Ivanov2024ACYT} under the assumption of compactness and using the relation with generalised geometry, see also \cite{ApostolovLeeStreets2024}. We conclude this section with an explicit example of an $\SU(3)$-structure with non-vanishing skew-symmetric Nijenhuis tensor $\hat{N}$.

\begin{example}[An almost CYT but not almost SKT nilmanifold]\label{sec: nilmanifold cyt but not skt}
\textup{Consider the 6-dimensional nilmanifold  associated to the nilpotent  Lie algebra with structure equations
\begin{equation*}
    de^{1} = e^{36}, \quad 
    de^{2} = 0, \quad
    de^{3} = 0, \quad
    de^{4} = e^{26}, \quad
    de^{5} = e^{23}, \quad
    de^{6} = 0.
\end{equation*}
and endowed with the $\SU(3)$-structure given by
\begin{gather*}
   \om = e^{12} + e^{34} + e^{56}, \\
   \Om = \Om^+ + i \Om^- = (e^1 + i e^2) \w (e^3 + i e^4)\w (e^5 + i e^6).
\end{gather*}
A direct computation shows that the only non zero torsion forms are 
$$
    \sigma_0 = \frac{1}{2},\quad
    \nu_3= \frac{3}{4}(e^{135}-e^{146}+3e^{236}-e^{245}).
$$
Hence from Proposition \ref{prop: expression for NJ} we deduce that $\hat{N}$ is indeed skew-symmetric but non zero. The torsion of $\nabla^B$ is given by
\begin{equation*}
    T_\om = e^{136} - 2e^{145} + e^{235} - e^{246}. 
\end{equation*}
Using Theorem \ref{thm: dTom} (or by direct computation) one finds that
\begin{equation*}
    dT_\om = - \om^2.
\end{equation*}
One can check easily that $\Rm^B \neq 0$. On the other hand, from Theorem \ref{thm: ricci form of bismut} we see that $\rho^B=0$. Hence this is an example of an almost CYT manifold which is not almost Bismut flat and not almost SKT.}
\end{example} 

We don't know if there are any examples with $N_J\neq 0$ which are both almost SKT and almost CYT but not almost Bismut flat. In the case when $N_J=0$, such examples have been constructed in \cite{BrienzaCYT} on mapping tori.
In Example \ref{sec: examples on S3xS3} below, we shall show that there is an example of an almost SKT and almost CYT structure on $S^3 \times S^3$ with $N_J\neq 0$; in this case however we also have $\Rm^B=0$. Surprisingly $S^3 \times S^3$ admits another SKT and CYT structure compatible with the same metric but whose $\SU(3)$-structure has $N_J=0$. In \cite{ApostolovLeeStreets2024} it was recently shown that a compact SKT and CYT $6$-manifold is either a Calabi-Yau $3$-fold, or a quotient of either $S^3 \times S^3$ or $S^3 \times \R \times \mathbb{C}$ with $\Rm^B=0$. The latter result crucially relies on the hypothesis that the underlying $6$-manifold is complex and hence the classification in the almost Hermitian case is still an open problem.


\section{\texorpdfstring{$S^1$}{} reduction of the \texorpdfstring{$\rmG_2T$}{} condition}\label{sec: s1 reduction}

Let $(M,\vp,g_{\vp})$ denote an arbitrary $\rmG_2$-structure and suppose that $X$ is a vector field on $M$ generating an $S^1$ action preserving the $3$-form $\vp$. In particular, note that $X$ is a Killing vector field since $g_\vp$ is determined by $\vp$. Assuming this action is locally free, the quotient orbifold $P=M/S^1$ inherits an $\SU(3)$-structure $(\om, \Om)$ given by
\begin{gather}
    \vp = t^{-3} (\eta \w \om + \Om^+),\label{equ: varphi}\\
    *_\vp \vp = t^{-4} (\frac{1}{2}\om^2 - \eta \w \Om^-),\label{equ: dualvarphi}\\
    g_{\vp} = t^{-2} (\eta \otimes \eta +g_{\om}),
\end{gather}
where we define
\begin{gather*}
    \eta := \frac{g_{\vp}(X,\cdot)}{g_{\vp}(X,X)} \qquad \text{and} \qquad  t := g_{\vp}(X,X)^{-\frac{1}{2}}.
\end{gather*}
It is easy to check that $\om,\Om^\pm,t, d\eta$ are all basic i.e. invariant by $X$ and horizontal, as such they indeed descend to the quotient $P$. In what follows we shall identify tensors on $P$ with their corresponding basic tensors on $M$ omitting any pullback maps. We can write the $\rmG_2$ torsion form $\tau_2$ as
\begin{equation*}
\tau_2 = \eta \w \tau_v + \tau_h,
\end{equation*}
where $\tau_v$ and $\tau_h$ denote a 1-form and 2-form on $M$ respectively. As before  $\tau_v$ and $\tau_h$ are basic and as such one can identify them with differential forms on the base. From the results in Section \ref{sec: preliminaries} we can further  write 
\begin{equation*}
    \tau_h= (\tau_h)_0 \om + (\tau_h)^2_6 + (\tau_h)^2_8
\end{equation*}
in terms of its $\SU(3)$-invariant components. 
\begin{proposition}
    The following holds:
    \begin{enumerate}
        \item $(\tau_h)_0 =0$
        \item $(\tau_h)^2_6 \w \om = - \frac{1}{2} \tau_v \w \Om^+$
    \end{enumerate}
\end{proposition}
\begin{proof}
    This follows from a direct computation using that $\tau_2 \w *_\vp\vp=0$.
\end{proof}
The above proposition shows that $\tau_2$ is completely determined by $(\tau_h)^2_6$ and $(\tau_h)^2_8$; this is indeed consistent with the fact that $\tau_2 \in \Lm^2_{14}(M)$ and $14=6+8$. 
In what follows we need to relate the Hodge star operators $*_{\vp}$ and $*_\om$, and to this end we use the following: 
\begin{lemma}\label{lemma: su3 identities 1}
    Let $\al_k \in \Lm^k(P)$, then the following hold:
    \begin{enumerate}
        \item $*_\vp(\al_1)= -t^{-5} \eta \w *_{\om}(\al_1)$, \quad  \quad $*_\vp(\eta \w \al_1)= t^{-3} *_{\om}(\al_1)$.
        \item $*_\vp(\al_2)= +t^{-3}\eta \w *_{\om}(\al_2)$, \quad  \quad $*_\vp(\eta \w \al_2)= t^{-1} *_{\om}(\al_2)$.
        \item $*_\vp(\al_3)= -t^{-1}\eta \w *_{\om}(\al_3)$, \quad  \quad $*_\vp(\eta \w \al_3)= t *_{\om}(\al_3)$.
        \item $*_\vp(\al_4)= +t\eta \w *_{\om}(\al_4)$, \quad \qquad   $*_\vp(\eta \w \al_4)= t^{3} *_{\om}(\al_4)$.
        \item $*_\vp(\al_5)= -t^{3}\eta \w *_{\om}(\al_5)$, \quad \quad \ \  $*_\vp(\eta \w \al_5)= t^{5} *_{\om}(\al_5)$.
    \end{enumerate}
\end{lemma}
Secondly, we shall also use the next lemma in various computations to identify different $\SU(3)$ isomorphisms. 
\begin{lemma}\label{lemma: su3 identities 2}
    The following $\SU(3)$ identities hold:
    \begin{enumerate}
        \item $*_\om(\al_1 \w \om) = - (J\al_1) \w \om$
        \item $*_\om(*_\om(\al_1 \w \Om^+) \w \Om^+) = -2\al_1$
        \item $*_\om(*_\om(\al_1 \w \Om^+) \w \Om^-) = +2J\al_1$
        \item $*_\om(\al_1 \w \om^2)=-2J\al_1$
    \end{enumerate}
\end{lemma}
We now express the $\rmG_2$ torsion forms $\tau_1$ and $\tau_2$ in terms of the torsion of the $\SU(3)$-structure and the curvature $2$-form $d\eta$.
\begin{proposition}\label{prop: tau2 in term of su3 torsion}
The torsion forms $(\tau_h)^2_6$ and $(\tau_h)^2_8$ can be expressed as
    \begin{equation*}
        -3 t (\tau_h)^2_6 \w \Om^- = *_{\om}(d\Om^-) \w \Om^- + 2\om \w d^c \om - 2d\eta \w \Om^+
    \end{equation*}
and 
\begin{equation*}
    t(\tau_h)^2_8 \w \om = (d \Om^-)^4_8.
\end{equation*}
The torsion form $\tau_1$ is given by
\begin{align*}
    \tau_1 &= \frac{1}{12}*_\om(d\Om^- \w \om) \eta - t^{-1}dt + \frac{1}{12} *_\om (\om \w d^c \om - d\eta \w \Om^+ - *_\om (d \Om^-)\w \Om^- )\\
    &= \frac{1}{2} \sigma_0 \eta - t^{-1}dt +\frac{1}{6}(\pi_1 + \nu_1)-\frac{1}{12}*_\om(d\eta \w \Om^+).
\end{align*}

\end{proposition}
\begin{proof}
    The result follows essentially by taking the exterior derivative of (\ref{equ: dualvarphi}) and comparing with (\ref{equ: g2 torsion 1}). The various $\SU(3)$ identities in Lemma \ref{lemma: su3 identities 1} and \ref{lemma: su3 identities 2} can then be used to rewrite the resulting expressions. We skip the details since the computation is similar as in \cite{UdhavFowdar3}*{\S 3}.
\end{proof}

\begin{theorem}[Gibbons-Hawking ansatz for $\rmG_2T$-structures]\label{thm: gibbons hawking ansatz}
    Let $(P,\om,\Om^\pm)$ denote an $\SU(3)$-structure such that $(d\Om^-)^4_8=0$, i.e. the torsion form $\sigma_2$ vanishes, and suppose that there exist a closed 2-form $F_\eta$ defining an integral cohomology class in $H^2(P,\mathbb{Z})$ and satisfying
    \begin{equation}
        *_{\om}(d\Om^-) \w \Om^- + 2\om \w d^c \om = 2F_\eta \w \Om^+.\label{equ: curvature condition}
    \end{equation}
    Then on the total space $M$ of the $S^1$-bundle over $P$ determined by $[F_\eta]$ we can define by (\ref{equ: varphi}) a $\rmG_2$T-structure,   
    where $\eta$ denotes a connection $1$-form satisfying $d\eta=F_{\eta}$ and $t$ is an arbitrary positive function. Conversely, any $\rmG_2T$-structure invariant under a locally free $S^1$-action arises from this construction. 
\end{theorem}
\begin{proof}
    Differentiating (\ref{equ: dualvarphi}) we get
    \begin{equation}
        d(*_\vp \vp ) = - 4 t^{-1} dt \w *_\vp \vp + t^{-4}(\om \w d \om - F_{\eta}\w \Om^- + \eta \w d \Om^-).\label{equ: derivative of dualphi}
    \end{equation}
    On the other hand from (\ref{equ: curvature condition}) we have that 
    \begin{equation*}
        *_{\om}(d\Om^-) \w \Om^+ - 2\om \w d \om + 2F_\eta \w \Om^- =0
    \end{equation*}
    and hence we can rewrite (\ref{equ: derivative of dualphi}) as
    \begin{equation}
        d(*_\vp \vp ) = - 4 t^{-1} dt \w *_\vp \vp + t^{-4}(\frac{1}{2}*_{\om}(d\Om^-) \w \Om^+ + \eta \w d \Om^-).\label{equ: derivative of dualphi 2}
    \end{equation}
    If we now write $d\Om^- = \sigma_0 \om^2 + \pi_1 \w \Om^-$, where we are using that $\sigma_2=0$, then (\ref{equ: derivative of dualphi 2}) becomes
    \begin{align*}
        d(*_\vp \vp ) = (- 4 t^{-1} dt +2 \sigma_0 \eta + \pi_1) \w *_\vp \vp. 
    \end{align*}
    This shows that $\tau_2=0$ and 
    \begin{align}
        4\tau_1 &= - 4 t^{-1} dt +2 \sigma_0 \eta + \pi_1 \nonumber\\
        &= - 4 t^{-1} dt+ \frac{1}{3}*_\om(d\Om^- \w \om)\eta -\frac{1}{2} *_\om (*_\om(d\Om^-)\w \Om^-).\label{equ: s1 invariant g2 lee form}
    \end{align}
    The converse follows immediately from Proposition \ref{prop: tau2 in term of su3 torsion} and this concludes the proof.
\end{proof}

\begin{remark}
    Observe that in Theorem \ref{thm: gibbons hawking ansatz} the condition (\ref{equ: curvature condition}) only imposes a constraint on the $\Lm^2_6$-component of the curvature form:
\begin{equation*}
    2(d\eta)^2_6 \w \Om^- - 2 \nu_1 \w \om^2 + \pi_1 \w \om^2 = 0,
\end{equation*}
    while there are no constraints on the $\langle\om\rangle$- and $\Lm^2_8$-components. This is rather different compared to the Gibbons-Hawking ansatz for closed and torsion free $\rmG_2$-structures whereby $F_\eta$ is fully determined by $d\Om^+$ and $dt$, see \cite{UdhavFowdar3}*{Theorem 3.3} and \cite{Apostolov2003}*{Section 1}. This shows that $\rmG_2T$-structures are more flexible than closed $\rmG_2$-structures. 
\end{remark}

Recall that the $\rmG_2$ Lee form is defined by
$\theta := 4\tau_1.$ 
From the proof of Corollary \ref{cor: tau0 is constant} we know that $d\theta \in \Lm^2_{14}(M)$ since $\tau_2=0$. Assuming $\vp$ is $S^1$-invariant as before, it follows from the proof of Theorem \ref{thm: gibbons hawking ansatz}
that $\theta$ is given by (\ref{equ: s1 invariant g2 lee form}) and hence
\begin{equation}
        d\theta 
        = \frac{1}{3}d(*_\om(d\Om^- \w \om)) \w \eta +\frac{1}{3}*_\om(d\Om^- \w \om)d\eta -\frac{1}{2} d(*_\om (*_\om(d\Om^-)\w \Om^-)).
    \end{equation}
Thus, we deduce the following:
\begin{proposition}\label{prop: closed Lee form}
    If the Lee form $\theta$ is closed on an $S^1$-invariant $\rmG_2T$ manifold then
    \begin{equation*}
        \sigma_0:=\frac{1}{6}*_\om(d\Om^- \w \om) 
    \end{equation*}
    is constant. Furthermore, if $\sigma_0 \neq 0$, then $M^7=S^1 \times P^6$ and
    \begin{equation*}
        d\eta = \frac{1}{4\sigma_0} d(*_\om (*_\om(d\Om^-)\w \Om^-)).
    \end{equation*}
    On the other hand, if $\sigma_0=0$, then $\Om^-$ is locally conformally closed i.e. there exists locally a function $f$ such that $d(f\Om^-)=0$.
\end{proposition}
\begin{remark}
    When the Lee form is closed, the $\rmG_2$-structure is locally conformally co-closed.
\end{remark}

Henceforth, we shall assume that we are in the setup of Theorem \ref{thm: gibbons hawking ansatz}. We write the curvature form in terms of its irreducible components:
\begin{equation*}
    d\eta = (d\eta)_0 \om +(d\eta)^2_6 +(d\eta)^2_8.
\end{equation*}
From (\ref{equ: varphi}) one can show by direct calculation that
\begin{align*}
    7\tau_0 &= *_\vp(d\vp \w \vp)\\
    &= t *_\om(-2 d\om \w \Om^+ + d\eta \w \om^2)\\
    &= t(12 \pi_0+6(d\eta)_0)
\end{align*}
and 
\begin{align}
    *_\vp d\vp &= - t *_\om d(t^{-3}\om) + \eta \w *_\om ( t d(t^{-3}\Om^+)+t^{-2}d\eta \w \om)\\
    &= 3t^{-3} *_\om(dt \w \om) + t^{-2}\Big(\frac{3}{2}\sigma_0 \Om^-+\frac{3}{2}\pi_0 \Om^++J\nu_1 \w \om-J\nu_3\Big)+ \nonumber\\
    &\ \quad t^{-2}\eta \w (2((d\eta)_0+\pi_0) \om+(d\eta)^2_6-(d\eta)^2_8 + \pi_2+ *_\om( (\pi_1-3t^{-1}dt) \w \Om^+) ).
\end{align}
Thus, we can express the torsion $3$-form $T_\vp$, defined by (\ref{equ: definition of T}), as
\begin{align}
    T_\vp =\ &\eta \w \Big( \frac{1}{6} t^{-2} *_\om(d\eta \w \om^2-2 d\om \w \Om^+)\om -t*_\om(d\eta \w t^{-3}\om+d(t^{-3}\Om^+))\label{equ: Tvp in terms of su3 torsion}\\ 
    &- t^{-2}*_\om((4t^{-1}dt+\frac{1}{2}*_\om(*_\om(d\Om^-) \w \Om^-)) \w \Om^+
    )
    \Big)\nonumber\\
    &+ \frac{1}{6} t^{-2} *_\om(d\eta \w \om^2-2 d\om \w \Om^+) \Om^+ +
    t *_\om(d(t^{-3}\om))\nonumber\\
    &+ \frac{1}{3}t^{-2} *_\om(d\Om^- \w \om) \Om^- + t^{-2}*_\om( (4t^{-1}dt+\frac{1}{2}*_\om(*_\om(d\Om^-) \w \Om^-)) \w \om),\nonumber
\end{align}
or equivalently, as
\begin{align}
    T_\vp =\ &t^{-2} \eta \w (*_\om (-t^{-1}dt+\frac{2}{3}\nu_1-\frac{1}{3}\pi_1) \w \Om^+ - (d\eta)_0 \om-\frac{1}{3}(d\eta)^2_6+(d\eta)^2_8-\pi_2)\label{equ: Tvarphi s1 invariant}\\
    &+t^{-2}\Big(\frac{1}{2}\sigma_0 \Om^- +(\frac{1}{2}\pi_0+(d\eta)_0)\Om^+ + J\nu_3 + *_{\om}\big( (t^{-1}dt-\frac{2}{3}\pi_1+\frac{1}{3}\nu_1+\frac{1}{3}*_\om(d\eta \w \Om^+))\w \om \big) \Big)\nonumber.
\end{align}

\textbf{Future question:} Studying the $S^1$-invariant strong $\rmG_2T$ condition, $dT_\vp=0$, in this level of generality is a too difficult problem. In \cite{Apostolov2003} Apostolov-Salamon considered the $S^1$ reduction of torsion free $\rmG_2$-structures and they impose that the quotient $P$ is K\"ahler. This allows them to completely classify such torsion free $\rmG_2$-structures by an analogous Gibbons-Hawking type result, see \cite{Apostolov2003}*{Theorem 1}. 
Motivated by the latter result, one can try to find new examples of strong $\rmG_2T$-structures by imposing suitable constraints on the $\SU(3)$-structure and curvature form $F_\eta$. For instance it would be interesting to classify examples for which $J$ is an integrable almost complex structure. We shall leave this problem for future investigation and conclude this section with an important special case: 

\subsection{Trivial bundle case: relation between strong \texorpdfstring{$\rmG_2T$}{} and almost SKT}
Let us now assume that we have a product structure so that $d\eta=0$ and assume also that $|X|=1$ i.e. $t=1$. Then condition  (\ref{equ: curvature condition}) becomes equivalent to
\begin{equation*}
    \pi_1 = 2 \nu_1,
\end{equation*}
and it follows from (\ref{equ: Tvarphi s1 invariant}) that the $\rmG_2$ torsion 
$3$-form is given by
\begin{equation*}
    T_\vp =  -\eta \w \pi_2
    +\Big(\frac{1}{2}\sigma_0 \Om^- +\frac{1}{2}\pi_0\Om^+ + J\nu_3 - *_{\om}\big( \nu_1\w \om \big) \Big).
\end{equation*}
If furthermore, we assume that the underlying $\SU(3)$-structure $(\om,\Om^\pm)$ has totally skew-symmetric Nijenhuis tensor i.e. $\pi_2=0$ (see Proposition \ref{prop: expression for NJ}) then 
\begin{equation*}
    T_\vp =  
    \frac{1}{2}\sigma_0 \Om^- +\frac{1}{2}\pi_0\Om^+ + J\nu_3 - *_{\om}\big( \nu_1\w \om \big) .
\end{equation*}
On the other hand, we also have that
\begin{align*}
    T_\om := d^c\om - \hat{N} &= J(-\frac{3}{2}\sigma_0 \Om^+ + \frac{3}{2}\pi_0 \Om^-+\nu_1 \w \om+\nu_3) - (-2 \pi_0 \Om^+ - 2 \sigma_0 \Om^-)\\
    &= 
    \frac{1}{2}\sigma_0 \Om^- +\frac{1}{2}\pi_0\Om^+ + J\nu_3 + J\nu_1 \w \om   = T_\vp
\end{align*}
and hence we deduce:
\begin{corollary}\label{cor: almost skt and strong g2t}
    If $(P,\om,\Om^\pm)$ is an almost SKT 6-manifold such that
    \begin{equation*}
        *_{\om}(d\Om^-) \w \Om^- + 2\om \w d^c \om =0
    \end{equation*}
    then the product $\rmG_2$-structure on $M^7=P^6\times S^1$ is a strong $\rmG_2T$-structure.
\end{corollary}
Next we give 2 examples of $\SU(3)$-structures on $S^3 \times S^3$ satisfying the above hypothesis.

\begin{example}[Example of an almost SKT and almost CYT structure on \texorpdfstring{$S^3\times S^3$}{}]\label{sec: almost skt s3xs3}\label{sec: examples on S3xS3}
\textup{
Consider $S^3 \times S^3$ with the usual left invariant co-framing $\{e^i\}_{i=1}^6$ satisfying
\[
de^1 = -2 e^{23}, \quad 
de^2 = -2 e^{31}, \quad 
de^3 = -2 e^{12}, \quad 
de^4 = -2 e^{56}, \quad 
de^5 = -2 e^{64}, \quad 
de^6 = -2 e^{45}. 
\]
We define a new co-framing by
\[
\eta^1 = \frac{1}{2}(e^1-e^4), \quad
\eta^2 = \frac{1}{2}(e^1+e^4), \quad
\eta^3 = \frac{1}{2}(e^2-e^5), \quad
\eta^4 = \frac{1}{2}(e^2+e^5), \quad
\eta^5 = \frac{1}{2}(e^3-e^6), \quad
\eta^6 = \frac{1}{2}(e^3+e^6).
\]
Consider the $\SU(3)$-structure determined by
\begin{gather*}
    \om = \eta^{12} + \eta^{34} + \eta^{56},\\
    \Om^+ + i \Om^- = (\eta^1 + i \eta^2) \w (\eta^3 + i \eta^4)\w (\eta^5 + i \eta^6).
\end{gather*}
It is easy to see that the induced metric is the bi-invariant one. Furthermore, a calculation shows that the only non-zero $\SU(3)$ torsion forms are:
\begin{gather*}
    \sigma_0 = -2,\\
    \nu_3 = 3 \eta^{135} + \eta^{146} + \eta^{236} + \eta^{245}.
\end{gather*}
Thus, from Proposition \ref{prop: expression for NJ} we see that $\hat{N}$ is a non-zero $3$-form. The torsion $3$-form is given by
\[
T_\om = J(d\om) - \hat{N} = -2 (\eta^{136}+\eta^{145}+\eta^{235}+\eta^{246})
\]
and one checks directly that $dT_\om=0$. Thus, this defines an almost SKT structure on $S^3 \times S^3$. From Theorem \ref{thm: ricci form of bismut} it also follows that $\rho^B=0$ i.e. this $\SU(3)$-structure is also almost CYT. In fact one can  verify by direct computation that $\Rm^B=0$ i.e. it is almost Bismut flat. Appealing to Corollary \ref{cor: almost skt and strong g2t} we recover the strong $\rmG_2T$ example found in \cite{FinoG2T2023}.
\begin{remark}\ 
\begin{enumerate}
    \item 
    Note that the above induced almost complex is not the standard nearly K\"ahler one. Indeed the homogeneous nearly K\"ahler structure is instead given by
    \begin{gather*}
    \om_{NK} =\frac{4}{3\sqrt{3}} (\eta^{12} + \eta^{34} + \eta^{56}),\\
    \Om^+_{NK} + i \Om^-_{NK} = (\frac{2}{\sqrt{3}}\eta^1 + i \frac{2}{3}\eta^2) \w (\frac{2}{\sqrt{3}}\eta^3 + i \frac{2}{3}\eta^4)\w (\frac{2}{\sqrt{3}}\eta^5 + i \frac{2}{3}\eta^6).
\end{gather*}
    The induced nearly K\"ahler metric is also Einstein but is different from the bi-invariant one. It is also worth pointing out that in this case $T_{\om_{NK}} = - \Om^-_{NK} $ is not closed, but instead $\nabla^B T_{\om_{NK}} =0$. The curvature $\Rm^B$ is also non-zero. However, from Theorem \ref{thm: ricci form of bismut} we see that again $\rho^B=0$ i.e. it is almost CYT.
    \item There is another almost SKT structure on $S^3 \times S^3$ for which the almost complex structure is in fact integrable i.e. $\hat{N}=0$. This is given by
    \begin{gather*}
        \om_{C} = e^{14} + e^{23} + e^{56},\\
        \Om^+_{C} + i\Om^-_{C} = (e^1 +i e^4)\w (e^2 + i e^3) \w (e^5 + i e^6), 
    \end{gather*}
    and it also induces the bi-invariant metric.     
    One verifies easily that indeed $dd^c\om_C=0$, see also \cite{LotayHeteroticS32024}*{\S 4.1}. Furthermore, a direct computation shows that
    \[
    \pi_1 = 2 \nu_1 = 2 (e^1 - e^4).
    \]
    Thus, this $\SU(3)$-structure also satisfies the hypothesis of Corollary \ref{cor: almost skt and strong g2t} and hence we get another strong $\rmG_2T$-structure on $S^1 \times S^3 \times S^3$. Both strong $\rmG_2$-structures induce the same bi-invariant metric.
\end{enumerate}
\end{remark}
}    
\end{example}


\section{Characteristic Ricci-flat strong \texorpdfstring{$\rmG_2T$}{}-structures}\label{sec: s1 reduction of ricci flat}

Let $(M,\vp,g_\vp)$ denote a characteristic Ricci-flat strong $\rmG_2T$ manifold i.e. $\tau_2=0$, $dT_\vp=0$ and $\Ric^T=0$. 
Let us also assume that the $\rmG_2$ Lee form is non zero. 
In view of Theorem \ref{cor: main corollary}, it follows that the vector field $\tau_1^\sharp$ preserves $\vp$ and has constant norm; by rescaling $\vp$ we shall assume without loss of generality that $|\tau_1|=1$. Suppose that the action generated by $\tau_1^\sharp$ is proper so that the quotient space $P^6:=M^7/\langle \tau_1^\sharp \rangle$ is a smooth manifold (otherwise one has to consider the transverse space to $\langle \tau_1^\sharp \rangle$). Then the latter inherits an $\SU(3)$-structure determined by $(\om,\Om^\pm)$ where
\begin{gather}
    \vp = \tau_1 \w \om + \Om^+,\label{equ: vp tau1}\\
    *_\vp\vp = \frac{1}{2} \om^2 - \tau_1 \w \Om^-.
\end{gather}
Thus, we are back to the set up investigated in Section \ref{sec: s1 reduction} whereby now $\eta=\tau_1$ and $t=1$. 

Since $d\tau_1 \in \Lm^2_{14}(M)$ and it is a basic $2$-form, we have 
\begin{equation*}
  0 =  d\tau_1 \w *_\vp \vp = \frac{1}{2}d\tau_1 \w \om^2 - \tau_1 \w d\tau_1 \w \Om^-
\end{equation*}
i.e.
\begin{equation*}
    d\tau_1 \in \Lm^2_8(P) = [\Lm^{1,1}_0(P)].
\end{equation*}
It follows from the proof of Theorem \ref{thm: gibbons hawking ansatz} that 
$\sigma_0=2$, $\pi_1=\nu_1=0$ and $\sigma_2=0$ i.e.
\begin{gather*}
    d\Om^- = 2 \om^2,\\
    \om \w d\om = 0.
\end{gather*}
From (\ref{equ: definition of T 2}) we have
\begin{equation*}
    *_\vp T_\vp = \frac{1}{6}\tau_0 *_\vp \vp + \tau_1 \w \vp - *_\vp \tau_3
\end{equation*}
and together with (\ref{equ: g2 torsion 1}) we get
\begin{equation*}
    *_\vp T_\vp + d\vp = \frac{7}{6}\tau_0 *_\vp \vp + 4 \tau_1 \w \vp.
\end{equation*}
Taking the exterior derivative of the latter and using that $\delta T_\vp=0$ (which holds from Theorem \ref{cor: main corollary} since $\Ric^T=0$) we find that
\begin{equation*}
    d\Om^+ = \frac{7}{12} \tau_0 \om^2,
\end{equation*}
where we recall from Corollary \ref{cor: tau0 is constant} that $\tau_0$ is constant. 
By rotating $\Om^+ +i \Om^-$, we can define another $(3,0)$-form on $P^6$ by
\begin{equation*}
\Psi^+ + i\Psi^- := \frac{24+7\tau_0 i}{\sqrt{576+49\tau_0^2}}(\Om^+ +i \Om^-).
\end{equation*}
The intrinsic torsion of the $\SU(3)$-structure determined $(\om,\Psi^\pm)$ satisfies:
\begin{gather}
    d\om \w \om = 0,\label{equ: half flat 1}\\
    d\Psi^+ =0,\label{equ: half flat 2} \\
    d\Psi^- = \frac{1}{2}\Big(\frac{49}{36}\tau_0^2+16\Big)^\frac{1}{2} \om^2.\label{equ: half flat 3}
\end{gather}
In other words, $(\om,\Psi^\pm)$ defines a co-coupled $\SU(3)$-structure i.e. it is half-flat with non-vanishing skew-symmetric Nijenhuis tensor. From (\ref{equ: Tvp in terms of su3 torsion}), it also follows that 
\begin{equation}
    T_\vp = \tau_1 \w d\tau_1 + \frac{7}{6}\tau_0 \Om^+ + 4\Om^- + *_\om (d\om).\label{equ: T in terms of SU3}
\end{equation}
By inverting this reduction as in Theorem \ref{thm: gibbons hawking ansatz}, we obtain the following Gibbons-Hawking type theorem:
\begin{theorem}\label{thm: gibbons hawking strong ricciT=0}
    Let $(P,\om,\Psi^\pm)$ denote a 6-manifold endowed with a half-flat $\SU(3)$-structure such that (\ref{equ: half flat 3}) holds for some constant $\tau_0$. Assume also that there exists a closed real $2$-form $F_{\eta} \in \Lm^{1,1}_0(P)$ such that $[F_\eta]\in H^2(P,\mathbb{Z})$ and
    \begin{equation}
        \Delta_{\om} \om = \Big(\frac{49}{36}\tau_0^2+16\Big) \om + *_\om(F_\eta \w F_\eta),\label{equ: eigenform PDE}
    \end{equation}
    where $\Delta_{\om}$ denotes the Hodge Laplacian. Then on the total space $M^7$ of the $S^1$-bundle over $P^6$ determined by $[F_\eta]$ we can define by (\ref{equ: vp tau1}) a characteristic Ricci-flat strong $\rmG_2T$-structure, where $\tau_1$ denotes a connection $1$-form satisfying $d\tau_1=F_{\eta}$. Conversely, any Ricci-flat strong $\rmG_2 T$-structure with non-vanishing Lee form arises locally from this construction.
\end{theorem}
\begin{proof}
  In view of the above reduction, it only suffices to justify (\ref{equ: eigenform PDE}), but this follows immediately by taking the exterior derivative of (\ref{equ: T in terms of SU3}) and using $dT_\vp=0$.
\end{proof}

We shall now rewrite (\ref{equ: eigenform PDE}) into  another interesting way. 
First, observe that since $\nu_1=0$ i.e. $\om \w d\om =0$ we have $d^c \om:= J(d\om)=*_\om d\om$. On the other hand from Proposition \ref{prop: expression for NJ} the Nijenhuis tensor of $(P^6,\om,\Om^\pm)$ is given by
\begin{equation*}
  \hat{N}  = -\frac{7}{6}\tau_0 \Om^+ - 4 \Om^-.
\end{equation*}
Denoting the torsion form of the Bismut connection of $(P^6,\om,\Om^\pm)$ by $T_{\om}$ then we get
\begin{align*}
    T_{\om} &= d^c\om - \hat{N}\\ 
    &= *_\om (d\om) + \frac{7}{6}\tau_0 \Om^+ + 4 \Om^-\\
    &= T_\vp - \tau_1 \w d\tau_1.
\end{align*}
Hence it follows that (\ref{equ: eigenform PDE}) is equivalent to the ``Bianchi-identity for the $\SU(3)$ heterotic system'':
\begin{equation}\label{equ: heterotic}
\boxed{dT_{\om} = - d\tau_1 \w d \tau_1 \in \Lm^{2,2}(P).}
\end{equation}
\noindent Some immediate consequences:
\begin{enumerate}[label=(\alph*)]
    \item \label{consequence 1} Since the $\rmG_2$ torsion forms  $\tau_0,\tau_1,\tau_3$ all have constant norm, see \ref{strong ricci flat constant T}, and we have normalised $|\tau_1|=1$, a direct computation using (\ref{equ: T in terms of SU3}) shows that
    \begin{equation}
        |d\tau_1|^2 + |\nu_3|^2 = 12 + \frac{49}{48}\tau_0^2.\label{equ: nu3 constant}
    \end{equation}
    In particular, the curvature form $F_{\eta}=d\tau_1$ has constant norm if and only if $|T_\om|$ is constant. Below we shall show that exist examples with both $F_\eta$ zero and non-zero.

    \item \label{consequence 2} Observe that since $F_{\eta} \in \Lm^2_8(P)=\Lm^{1,1}_0$ we can write
    \begin{equation}
    *_\om(F_{\eta} \w F_{\eta}) = -\frac{1}{3}|F_{\eta}|^2 \om + \gamma,\label{equ: FA square}
    \end{equation}
    where $\gamma \in \Lm^2_8(N)$. Thus, from (\ref{equ: heterotic}) it follows that $dT_{\om}=0$ if and only if $d\tau_1=0$. In other words, $(P^6,\om,\Om)$ is an almost SKT manifold if and only if $M^7=S^1 \times P^6$ is a Riemannian product, see also Proposition \ref{prop: closed Lee form}. 
    
    \item \label{consequence 3} Using \cite{Bedulli2007}*{(3.11)} together with \ref{strong ricci flat g2t scal} we deduce that the scalar curvature of $P^6$ is given by
    \begin{equation*}
        \Scal(g_{\om}) = \frac{49}{24}\tau_0^2 + 24 + \frac{1}{2}|d\tau_1|^2>0.
    \end{equation*}
    Thus, if $P$ is a compact manifold then $\hat{\mathcal{A}}(P)=0$.
    \item \label{consequence 4} Since $\nu_1=\pi_1=0$ and $\pi_0$ is constant, we immediately deduce from Theorem \ref{thm: ricci form of bismut} that the Bismut-Ricci curvature $\rho_B$ vanishes i.e. the quotient space $(P^6,\om,\Om^\pm)$ is almost CYT.
\end{enumerate}

\textbf{Future question:} In \cite{FinoG2T2023}*{Corollary 3.2} it is shown that there are no (non-trivial) invariant strong $\rmG_2T$-structures on unimodular solvable Lie groups; in particular, there are no such characteristic Ricci-flat examples. On the hand in \cite{Freibert2012} it is shown that every non-solvable 6-dimensional Lie algebra admits an invariant half-flat $\SU(3)$-structure. Thus, in view of Theorem \ref{thm: gibbons hawking strong ricciT=0} it is natural to ask if there are non-solvable 6-dimensional Lie algebras satisfying the hypothesis of Theorem \ref{thm: gibbons hawking strong ricciT=0}.

\subsection{Strong \texorpdfstring{$\rmG_2T$}{}-structures with parallel torsion}

In this section we consider strong $\rmG_2T$-structures with characteristic parallel torsion form $T_\vp$. First observe that since $\nabla^T$ is a $\rmG_2$ connection
from (\ref{equ: definition of T 2}) we have 
\begin{equation*}
    \nabla^T T_\vp = *_\vp (\nabla^T\tau_1 \w \varphi) - \nabla^T\tau_3,
\end{equation*}
where we used from Corollary \ref{cor: tau0 is constant} that $\tau_0$ is constant. In particular, it follows from Theorem \ref{cor: main corollary} that if $\nabla^T T_\vp = 0 $ then $\Ric^T=0$. Thus, if $(M,\vp,g_\vp)$ is a strong $\rmG_2T$-structure with parallel torsion then it is characteristic Ricci-flat and moreover, $\nabla^T \tau_3 =0$; henceforth we shall assume this is the case. 

From (\ref{equ: T in terms of SU3}) it is easy to show
\begin{equation*}
    *_\vp \tau_3 = \tau_1 \w (\frac{1}{8}\tau_0 \Om^- - \nu_3)+d\tau_1 \w \om + \frac{1}{12}\tau_0 \om^2.
\end{equation*}
Observe that as $\nabla^T \tau_1=0$ the holonomy algebra of $\nabla^T$, $\mathfrak{hol}(\nabla^T)$, lies in $\mathfrak{su}(3)\subset \mathfrak{g}_2$ and thus,
\begin{equation*}
  0=*_\vp \nabla^T \tau_3 = -\tau_1 \w \nabla^T\nu_3+(\nabla^T d\tau_1 )\w \om. 
\end{equation*}
Since $d\tau_1 \in \Lm^2_8 \cong \mathfrak{su}(3)$ and every element in $\mathfrak{su}(3)$ is conjugate to an element in the Cartan subalgebra $\mathfrak{t}\cong \R^2$, we can write
\begin{equation*}
    d\tau_1 = \lm_1 e^{23} + \lm_2 e^{45} - (\lm_1+\lm_2) e^{67},
\end{equation*}
where $\{e^i\}$ denotes a $\rmG_2$ co-framing at a given fixed point, $\tau_1=e^1$ and $(\lm_1,\lm_2)\in \mathfrak{t}$, see for instance \cite{Bryant06someremarks}*{\S 2.7.2}. 
If $d\tau_1=0$, then from the results in the last section it follows that $M^7=S^1 \times P^6$, $T_\vp=T_\om$ and $\nabla^T=\nabla^B$. Thus, we have $\nabla^B T_\om = \nabla^T T_\vp =0$. From \cite{AGRICOLA201559}*{Theorem 4.1} it follows that $(M^7,g_\vp)$ is the product of $\mathbb{R}$ and a Lie group with its bi-invariant metric up to covering. 

Let us now assume that $(\lm_1,\lm_2)\neq 0$ is generic, more precisely we require 
\begin{equation*}
(\lm_1-\lm_2)(2\lm_1+\lm_2)(\lm_1+2\lm_2)\neq0,
\end{equation*} 
so that the stabiliser of $d\tau_1$ has Lie algebra precisely $\mathfrak{t}$. In other words, in this case the condition $\nabla^T d\tau_1=0$ implies that $\mathfrak{hol}(\nabla^T)\subset \mathfrak{t}$. From \cite{AGRICOLA201559}*{Proposition 2.2 (3)} it follows that $\nabla^T\mathrm{Rm}^T=0$ and hence $M^7$ is naturally reductive. We shall now investigate under which condition one can further reduce the holonomy group.

First note that $\nu_3 = 0$ precisely when $P^6$ is nearly K\"ahler; let's assume this is not the case. An inspection of the irreducible $\SU(3)$ module $\Lm^3_{12}$ shows that there are no non-trivial element preserved by $\mathfrak{t}$. Hence since $\nabla^T \nu_3=0$ it follows that $\mathrm{dim}(\mathfrak{hol}(\nabla^T)) < 2$. If $\mathrm{Hol}_0(\nabla^T)=\mathrm{diag}(e^{ik\theta},e^{il\theta},e^{-i(k+l)\theta}) \subset \SU(3)$ where $k,l\in \mathbb{Z}$, then one can show that for $k,l,(k+l)$ all non-zero there are no non-zero invariant element in $\Lm^3_{12}$. Since we can permute $k,l$ and $(k+l)$, we can assume without loss of generality that $l=0$. One then finds that the most general invariant element in $\Lm^3_{12}$ depends on four parameters and is of the form:
\begin{align*}
    \nu_3 =\ &((\mu_1+2\mu_4)e^4+\mu_2 e^5)\w (e^{23}-e^{67}) +
    (\mu_4 e^5 - \mu_3 e^4)\w (e^{26}-e^{37})\\ &- (\mu_4 e^4 + \mu_3 e^5)\w (e^{27}+e^{36}),
\end{align*}
where $\mu_i\in \R$; indeed it is easy to check that $(e^{23}-e^{67}) \diamond \nu_3=0$. We shall say that $\nu_3$ is generic if it is not of the latter form. Thus, we can the summarise the above into the following:
\begin{theorem}
Let $(M,\vp,g_\vp)$ denote a strong $\rmG_2T$ manifold with characteristic parallel torsion form $T_\vp$. If the curvature $2$-form $d\tau_1$ is generic then $\nabla^T \Rm^T=0$ i.e. $M^7$ is naturally reductive.
If furthermore, the torsion form $\nu_3$ of $P^6$, as in Theorem \ref{thm: gibbons hawking strong ricciT=0}, is also generic then $\Rm^T=0$ i.e. it is characteristic flat.
\end{theorem}

\subsection{Explicit examples.}\label{sec: explicit example}

\begin{example}\label{example: s3s3s1} Consider on $S^3 \times S^3 \times S^1$ the global coframe $\{E^i\}$ satisfying $dE^1 = 0$ and  
\begin{equation*}
dE^2 = 2\sqrt{2} e^{45}, \quad 
dE^3 = 2\sqrt{2} e^{67}, \quad 
dE^4 = 2\sqrt{2} e^{52}, \quad 
dE^5 = 2\sqrt{2} e^{24}, \quad 
dE^6 = 2\sqrt{2} e^{73}, \quad
dE^7 = 2\sqrt{2} e^{36}.  
\end{equation*}
We define a $\rmG_2$ $3$-form by
\[
\vp_C := E^{123}+E^{145}+E^{167}+E^{246}-E^{257}-E^{347}-E^{356}.
\]
Up to a constant factor, this corresponds to the strong $\rmG_2T$-structure obtained from the SKT structure on $S^3 \times S^3$, i.e. the one with $\hat{N}=0$ via Corollary \ref{cor: almost skt and strong g2t}, see Example \ref{sec: almost skt s3xs3}.
A direct computation shows that the only non-zero $\rmG_2$ torsion forms are given by
\begin{gather*}
    \tau_1 = -\frac{1}{\sqrt{2}}(E^2-E^3),\\
    \tau_3 = \frac{1}{\sqrt{2}} (-E^{146}+E^{147}+E^{156}+E^{157}-3E^{245}+
    E^{267}+E^{345}-3E^{367}).
\end{gather*}
One can check that the unit vector field $\tau_1^\sharp=-\frac{1}{\sqrt{2}}(E_2-E_3)$ preserves $\vp_C$. In view of Corollary \ref{cor: almost skt and strong g2t} we already know that 
\[
T_{\vp_C} = 2\sqrt{2}(E^{245}+E^{367})
\]
is closed, and it also easy to check that it is co-closed as well. Thus, from Theorem \ref{cor: main corollary} it follows that $\Ric^T=0$. In fact, one can also verify that $\nabla^T T_\vp=0$ and $\Rm^T=0$.

The above quotient construction shows that $P^6=S^1\times S^2 \times S^3 \cong (S^1 \times S^3 \times S^3)/S^1_{\tau_1}$ admits a co-coupled $\SU(3)$-structure which is almost CYT but not almost SKT since
\begin{equation}
    d\tau_1 = -2(E^{45}-E^{67}).\label{equ: hypothesis of rmt=0}
\end{equation}
Put differently, $S^1 \times S^2 \times S^3$ admits a co-coupled $\SU(3)$-structure and a non-trivial complex line bundle with a Hermitian connection $\tau_1$ solving the heterotic $\SU(3)$ Bianchi-identity (\ref{equ: heterotic}).

On the other hand, if we consider the strong $\rmG_2T$-structure arising from the \textit{strictly} almost SKT and almost CYT structure on $S^3 \times S^3$ via Corollary \ref{cor: almost skt and strong g2t}, then it turns out that $\tau_1$ corresponds to the $1$-form on the product $S^1$ factor and as such in this case 
$$d\tau_1=0.$$ 
Thus, we now have that $P^6=S^3 \times S^3 \cong (S^1 \times S^3 \times S^3)/S^1_{\tau_1}$. 
These two examples illustrate that one can have both $F_\eta$ zero or non-zero in Theorem \ref{thm: gibbons hawking strong ricciT=0}. This also shows that $S^3 \times S^3 \times S^1$ admits two distinct strong $\rmG_2T$-structures; these are isometric but not equivalent as $\rmG_2$-structures since $\tau_1$ is closed for one but not the other. In particular, these two $\rmG_2$-structures cannot be diffeomorphic.

It was shown in \cite{FinoG2T2023} that 
the only compact homogeneous examples of strong $\rmG_2T$ manifolds are $S^3 \times S^3 \times S^1$ and $S^3\times \mathbb{T}^4$. Our result here shows that the underlying strong $\rmG_2T$-structures are however not unique in general. It would interesting to classify all the possible distinct strong $\rmG_2T$-structures on these spaces.

Since these examples have $\Rm^T=0$ (hence $\Ric^T=0$), one might naively expect that this is always the case for strong $\rmG_2T$-structures. However, in Section \ref{sec: lots of examples} we shall exhibit local examples of strong $\rmG_2T$-structures with $\Ric^T \neq 0$. From Theorem \ref{cor: main corollary} we know that in this case we cannot have both $T_\vp$ closed and co-closed, and $\tau_1^\sharp$ preserving $\vp$. Indeed we shall show that there are examples with $T_\vp$ closed and co-closed but $\tau_1^\sharp$ non Killing, and also there are examples where $T_\vp$ is closed but not co-closed.
\end{example}
\begin{example}\label{example: N4 x S3}
    Let $N^4$ denote a hyperK\"ahler $4$-manifold with (self-dual) hyperK\"ahler triple given by $\om_1,\om_2,\om_3$. Consider now $M^7=S^3 \times N^4$ endowed with the product $\rmG_2$-structure defined by
    \[
    \vp = \om_1 \w e^1 + \om_2 \w e^2 + \om_3 \w e^3 - e^{123},
    \]
    where the left-invariant co-framing on $S^3$ satisfies $de^i=\frac{1}{2}\epsilon_{ijk}e^{jk}$. A straightforward calculation shows that $d*_\vp\vp=0$ i.e. $\tau_1=\tau_2=0$ and
    \begin{gather*}
        \tau_0=\frac{6}{7},\\
        \tau_3= \frac{1}{7}(\om_1 \w e^1 + \om_2 \w e^2 + \om_3 \w e^3)+\frac{6}{7}e^{123}.
    \end{gather*}
    It now follows that
    \[
    T_\vp = -e^{123}
    \]
    and hence we see that $T_\vp$ is both closed and co-closed. Thus, this corresponds to a strong $\rmG_2T$-structure. This seems to have first been observed in \cite{Passias2021}.
    
    Since the $\rmG_2$ Lee form $\theta=4\tau_1$ vanishes, it follows from Theorem \ref{thm: ricciT and Lie derivative} that $\Ric^T=0$. Moreover, it is not hard to see that $\Rm^T$ is not identically zero unless $N^4$ is flat (as $\Rm^T$ essentially corresponds to the Weyl curvature of $N^4$). Thus, this example shows that there exist non-torsion free strong $\rmG_2T$-structures with vanishing Lee form and non-vanishing characteristic curvature. To the best of our knowledge, these are the only known examples with $\Ric^T=0$ but $\Rm^T\neq 0$. 
\end{example}


\section{Examples of non characteristic Ricci-flat strong \texorpdfstring{$\rmG_2 T$}{}-structures}\label{sec: lots of examples}

\subsection{A harmonic but non characteristic Ricci-flat strong \texorpdfstring{$\rmG_2 T$}{}-structure}\label{example: harmonic but not ricci flat}

Consider the same $6$-dimensional nilmanifold as in Example \ref{sec: nilmanifold cyt but not skt} but now with the structure equations given by
\begin{gather*}
    de^2 = 0, \qquad de^4 = 0, \qquad de^6 = 0, \qquad
    de^3 = -e^{46}, \qquad de^5 = e^{26}, \qquad de^7 = -e^{24}. 
\end{gather*}
On the manifold $M^7=\R^+ \times N^6$, if we define an orthonormal co-framing by
\begin{gather*}
 E^1=t^{3/2}dt, \quad   E^2 = te^2, \quad E^4 = te^4, \quad E^6 = te^6, \quad
    E^3 = t^{-1/2}e^3, \quad E^5 = t^{-1/2}e^{5}, \quad E^7 = t^{-1/2}e^{7}. 
\end{gather*}
then it turns out that
\begin{equation}
\vp:=
E^{123}+E^{145}+E^{167}+E^{246}-E^{257}-E^{347}-E^{356} \label{equ: expression for varphi}
\end{equation}
defines a torsion free $\rmG_2$ holonomy metric. On the other hand, if we define an orthonormal co-framing by
\begin{gather*}
 E^1=dt, \quad   E^2 = 2t^{1/2}e^2, \quad E^4 = 2t^{1/2}e^4, \quad E^6 = 2t^{1/2}e^6, \quad
    E^3 = e^3, \quad E^5 = e^{5}, \quad E^7 = e^{7},
\end{gather*}
and define $\vp$ again by the expression (\ref{equ: expression for varphi}), then we have that $\tau_0=0=\tau_2$ and  
\begin{gather*}
    \tau_1 = \frac{3}{16 t} E^1,\\
    \tau_3 = - \frac{1}{16 t} (E^{247}+E^{256}+E^{346}+E^{357}).
\end{gather*}
Observe that the $\rmG_2$ Lee form is closed (with $\sigma_0=0$ as in Proposition  \ref{prop: closed Lee form}). It follows by a simple computation that
\begin{equation*}
    T_{\vp} = \frac{1}{4t}(E^{247}+E^{256}+E^{346}).
\end{equation*}
One verifies directly from the latter expression that $T_\vp$ is both closed and co-closed i.e. $\vp$ defines a harmonic strong $\rmG_2T$-structure.

On the other hand from the above expressions it is also clear that $\tau_1$ and $T_\vp$ do not have constant norms, and hence in view of \ref{strong ricci flat constant tau1} and \ref{strong ricci flat constant T} $\Ric^T \neq 0$. Nonetheless since $T_\vp$ is both closed and co-closed, we know that $\Ric^T$ is symmetric, see Theorem \ref{thm: T is harmonic and RicciT}. From Theorem \ref{thm: ricciT and Lie derivative} and Proposition \ref{prop: nablaT tau1} it follows that
\begin{equation*}
    \mathcal{L}_{\tau_1^\sharp} g_\vp 
    = -\frac{1}{2} \Ric^T 
    = \frac{1}{2} \nabla^T \theta
    = -\frac{3}{16 t^2} \Big( 2 (E^1)^2 - (E^2)^2 - (E^4)^2 - (E^6)^2 \Big).
\end{equation*}

\begin{remark}
Both the aforementioned torsion free and harmonic strong $\rmG_2T$-structures are $\mathbb{T}^3$-invariant but not Riemannian products; in particular, the torsion free metric has (local) Riemannian holonomy group equal to $\rmG_2$ and the strong $\rmG_2T$-structure has holonomy group equal to $\SO(7)$. This is different from the examples on $S^3 \times S^3 \times S^1$ which instead have holonomy group equal to $\SO(3)\times \SO(3)$.
\end{remark}

\subsection{A non harmonic strong \texorpdfstring{$\rmG_2 T$}{}-structure}

So far all the examples of strong $\rmG_2T$-structures we have seen had $T_\vp$ co-closed, but next we give an example which shows that this is not always the case. 

Consider the $6$-dimensional nilmanifold with the structure equations given by
\begin{gather*}
    de^2 = 0, \qquad de^3 = 0, \qquad de^4 = 0, \qquad
    de^5 = e^{27}, \qquad de^6 = e^{42}, \qquad de^7 = 0. 
\end{gather*}
We define a $\rmG_2$ co-framing on $M^7=(0,C) \times N^6$ by
\begin{gather*}
 E^1=\frac{t^2}{\sqrt{C^2-t^2}}dt, \quad 
 E^2 = t e^2, \quad 
 E^3 = e^3, \quad 
 E^4 = te^4, \quad
 E^5 = e^5, \quad 
 E^6 = e^6, \quad 
 E^7 = t e^7,
\end{gather*}
where as before $\vp$ is given by expression (\ref{equ: expression for varphi}). A calculation shows that $\tau_0=0=\tau_2$ and
\begin{gather*}
    \tau_1 = \frac{\sqrt{C^2-t^2}}{2t^3} E^1,\\
    \tau_3 =\frac{1}{t^2}(E^{145}-E^{167})-\frac{\sqrt{C^2-t^2}}{2t^3}(3E^{247}-E^{256}-E^{346}+E^{357}).
\end{gather*}
One can check that $dT_\vp=0$ but
\begin{equation*}
    \delta T_\vp = \frac{2\sqrt{C^2-t^2}}{t^5}(E^{45}-E^{67}) \in \Lm^2_{14}.
\end{equation*}
We also have that
\begin{equation*}
    \mathcal{L}_{\tau_1^\sharp} g_\vp = -\frac{2t^2-3C^2}{t^6} (E^1)^2 + \frac{C^2-t^2}{t^6}\big( (E^2)^2 + (E^4)^2 + (E^7)^2 \big).
\end{equation*}
Hence from Theorem \ref{thm: T is harmonic and RicciT} and \ref{thm: ricciT and Lie derivative} we get
\begin{equation*}
    \Ric^T =  -2\mathcal{L}_{\tau_1^\sharp} g_\vp + \frac{1}{2}\delta T_\vp.
\end{equation*}

\subsection{Cohomogeneity one examples with \texorpdfstring{$\SU(2)^3$}{} symmetry} 

In \cite{Bryant1989} Bryant and Salamon showed that the spinor bundle of $S^3$, $\slashed{S}(S^3)\cong \R^4\times S^3$, admits a complete 
$\SU(2)^3$-invariant
cohomogeneity one $\rmG_2$ holonomy metric. Thus, it is natural to ask if one can find non-trivial cohomogeneity one examples of strong $\rmG_2$-structures on this space. Here we shall show that there exists half-complete examples i.e. these examples have two ends of which exactly one is incomplete. 

Consider the co-framing $\{\eta^i\}_{i=1}^6$ on $S^3 \times S^3$ as defined in Example \ref{sec: almost skt s3xs3}. 
We define a $\rmG_2$ co-framing on $\R_t \times S^3 \times S^3$ by
\begin{gather*}
 E^1 = h(t)\eta^1, \quad 
 E^2 = f(t)\eta^2, \quad 
 E^3 = h(t)\eta^3, \quad 
 E^4 = f(t)\eta^4, \quad
 E^5 = h(t)\eta^5, \quad 
 E^6 = f(t)\eta^6, \quad 
 E^7 = dt,
\end{gather*}
so that the $\SU(2)^3$-invariant $\rmG_2$ $3$-form can be expressed as $$\vp:= (E^{12}+E^{34}+E^{56}) \w E^7+ E^{135}-E^{146}-E^{236}-E^{245}.$$

A long but straightforward computation shows that $\tau_0=0=\tau_2$ and $\delta T_\vp=0$. On the other hand, we find that $dT_\vp=0$ if and only if
\begin{gather}
    h'(t) h(t) f(t) - 2 f(t)^2 = C,\label{equ1: g2T S3 x S3 x R}\\    -2h(t)^2f(t)^2f'(t)+f(t)^3h(t)h'(t)-4f(t)^4+2f(t)^2h(t)^2= C h(t)^2,\label{equ2: g2T S3 x S3 x R}
\end{gather}
where $C$ is a constant. 

We summarise the explicit solutions we have been able to find:
\begin{enumerate}
    \item When $C=0$, the pair (\ref{equ1: g2T S3 x S3 x R})-(\ref{equ2: g2T S3 x S3 x R}) reduces to the torsion free condition and we recover the Bryant-Salamon solution by setting $dt = (1-s^{-3})^{-\frac{1}{2}} ds$ and
\begin{equation*}
  h(s)=  \frac{2}{\sqrt{3}}s, \qquad  f(s) = \frac{2}{3} s\sqrt{1-s^{-3}},
\end{equation*}
where $s\in [1,+\infty)$, see also \cite{Lotay2016}*{\S 2.2}.

    \item If we assume that $C<0$ and $f'(t)=h'(t)=0$ then we get $2f(t)^2 = 2 h(t)^2 = - C$. This yields a \textit{complete} solution on $\R \times S^3 \times S^3$ which, up to covering, corresponds to the example described in Example \ref{sec: almost skt s3xs3}. Recall that in this case we have $\Ric^T=0$.
    
    \item If instead we assume that $C>0$ and set $f(t)=\sqrt{C/2}$, then solving for $h(t)$ we find
    \[h(t)=\sqrt{4t\sqrt{2C}-c_0},\] 
    where $c_0\in \R$. Hence we get a solution which is only well-defined for $t>\frac{c_0}{4\sqrt{2C}}$. In contrast to the previous example, in this case $\Ric^T = -2\mathcal{L}_{\tau_1^\sharp}g_\vp \neq 0$ (see Theorem \ref{thm: T is harmonic and RicciT} and \ref{thm: ricciT and Lie derivative}). 
    This example is curiously similar to the one described in 
    Section \ref{example: harmonic but not ricci flat}; both have harmonic $T_\vp$ but are non characteristic Ricci-flat and moreover, both underlying spaces admit a torsion free $\rmG_2$ holonomy metric (this `duality' was first observed in \cite{Gibbons01} by taking a ``Heisenberg limit'').
\end{enumerate}

Let us now assume that $f'(t)\neq 0 \neq h'(t)$. 
From (\ref{equ1: g2T S3 x S3 x R}) we have
\begin{equation*}
dt = \frac{hf}{C+2f^2} dh    
\end{equation*}
and hence by changing coordinates we can rewrite the induced metric as
\begin{equation*}
    g_\vp =  \frac{h^2 f^2}{(C+2f^2)^2} dh^2 + h^2 (\eta_1^2+\eta_3^2+\eta_5^2) +
    f^2 (\eta_2^2+\eta_4^2+\eta_6^2).
\end{equation*}
The pair (\ref{equ1: g2T S3 x S3 x R})-(\ref{equ2: g2T S3 x S3 x R}) can be combined into the single equation:
\begin{equation}
    \frac{df}{dh} = \frac{(C-2f^2)(f^2-h^2)}{2hf(C+2f^2)}.\label{equ: single ode}
\end{equation}
If we write down a power series solution to (\ref{equ1: g2T S3 x S3 x R})-(\ref{equ2: g2T S3 x S3 x R}) defined in a neighbourhood of the zero section $S^3$ then a calculation shows that the metric $g_\vp$ extends to a smooth solution if and only if $C=0$ (this is obtained by appealing to \cite{Lotay2016}*{Lemma 8}). Thus, we deduce that this ansatz never yields a complete solution on $\slashed{S}(S^3)$ aside from the torsion free one. Nonetheless, numerics show that there are solutions defined on the complement of a neighbourhood of the zero section $S^3$ and moreover, that these solutions are asymptotic to the Bryant-Salamon cone metric; this is illustrated in Figure \ref{fig:figure1}. In all the three graphs the asymptotic linear curve corresponds to the Bryant-Salamon cone metric. The complete Bryant-Salamon metric corresponds to the curves emanating from the positive $x$-axis for $C=0$. In the case $C=-2$, there is a fixed point occurring at $(1,1)$ corresponding to the complete solution on $\R \times S^3 \times S^3.$ In the case $C=+2$, the horizontal line corresponds to the aforementioned half complete explicit solution.

\begin{figure}[htp]
\centering
\includegraphics[width=.33\textwidth]{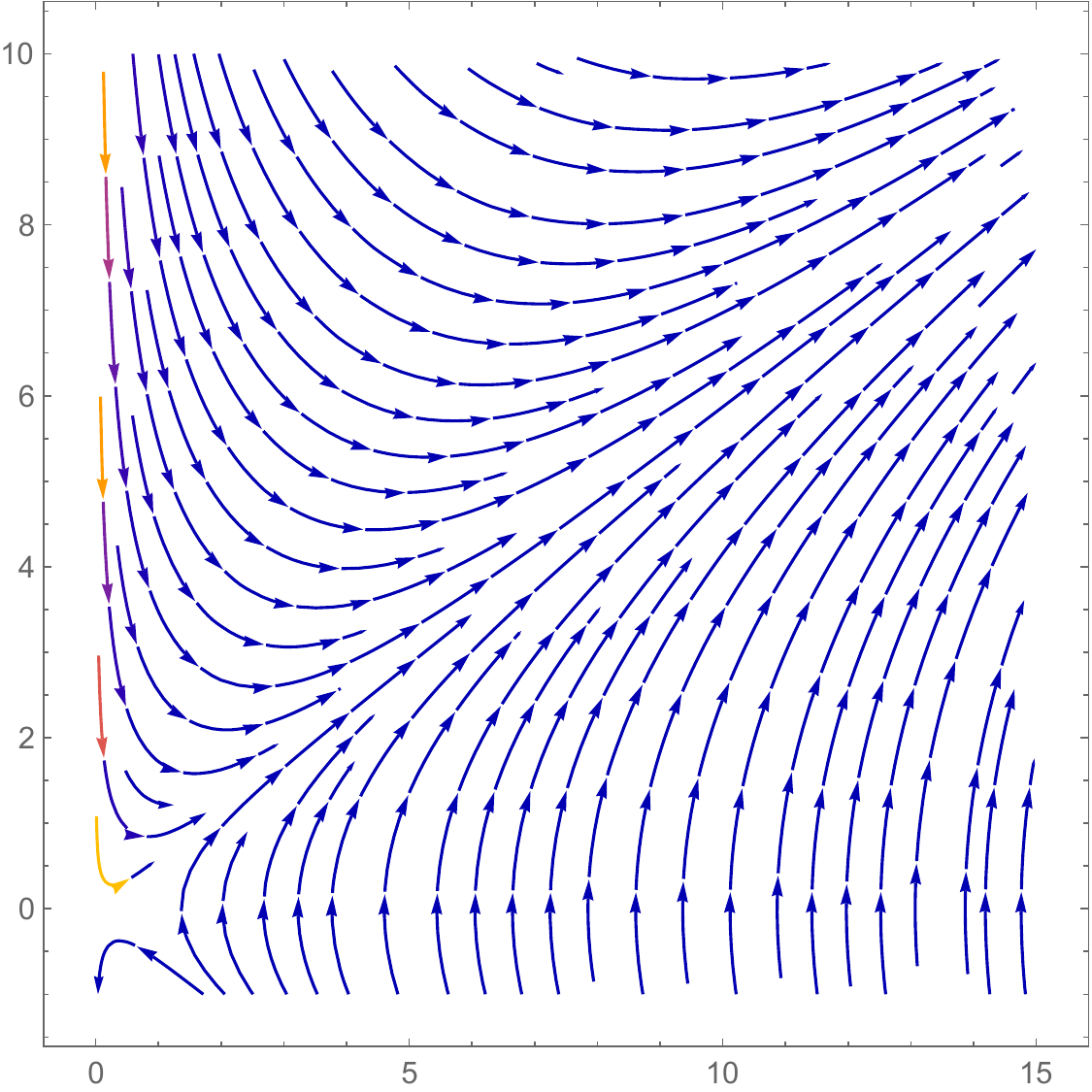}\hfill
\includegraphics[width=.33\textwidth]{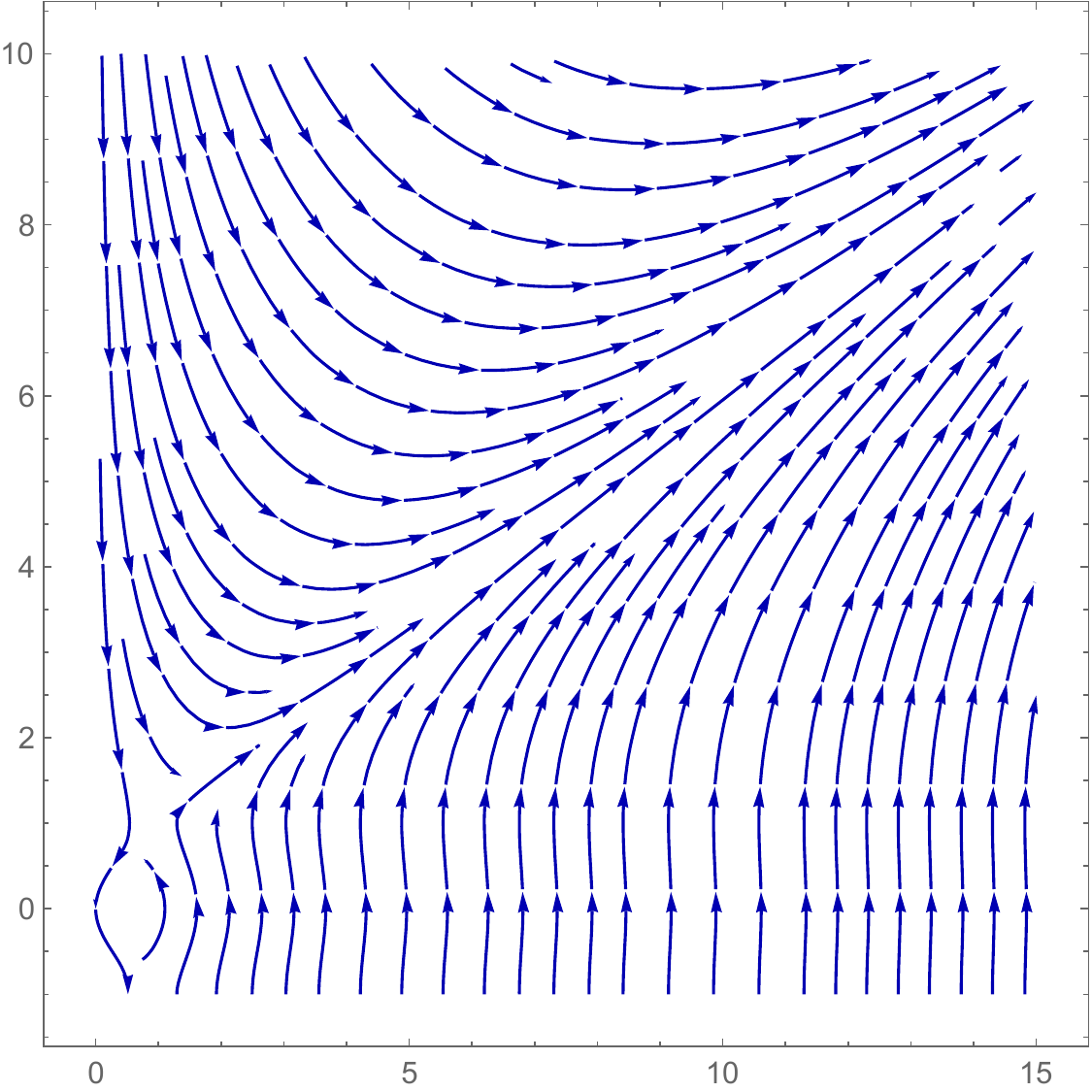}\hfill
\includegraphics[width=.33\textwidth]{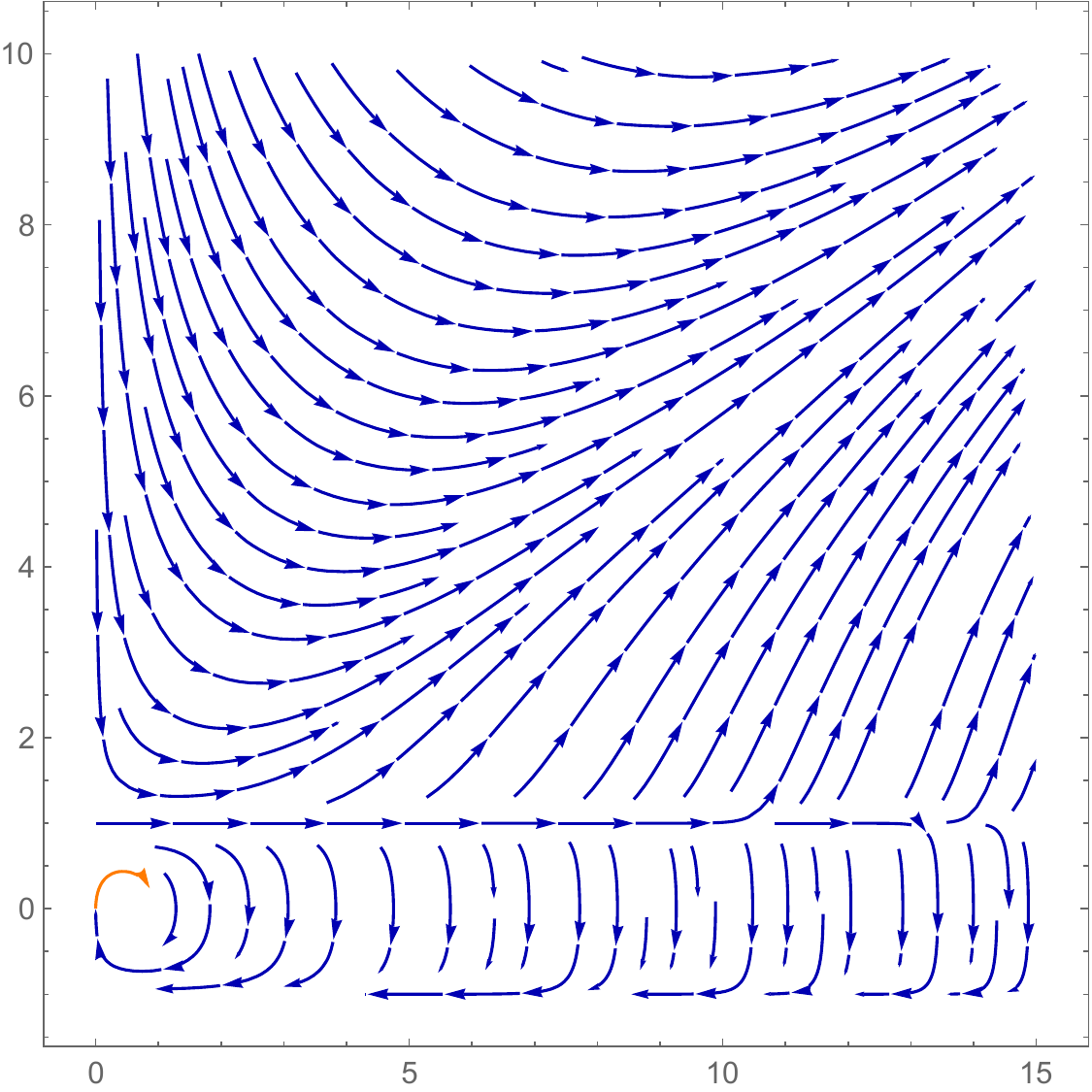}
\caption{{Graphical solutions to (\ref{equ: single ode}) for $C=0,-2,+2$.}}
\label{fig:figure1}
\end{figure}

\begin{remark}
It is worth noting that the above ansatz with $\SU(2)^3$ symmetry automatically implies that $\vp$ defines a $\rmG_2T$-structure i.e. $\tau_2=0$. By contrast, for the other Bryant-Salamon manifolds, $\Lm^2_- \C\mathbb{P}^2$ and $\Lm^2_- S^4$ with cohomogeneity one symmetry $\SU(3)$ and $\Sp(2)$ respectively, imposing that $\tau_2=0$ gives a differential constraint which already forces the metric to be conformally parallel. This is somewhat suggestive that for these manifolds a more natural condition to consider is the closure of $\vp$. Indeed, Laplacian solitons have been constructed on these spaces in \cites{Haskins2021, Haskins2025}.
\end{remark}


\section{Flows of \texorpdfstring{$\rmG_2$}{}-structures}\label{sec: flows of g2}

We have seen that it is in general a very hard problem to find \textit{complete}  non-trivial examples of strong $\rmG_2T$-structures. Motivated by this, in this section we discuss possible flows of  $\rmG_2T$-structures. To the best of our knowledge the only known flows of $\rmG_2$-structures which preserve the condition $\tau_2=0$ are the Laplacian co-flow and the modified Laplacian co-flow; both in fact preserve the co-closed condition (which is in general too stringent for our purpose). We should also point out that only the modified Laplacian co-flow is known to have short time existence, see \cite{Grigorian2013short}. 
Motivated by the ideas in \cites{Bryant06someremarks,Dwivedi2023}, see also \cite{FowdarSaEarp2024}*{\S 6}, we classify here the most general $\rmG_2$-flows which are second order quasilinear and with $\tau_2=0$. 

We begin by defining the deformation tensor 
\begin{equation*}
S:=(f_0g_\vp + \bar{h} + Y) \in \langle g_\vp \rangle \oplus S^2_0(\Lm^1_7) \oplus \Lm^2_7,
\end{equation*}
so that a general flow of a $\rmG_2$-structure is expressible as
\begin{align}
    \partial_t (\vp) &= S \diamond \vp \nonumber\\
    &= 3f_0 \vp + f_3 + *_\vp(f_1 \w \vp),\label{equ: general flow vp}
\end{align}
where $f_0$ is a function, $f_3:=\bar{h} \diamond \vp \in \Lm^3_{27}$ and $f_1 \in \Lm^1_7$ is defined by $*_\vp(f_1 \w \vp)= Y \diamond \vp$. In particular, it follows that
\begin{align}
    \partial_t(g_\vp) &= S \diamond g_\vp \nonumber\\
    &= 2f_0 g_\vp + 2\bar{h},\label{equ: general flow metric}
\end{align}
and 
\begin{align}
    \partial_t(*_\vp\vp) &= S \diamond *_\vp\vp\nonumber \\ &= 4f_0 *_\vp\vp - *_\vp f_3 + f_1 \w \vp.\label{equ: general flow dualvp}
\end{align}
Thus, in order to define a $\rmG_2$-flow we need to prescribe the deformation tensor $S$, or equivalently, the differential forms $f_0$, $f_1$ and $f_3$. An explanation of how to make such choices for general $H$-flows is given in \cite{FowdarSaEarp2024}*{\S 6}, see also \cite{Fadel2022}. In our set up
the possible choices for a second order quasilinear flow of $\rmG_2T$-structures are as follows:

\begin{proposition}\label{prop: second and first order invariants g2t}
The second order invariants which can be used for a flow of $\rmG_2T$-structures are given by:
\[ \mathrm{Scal}(g_\vp),\ d\tau_0, \ \mathcal{L}_{\tau_1^\sharp}g_\vp,\ \mathrm{Ric}(g_\vp),
\]
or equivalently, by $\delta\tau_1, d\tau_0, d(*_\vp(\tau_1 \w *_\vp \vp )), d\tau_3$. 
The quadratic first order invariants which can be used for a flow of $\rmG_2T$-structures are given by:
\[
\tau_0^2,\ 
|\tau_1|^2,\  
|\tau_3|^2, \
\tau_0\tau_1,\ 
*_\vp(*_\vp(\tau_1 \w \tau_3) \w \vp),\ 
\tau_0 \tau_3,\ 
\tau_1 \w \tau_3,\ 
\tau_1 \w *_\vp (\tau_1 \w *_\vp \vp),\  \textup{\textsf{Q}}(\tau_3,\tau_3),\
\widehat{\textup{\textsf{Q}}}(\tau_3,\tau_3).
\]
\end{proposition}
\begin{proof}
    This follows from Theorem \ref{thm: second order invariants} and \ref{thm: first order invariants}, and the expressions for $\mathrm{Scal}(g_\vp)$, $\mathcal{L}_{\tau_1^\sharp}g_\vp$,  $\mathrm{Ric}(g_\vp)$ and $W_{27}$ in Section \ref{sec: g2 curvature}.
\end{proof}
\begin{remark}\
    \begin{enumerate}
    \item Observe that the Weyl curvature term $W_{27}$ does not occur in the latter proposition since when $\tau_2=0$, it can be expressed (to highest order) in terms of the $S^2_0(\Lm^1_7)$-components of $\Ric(g_\vp)$ and $\mathcal{L}_{\tau_1^\sharp}g_\vp$, see (\ref{equ: Ricci curvature general}), (\ref{equ: Weyl27 curvature general}) and (\ref{equ: lie derivative 2}). This also happens in the case of the Laplacian flow of closed $\rmG_2$-structures and is implicitly used in \cite{Bryant06someremarks}; this can also easily be seen from (\ref{equ: Ricci curvature general}) and (\ref{equ: Weyl27 curvature general}) by setting $\tau_0=\tau_1=\tau_3=0$.
    
    \item Suppose that there exists a $\rmG_2$-flow preserving the strong $\rmG_2T$ condition i.e. with both $\tau_2(t)=0$ and $dT_\vp(t)=0$. Then from the results in Section \ref{sec: g2 curvature} we deduce that the only possible second order invariant is given by $\Ric(g_\vp)$: since $d\tau_0=0$ (Corollary \ref{cor: tau0 is constant}), $\mathrm{Scal}(g_\vp)$ is expressible in terms of terms of quadratic first order invariants, see (\ref{equ: scal of strong g2t}), and the highest order terms of $\mathcal{L}_{\tau_1^\sharp}g_\vp$ and  $\Ric(g_\vp)$ are equivalent, see (\ref{equ: traceless ric strong g2t 2}) and (\ref{equ: lie derivative 2}). This is suggestive that a flow of strong $\rmG_2T$-structures is likely to have good behaviour akin to the Ricci flow. Again a similar phenomenon also happens in the case of the Laplacian flow since when $\tau_0=\tau_1=\tau_3=0$ the only possible second order invariant is $\Ric(g_\vp)$, see also \cite{Bryant06someremarks}*{Remark 11}.
    \end{enumerate}
\end{remark}
By making special choices of first and second order invariants in Proposition \ref{prop: second and first order invariants g2t}, we next suggest some possible flows worth investigating in more detail in future works.

\subsection{Flows of co-closed \texorpdfstring{$\rmG_2$}{}-structures}

Constructing $\rmG_2$-flows which preserve the $\rmG_2T$ condition appears to be a rather hard problem, however if one further imposes that the $4$-form $*_\vp\vp$ stays closed then one can define flows which preserve this cohomological condition; in this section we investigate such possible flows.

Consider the flow of \textit{co-closed} $\rmG_2$-structures defined by
\begin{align}
\partial_t(*_\vp\vp) &= - dT_\vp  \label{equ: flow dT}\\
&=  d\tau_3 -\Big(\frac{1}{6}\Big) d\tau_0 \w \vp -\frac{1}{6}\tau_0 (\tau_0 *_\vp\vp+*_\vp\tau_3).\nonumber
\end{align}
Observe that the critical points of (\ref{equ: flow dT}) correspond to co-closed strong $\rmG_2T$-structures {(recall that all the known examples of this form are $S^3 \times N^4$, where $N^4$ is a hyperK\"ahler $4$-manifold, see Example \ref{example: N4 x S3}).} Also recall from \cite{SpiroCoflow2012} and \cite{Grigorian2013short} that the Laplacian co-flow is given by 
\begin{align*}
\partial_t(*_\vp\vp) &= \Delta_\vp (*_\vp \vp) \\
&=  d\tau_3 + d\tau_0 \w \vp + \tau_0 (\tau_0 *_\vp\vp+*_\vp\tau_3)
\end{align*}
and the modified Laplacian co-flow by
\begin{align*}
\partial_t(*_\vp\vp) &= \Delta_\vp (*_\vp \vp) + 2d\big( (C-\frac{7}{4}\tau_0)\vp\big)  \\
&=  d\tau_3 -\Big(\frac{5}{2}\Big) d\tau_0 \w \vp + \Big(2C-\frac{5}{2}\tau_0 \Big)(\tau_0 *_\vp\vp+*_\vp\tau_3).
\end{align*}
The above more generally motivates us to consider the family of flows of co-closed $\rmG_2$-structures defined by
\begin{align}
\partial_t(*_\vp\vp) &= - dT_\vp + 2d\big( (C-\lm \tau_0)\vp\big) \label{equ: family of flows}  \\
&=  d\tau_3 -\Big(\frac{1}{6}+2\lm \Big) d\tau_0 \w \vp + \Big(2C- \Big(\frac{1}{6}+2\lm \Big)\tau_0\Big) (\tau_0 *_\vp\vp+*_\vp\tau_3),\nonumber
\end{align}
where $C,\lm \in \R$ are arbitrary parameters. From Proposition \ref{prop: second and first order invariants g2t} it is not hard to see that (\ref{equ: family of flows}) is the most general second order quasilinear flow of co-closed $\rmG_2$-structures with at most quadratic lower order terms. Note also that the term `$Cd\vp$' corresponds to a linear {first} order $\rmG_2$-invariant.

\begin{remark}
Observe that when $C=0$ and $\lm=-\frac{7}{12}$ we recover the Laplacian co-flow, and when $\lm=\frac{7}{6}$ we obtain the modified Laplacian co-flow.    
\end{remark}
Before addressing the problem for which parameters the flow exist, we first characterise the critical points of (\ref{equ: family of flows}). To do so we shall appeal to the following lemma:
\begin{lemma}\label{lem: key lemma mLCF}
    Given an arbitrary $\rmG_2$-structure $(M,\vp,g_\vp)$, the $\Lm^4_7$-component of $d\tau_3$ is given by
    \begin{equation*}
        *_\vp(d\tau_3) \w \vp = -3 *_\vp(d\tau_0) + \frac{3}{2} (d\tau_2) \w \vp -3 \tau_0 *\tau_1 - \frac{3}{2} \tau_1 \w \tau_2 \w \vp + 2 *_\vp(\tau_1 \w \tau_3 ) \w \vp + \frac{5}{2} \tau_2 \w  *_\vp \tau_3,
    \end{equation*}
    or equivalently, by
    \begin{equation*}
        *_\vp(d\tau_3) \w \vp = 
        -3 *_\vp(d\tau_0) 
        - 6 (d\tau_1) \w *_\vp\vp 
        -3 \tau_0 *_\vp\tau_1 
        + 2 *_\vp(\tau_1 \w \tau_3 ) \w \vp 
        +  \tau_2 \w  *_\vp \tau_3,
    \end{equation*}
    and the $\langle *_\vp\vp \rangle$-component is determined by
    \begin{equation*}
        (d\tau_3) \w \vp = |\tau_3|^2 \vol.
    \end{equation*}
In particular, if $d *_\vp\vp=0$ then
\begin{equation*}
    d\tau_3 = \frac{1}{7}|\tau_3|^3*_\vp\vp + \frac{3}{4}d\tau_0 \w \vp +(d\tau_3)^4_{27}.
\end{equation*}
\end{lemma}
\begin{proof}
    Since $d\tau_3 \in \Lm^4 \cong \Lm^4_1 \oplus \Lm^4_7 \oplus \Lm^4_{27}$ is a second order $\rmG_2$-invariant, we can express its irreducible components terms of the basis elements in Theorem \ref{thm: second order invariants} and \ref{thm: first order invariants}. Again it suffices to extract the coefficient constants as we did before.
\end{proof}
Using Lemma \ref{lem: key lemma mLCF} we can now rewrite (\ref{equ: family of flows}) equivalently as 
\begin{align}
\partial_t(*_\vp\vp)
=\ &(d\tau_3)^4_{27}+ \Big(2C- \Big(\frac{1}{6}+2\lm \Big)\tau_0\Big) *_\vp\tau_3
+\Big(\frac{7}{12}-2\lm \Big) d\tau_0 \w \vp \label{equ: family of flows 2} \\ 
&+
\Big( \frac{1}{7}|\tau_3|^2+
\Big(2C- \Big(\frac{1}{6}+2\lm \Big)\tau_0\Big)\tau_0\Big) *_\vp\vp,\nonumber
\end{align}
and hence we deduce:
\begin{proposition}\label{prop: critical point of coclosed flows}
    The critical points of (\ref{equ: family of flows}) satisfy the following:
    \begin{enumerate}
        \item $\Big(\frac{7}{12}-2\lm\Big)d\tau_0=0$.
        \item $|\tau_3|^2-7\tau_0^2\Big(2\lm+\frac{1}{6}\Big)=-14C\tau_0$.
        \item $d\tau_3=\Big(2\lm+\frac{1}{6}\Big)d\tau_0 \w \vp+\Big((2\lm+\frac{1}{6})\tau_0-2C\Big)(\tau_0*_\vp\vp+*_\vp\tau_3)$.
    \end{enumerate}
\end{proposition}
Some immediate consequences: 
\begin{itemize}
    \item If $\lm\neq \frac{7}{24}$ then the torsion forms $\tau_0$ and $\tau_3$ have constant norms. Moreover, $d*_\vp\tau_3=0$ and hence, after suitable normalisation, $*_\vp\tau_3$ defines a calibration on $M$. In particular, this applies to critical points of the modified Laplacian co-flow (indeed we do have such examples with $\tau_3\neq0$).
    \item If $C=0$ and $\lm<-\frac{1}{12}$ then the critical points of (\ref{equ: family of flows}) correspond to torsion free $\rmG_2$-structures. In particular, this applies to critical points of the Laplacian co-flow.
\end{itemize}
From Proposition \ref{prop: critical point of coclosed flows} we also observe that if $\tau_0=0$ then the critical points of (\ref{equ: family of flows}) are necessarily torsion free $\rmG_2$-structures. Thus, this suggests to investigate the behaviour of $\tau_0$ under (\ref{equ: family of flows}). 
\begin{proposition}
    Under the flow (\ref{equ: family of flows}), the torsion form $\tau_0$ evolves by
    \begin{align}
        \partial_t(\tau_0)=\ &\Big( \frac{1}{3}-\frac{8}{7}\lm\Big)\Delta_\vp \tau_0 - \frac{2}{7}g_\vp(d\tau_3,*_\vp \tau_3)
        -\frac{\tau_0}{4}\Big(\frac{|\tau_3|^2}{7}+\Big(2C-(2\lm+\frac{1}{6})\tau_0\Big)\tau_0 \Big)\label{equ: evolutiion of tau0}\\
        &-\frac{2}{7}|\tau_3|^2\Big(2C-(2\lm+\frac{1}{6})\tau_0\Big)\nonumber,
    \end{align}
where $\Delta_\vp:=\delta d$ denotes the Hodge Laplacian.
\end{proposition}
\begin{proof}
    Given a general $\rmG_2$-flow, by taking the time derivative of $\tau_0 \vp \w *_\vp\vp = d\vp \w \vp$ using (\ref{equ: general flow vp}) and (\ref{equ: general flow dualvp}), it is easy to show that 
    \begin{equation*}
    7 \partial_t(\tau_0) = 4 (\delta f_1) - 7f_0 \tau_0 + 2g_\vp(f_3,\tau_3) + 6 g_\vp(\tau_1 \w \vp, f_1 \w \vp).
\end{equation*}
In view of (\ref{equ: general flow dualvp}) and (\ref{equ: family of flows 2}) the result follows by setting $4f_0=\frac{1}{7}|\tau_3|^2+
\Big(2C- \Big(\frac{1}{6}+2\lm \Big)\tau_0\Big)\tau_0$, $f_1=\Big(\frac{7}{12}-2\lm \Big) d\tau_0$, $-*_\vp f_3=(d\tau_3)^4_{27}+ \Big(2C- \Big(\frac{1}{6}+2\lm \Big)\tau_0\Big) *_\vp\tau_3$ and simplifying the resulting expression.
\end{proof}
\begin{remark}
Observe that the purely co-closed condition `$d\vp\w\vp=0$' is not preserved in general. 
If $\lm<\frac{7}{24}$ then from (\ref{equ: evolutiion of tau0}) we see that $\tau_0$ evolves by a backwards heat type flow.
On the other hand, if $\lm>\frac{7}{24}$ and $C=0$ then one could try to apply the maximum principle to (\ref{equ: evolutiion of tau0}) to deduce that if $\tau_0(0)=0$ then $\tau_0(t)=0$, unfortunately this argument does not work due to the presence of the term $g_\vp(d\tau_3,*_\vp\tau_3)$ which is a priori is not necessarily zero. It would be interesting to see if there are specific values for the parameters $\lm$ and $C$ such that $\tau_0$ is decreasing along the flow.
\end{remark}

Next we consider the problem for what values of the parameter $\lm$ the flow (\ref{equ: family of flows}) exists. Clearly when $\lm=\frac{7}{6}$ we have short time existence since this corresponds to the modified Laplacian co-flow. So one can expect that there is an open neighbourhood for which short time existence holds, indeed we have:
\begin{theorem}\label{thm: short time existence generalisation}
    The flow (\ref{equ: family of flows}) has short time existence and uniqueness for $\mu \in \big(\frac{7}{16}(1-2\sqrt{2}),\frac{7}{16}(1+2\sqrt{2}) \big)\approx(-0.80,1.67)$, where $\mu:=\lm-\frac{7}{6}$.
    \end{theorem}
\begin{proof}
To prove the desired result we only need to show that the flow (\ref{equ: family of flows}) is weakly parabolic in the direction of closed $4$-forms; this is precisely the strategy employed in \cite{Grigorian2013short} and is itself based on \cite{BryantXu2011}. Short time existence and uniqueness then follow by repeating the same proof as in \cite{Grigorian2013short}*{\S 6}, or the simplified argument in \cite{Bedulli2020}.
For convenience, we shall follow the notation as in the proof of \cite{Grigorian2013short}*{Theorem 5.1}. First 
we can rewrite (\ref{equ: family of flows}) equivalently as 
\begin{equation}
\partial_t(*_\vp\vp) = \hat{Q}_\psi -2 \mu d(\tau_0 \vp),
\end{equation}
where $\psi:=*_\vp\vp$ and $\hat{Q}_\psi$ denotes the right hand side of the modified Laplacian co-flow. Next we need to compute the linearisation of the right hand side and show that this is weakly parabolic in the direction of a general closed $4$-form $\chi:=*_\vp(X\ip \psi + 3 \textsf{i}_\vp(h) )$, where $X$ denotes a vector field, $h \in S^2(M)$ and $\textsf{i}_\vp$ is a $\rmG_2$-equivariant isomorphism $S^2(M) \cong \langle \vp \rangle \oplus \Lm^3_{27}$.
From the proof of \cite{Grigorian2013short}*{Theorem 5.1} we have 
\begin{align*}
    \langle \sigma_\xi(D_\psi \hat{Q}_\psi)\chi ,\chi \rangle =\ &2 |\xi|^2|X|^2 + 2(\xi_aX^a)^2+\xi_m\xi^nh_{np} \vp^{mpa}X_a\\
    &+ |\xi|^2|h|^2-\xi_d\xi^a h_{ae} h^{de} - \xi_b\xi_m h_{cn} h_{dp}\vp^{bcd}\vp^{mnp}.
\end{align*}
An analogous computation as in \cite{Grigorian2013short} shows that
\begin{equation*}
\langle \sigma_\xi(D_\psi (-2 \mu d(\tau_0 \vp)))\chi ,\chi \rangle = +\frac{16}{7}\mu (|\xi|^2|X|^2+(\xi_a X^a)^2+ \xi_m\xi^nh_{np}\vp^{mpa}X_a).
\end{equation*}
If we now define
\begin{equation*}
\hat{\om}_{ab}=  \xi_m \vp^{mn}_{\quad(a}h_{b)n}+ \frac{1}{2}(1+\frac{16}{7}\mu)\xi_{(a}X_{b)}, 
\end{equation*}
then a calculation shows that
\begin{equation*}
\langle \sigma_\xi(D_\psi (\hat{Q}_\psi-2 \mu d(\tau_0 \vp)))\chi ,\chi \rangle = \Big(2+\frac{16}{7}\mu-\frac{1}{4}(1+\frac{16}{7}\mu)^2\Big) (|\xi|^2|X|^2+(\xi_a X^a)^2) + 2\hat{\om}_{ab}\hat{\om}^{ab}.
\end{equation*}
Thus, $\langle \sigma_\xi(D_\psi (\hat{Q}_\psi-2 \mu d(\tau_0 \vp)))\chi ,\chi \rangle>0$ provided that $2+\frac{16}{7}\mu-\frac{1}{4}(1+\frac{16}{7}\mu)^2>0$, which corresponds precisely to the stated range for $\mu$. When $\mu=0$, the proof reduces precisely to the one given in \cite{Grigorian2013short}.
\end{proof}
\textbf{Future problem:} Note that the range in Theorem \ref{thm: short time existence generalisation} does not include $\lm=0$ i.e. $\mu=-\frac{7}{6}$, and hence we cannot deduce short time existence of the flow (\ref{equ: flow dT}): $\partial_t(*_\vp\vp) = - dT_\vp$. 
Nonetheless, in the context of anomaly flows it might be possible to supplement this flow by coupling with the evolution of a connection 1-form. In the Hermitian case this is known as the anomaly flow, see for instance \cites{Phong18,Phong2018, Phong2024}, and is currently an active area of  research. Recently in \cite{Ashmore2024}*{\S 5} a similar flow was also proposed in the context of $\rmG_2$ geometry from a physics perspective; however, it also remains an open problem whether this flow has short time existence.

\subsection{Possible flows of strong \texorpdfstring{$\rmG_2T$}{}-structures}

Given a Hermitian manifold with an SKT metric $\om_0$, the pluriclosed flow is defined by
\begin{gather}
    \partial_t (\om) = - (\rho^B)^{1,1},\\
    \om(0) = \om_0.\nonumber
\end{gather}
The latter was introduced by Streets-Tian in \cite{StreetsTian10}, where it was also shown that the flow preserves the SKT condition i.e. $dd^c\om(t)=0$ and that $\om(t)$
can be viewed as a gauge-fixed solution to the generalised Ricci flow, see also \cite{FernandezStreetsBook}*{Proposition 9.8}. Based on this, it is natural to ask if there is a $\rmG_2$ analogue of the pluriclosed flow, i.e. if one can define a $\rmG_2$-flow which corresponds to a gauge-fixed solution to the generalised Ricci flow.

By analogy, we propose the following: 
Given a strong $\rmG_2T$-structure determined by $\vp_0$, consider the $\rmG_2$-flow defined by 
\begin{gather}
\partial_t(*_\vp\vp) = \Big( -\mathrm{Ric}(g_\vp)+\frac{1}{4}T_\vp^2\Big)\diamond (*_\vp\vp) - 4\lm \mathcal{L}_{\tau_1^\sharp}(*_\vp\vp)+S^2_7 \diamond (*_\vp\vp),\label{equ: possible strong flow}\\
\vp(0)=\vp_0,\nonumber
\end{gather}
where $T_\vp^2:= (e_j \ip e_i \ip T_\vp) \otimes (e_j \ip e_i \ip T_\vp)$, $\lm\in \R$ is a parameter and $S^2_7$ denotes an arbitrary $2$-form in $\Lm^2_7$. 
In view of (\ref{equ: ricciH}) and Theorem \ref{cor: main corollary}, it follows that characteristic Ricci-flat strong $\rmG_2T$ structures are critical points of the latter flow when $S^2_7=0$. Since it is not apriori clear that the flow (\ref{equ: possible strong flow}) preserves the strong $\rmG_2T$ condition, we shall view this as a general $\rmG_2$-flow. 
We hope to eventually be able to find a suitable choice of $S^2_7$ for which the strong condition is preserved.
Before addressing the problem of short time existence, it is useful to express the flow equation (\ref{equ: possible strong flow}) in terms of the $\rmG_2$ torsion forms since this is typically easier to compute in concrete examples.

We first recall that for a general differential form $\alpha$ and vector field $X$, the following identity holds
\begin{equation}
    \mathcal{L}_X \alpha = \nabla_X \alpha + \frac{1}{2}( \mathcal{L}_Xg_{\vp} + dX^\flat) \diamond \alpha.\label{equ: lie derivative relation}
\end{equation}
Using (\ref{equ: lie derivative relation}) with $\al=*_\vp\vp$ and $X=\tau_1^\sharp$, we can rewrite (\ref{equ: possible strong flow}) equivalently as 
\begin{gather}
\partial_t(*_\vp\vp) = \Big( -\mathrm{Ric}(g_\vp)+\frac{1}{4}T_\vp^2-2\lm \mathcal{L}_{\tau_1^\sharp}g_{\vp}\Big)\diamond (*_\vp\vp) + (S^2_7-2\lm d\tau_1) \diamond (*_\vp\vp)- 4\lm \nabla_{\tau_1^\sharp}(*_\vp\vp).\label{equ: possible strong flow 2}
\end{gather}
Since the last two terms in (\ref{equ: possible strong flow 2}) lie in $\Lm^4_7\cong \Lm^2_7$, it is easy to see from (\ref{equ: general flow metric}) that the induced metric evolves by
\begin{equation}
    \partial_t(g_\vp) = -2\mathrm{Ric}(g_\vp) + \frac{1}{2}T_\vp^2  -  \lm \mathcal{L}_{\theta^\sharp}g_\vp.
    \label{equ: possible strong flow metric}
\end{equation}
The latter corresponds precisely to one of the generalised Ricci flow equation, up to gauge fixing, cf. \cite{FernandezStreetsBook}*{Proposition 9.8 (3)}. Indeed any $\rmG_2$-flow inducing the gauge-fixed solution (\ref{equ: possible strong flow metric}) to the generalised Ricci flow is necessarily of the form (\ref{equ: possible strong flow}), hence our proposed flow.

Finally, we can rewrite (\ref{equ: possible strong flow 2}) without the operator $\diamond$ and any covariant derivative term as follows:
\begin{proposition}
    The $\rmG_2$-flow (\ref{equ: possible strong flow}) can be equivalently expressed as
    \begin{align}
    \partial_t(*_\vp\vp) 
    =\ &\Big(
    \frac{16(\lm-3)}{7} \delta\tau_1 - \frac{4}{3}\tau_0^2-\frac{96}{7}|\tau_1|^2+\frac{8}{7}|\tau_3|^2+\frac{2}{7}|\tau_2|^2
    \Big) *_\vp\vp \label{equ: possible strong flow 3}\\
    &+\pi^4_{27}\Big(d\tau_3+
    (4\lm-5)*_\vp d \big(*_\vp (\tau_1 \w *_\vp\vp)\big)
    +(11-12\lm)*_\vp(\tau_1 \w *_\vp(\tau_1 \w *_\vp \vp))\nonumber\\
    &-\frac{1}{6}\tau_0 *_\vp\tau_3 
    +(5-4\lm) \tau_1 \w \tau_3-d\tau_2+\tau_1 \w \tau_2 +\frac{1}{2}\tau_2 \w \tau_2
    \Big)\nonumber\\
    &+\Big(f^1_7+\lm \big( 
    2*_\vp(d\tau_1 \w *_\vp \vp) + 
    \tau_0 \tau_1 
    -2 *_\vp(\tau_1 \w \tau_2 \w \vp)
    -2*_\vp(*_\vp(\tau_1 \w \tau_3) \w \vp)
    \big) \Big)\w \vp, \nonumber
\end{align}
where $f^1_7$ is a $1$-form defined by $S^2_7\diamond *_\vp\vp=f^1_7 \w \vp$.
\end{proposition}
\begin{proof}
Using our representation theoretic method, one can show that
\begin{equation*}
    \nabla_{\tau_1^\sharp}(*_\vp\vp) = -\frac{1}{4}\tau_0 \tau_1\w \vp+\frac{1}{2} *_\vp (\tau_1 \w \tau_2 \w \vp)\w \vp +\frac{1}{2} *_\vp(*_\vp(\tau_1 \w \tau_3) \w \vp)\w \vp 
\end{equation*}
and thus, we can rewrite (\ref{equ: possible strong flow 2}) without any covariant derivative term as:
\begin{align*}
\partial_t(*_\vp\vp) =\ &\Big( -\mathrm{Ric}(g_\vp)+\frac{1}{4}T_\vp^2-2\lm \mathcal{L}_{\tau_1^\sharp}g_{\vp}\Big)\diamond (*_\vp\vp) + (S^2_7-2\lm d\tau_1) \diamond (*_\vp\vp)\nonumber \\
&+\lm\big(\tau_0 \tau_1 -2 *_\vp(\tau_1 \w \tau_2 \w \vp)-2*_\vp(*_\vp(\tau_1 \w \tau_3)\w \vp)\big)\w \vp.
\end{align*}
The result now follows by a direct calculation using (\ref{equ: Ricci curvature general}), (\ref{equ: lie derivative 2}) and that $T^2_\vp \in S^2(M)$ can be identified with a $4$-form by
   \begin{gather*}
    \pi^4_1(T_\vp^2\diamond (*_\vp\vp)) = \frac{24}{7}|T_\vp|^2*_\vp\vp, \\
    \pi^4_7(T_\vp^2\diamond (*_\vp\vp)) = 0, \\
    \pi^4_{27}(T_\vp^2\diamond (*_\vp\vp)) = 4\cdot \pi^4_{27}\Big(
    *_\vp(\tau_1 \w *_\vp(\tau_1 \w *_\vp \vp)) 
    +\frac{1}{3}\tau_0 *_\vp\tau_3 
    +2 \tau_1 \w \tau_3
    -\frac{1}{16}*_\vp \textsf{Q}(\tau_3,\tau_3)
    \Big).
\end{gather*}
\end{proof}
While the flow equation (\ref{equ: possible strong flow 3}) might appear rather involved, it is typically much easier to compute in practice than (\ref{equ: possible strong flow}) since it avoids having to compute the operator $\diamond$. Furthermore, (\ref{equ: possible strong flow 3}) provides a more efficient way to study the flow for different choices of the parameters $\lm$ and $S^2_7$. In particular, we observe that by carefully choosing $\lm$ and $S^2_7$ one can modify certain highest order terms, which is important in studying the existence of the flow.

Next we wish to establish when the flow (\ref{equ: possible strong flow}) exists for short time. To do so we need to consider the linearisation of the flow equation and compute its principal symbol to see when it is parabolic (modulo diffeomorphism). Since this was recently done in \cite{Dwivedi2023} for general quasilinear second order $\rmG_2$-flows, we want to appeal to their result. However, the approach therein differs from ours, so we first need to gather some preliminary results. We begin by defining the torsion $2$-tensor $\hat{T}$\footnote{This is denoted by $T$ in \cite{Dwivedi2023}*{(2.55)}. Note also that there are two conventions in the literature to define the orientation induced by the $\rmG_2$ $3$-form $\vp$ and the one in \cite{Dwivedi2023} is opposite to ours as such we need to modify some signs.} by
\begin{equation}
    \hat{T}_{pq}:= -\frac{1}{24}\nabla_p \vp_{jkl}(*_\vp\vp)_{qjkl}.
\end{equation}
In \cite{Dwivedi2023}, the following quantities are defined
\begin{equation}
    (\textsf{V}\hat{T})_j := \hat{T}_{pq}\vp_{pqj}, \qquad 
    \mathrm{div}(\hat{T})_j := \nabla_i\hat{T}_{ij}, \qquad \mathrm{div}(\hat{T}^t)_j := \nabla_i\hat{T}_{ji}. \label{equ: definition of harmonic}
\end{equation}
While the above expressions are convenient for analytical computations, these are rather tedious to compute in concrete examples. Using our representation theoretic technique, we can show the following:
\begin{proposition}\label{prop: divT}
Given a general $\rmG_2$-structure $(M,\vp,g_\vp)$, the following holds:
\begin{gather*}
    \textup{\textsf{V}}\hat{T} = 6 \tau_1,\\
    \mathrm{div}(\hat{T}) = -\frac{7}{4}d\tau_0-6*_\vp(d\tau_1 \w *_\vp\vp) + 
    3 *_\vp (*_\vp(\tau_1 \w \vp)\w \tau_3) -\frac{3}{2} \tau_0 \tau_1 + 3*_\vp(\tau_1 \w \tau_2 \w \vp),\\
    \mathrm{div}(\hat{T}^t) = -\frac{7}{4}d\tau_0+
    3 *_\vp (*_\vp(\tau_1 \w \vp)\w \tau_3) -\frac{3}{2} \tau_0 \tau_1 
    -3*_\vp(\tau_1 \w \tau_2 \w \vp).
\end{gather*}
\end{proposition}
An immediate corollary of Proposition \ref{prop: divT} is    \cite{Grigorian2019}*{Theorem 4.3}, namely, that $\mathrm{div}(\hat{T})=0$ when either $\tau_0$ is constant and $\tau_1=0$, or $\tau_0=\tau_2=\tau_3=0$. Moreover, Proposition \ref{prop: divT} shows that the same statement also applies to $\mathrm{div}(\hat{T}^t)$. $\rmG_2$-structures, or more generally $H$-structures, satisfying the condition $\mathrm{div}(\hat{T})=0$ are sometimes called harmonic structures in the literature. The terminology ``harmonic'' stems from the fact that these $H$-structures arise as critical points of a Dirichlet energy functional, see \cite{LoubeauSaEarp2019}*{\S 1.4}. Proposition \ref{prop: divT} provides a much simpler formula for computing the harmonic condition compared to (\ref{equ: definition of harmonic}) and hence we hope will lead to the construction of new examples of harmonic $\rmG_2$-structures. In particular, we deduce:

\begin{corollary}
    If $(M,\vp,g_\vp)$ is a $\rmG_2T$ manifold then $\mathrm{div}(\hat{T})=\mathrm{div}(\hat{T}^t)$. If furthermore, $\tau_0$ is constant then 
    \[
    \mathrm{div}(\hat{T})=\frac{3}{4}*_\vp(\delta T_\vp \w *_\vp \vp ).\]
    In particular, strong $\rmG_2T$-structures with $\delta T_\vp=0$ are harmonic in the sense of \cite{LoubeauSaEarp2019}.
\end{corollary}
\begin{proof}
    The first part is immediate from Proposition \ref{prop: divT}. The second part follows from a simple computation using Corollary \ref{prop: codifferential of T}.
\end{proof}
We next state the main result of \cite{Dwivedi2023} using our conventions:
\begin{theorem}\label{thm: spiro}
    Let $(M,\vp_0,g_{\vp_0})$ denote a compact $7$-manifold with a $\rmG_2$-structure determined by $\vp_0$. Consider the flow defined by
    \begin{gather*}
\partial_t(\vp) = \Big( -\mathrm{Ric}(g_\vp)+6a\mathcal{L}_{\tau_1^\sharp}g_{\vp}\Big)\diamond \vp +
\Big(
\frac{7}{4}(b_1+b_2) d\tau_0 + 6b_1 *_\vp(d\tau_1 \w *_\vp\vp)
\Big)^\sharp \ip *_\vp\vp + \text{l.o.t},\\
\partial_t(*_\vp\vp) = \Big( -\mathrm{Ric}(g_\vp)+6a\mathcal{L}_{\tau_1^\sharp}g_{\vp}\Big)\diamond (*_\vp\vp) - \Big(
\frac{7}{4}(b_1+b_2) d\tau_0 + 6b_1 *_\vp(d\tau_1 \w *_\vp\vp)
\Big) \w \vp+ \text{l.o.t},\\
\partial_t(g_\vp) =  -2\mathrm{Ric}(g_\vp)+12a\mathcal{L}_{\tau_1^\sharp}g_{\vp}+ \text{l.o.t},\\
\vp(0)=\vp_0,
\end{gather*}
    and suppose that $0\leq b_1 - a-1<4$ and $b_1+b_2\geq 1$.
    Then there exists an $\epsilon>0$ and a unique smooth solution $\vp(t)$ for $t\in [0,\epsilon)$.
\end{theorem}
\begin{proof}
    From \cite{Dwivedi2023}*{Theorem 6.76} we have that the flow
    \begin{equation}
        \partial_t(\vp) = \Big( -\mathrm{Ric}(g_\vp)+a\mathcal{L}_{\textsf{V}\hat{T}^\sharp}g_{\vp}\Big)\diamond \vp -
        (
        b_1 \mathrm{div}(\hat{T}) +
        b_2 \mathrm{div}(\hat{T}^t)
        )^\sharp \ip *_\vp\vp + \text{l.o.t}
    \end{equation}
    has short time existence and uniqueness for $0\leq b_1 - a-1<4$ and $b_1+b_2\geq 1$. The result now follows from Proposition \ref{prop: divT}.
\end{proof}

\begin{corollary}\label{cor: main result flow}
    The $\rmG_2$-flow (\ref{equ: possible strong flow}), inducing the gauge-fixed generalised Ricci flow (\ref{equ: possible strong flow metric}), has a unique smooth solution for short time provided that $S^2_7 \in \Lm^2_7$ is of the form
    \begin{equation*}
        S^2_7 = \lm_1 (d\tau_1)^2_7 + \frac{7}{12}\lm_2 *_\vp(d\tau_0 \w *_\vp\vp) + \text{l.o.t},
    \end{equation*}
    where $6 \leq \lm_1  <30$ and $\lm_2 \geq 1$.
\end{corollary}
\begin{proof}
    The result follows from a simple computation using Theorem \ref{thm: spiro} and comparing with (\ref{equ: possible strong flow 2}).
\end{proof}
Note that we did not impose the constraint that the flow (\ref{equ: possible strong flow}) preserves the condition $\tau_2=0$ nor $dT_\vp=0$. Indeed the flow has short time existence even if the initial data is not a strong $\rmG_2T$-structure. Thus, one can investigate the evolution of the flow in general. As already mentioned above, in \cite{StreetsTian10} Streets-Tian showed that the pluriclosed flow preserves the SKT condition, and so one might hope that by suitably choosing the $2$-form $S^2_7$, whereby the lower order terms are chosen from Proposition \ref{prop: second and first order invariants g2t},  one can also show that the strong $\rmG_2T$ condition is preserved. 
It is worth pointing out that if such a flow of strong $\rmG_2T$-structures exists, then from Corollary 
\ref{cor: tau0 is constant} we know that $S^2_7$ can only consists of lower order terms and as such this flow will be a `Ricci like flow'. 
\textbf{Future question:} It is an interesting problem to see if there are choices for $S^2_7$ which induce the flow
\[
\partial_t(T_\vp) = -\Delta_\vp T_\vp - 4 \lm\mathcal{L}_{\tau_1^\sharp}T_\vp, 
\]
for the torsion form $T_\vp$. The latter corresponds precisely to the second equation of the generalised Ricci flow (up to gauge fixing), see \cite{FernandezStreetsBook}*{Proposition 4.22 and Remark 4.23}. The difficulty here is that $\vp$ determines both $g_\vp$ and $T_\vp$, hence the computation is much more involved. In the Hermitian case, things are a bit simpler since the flow preserves the complex structure $J$ (which is essentially independent of $\om$), whereas in the $\rmG_2$ case it is more involved to show if the analogous condition $\tau_2=0$ is preserved as this is itself dependent on $\vp$.





\section{Appendix}

In this section we record the classification of (quadratic) first and second order $\rmG_2$-invariants; this will play a fundamental role in this article. We should mention that the results stated here were already known to Bryant, who employed them in \cite{Bryant06someremarks} to derive many fundamental formulae in $\rmG_2$ geometry (although the precise statement given below was not specifically stated therein). 

The space of first order $\rmG_2$-invariants $V_1(\mathfrak{g}_2)$ is given by
\begin{equation*}
    V_1(\mathfrak{g}_2) \cong \mathbb{R} \oplus  \Lm^1_7 \oplus \Lm^2_{14} \oplus \Lm^3_{27}
\end{equation*}
and the space of second order $\rmG_2$-invariants $V_2(\mathfrak{g}_2)$ by
\begin{equation*}
    V_2(\mathfrak{g}_2) \cong \R \oplus 2 \Lm^1_7 \oplus \Lm^2_{14} \oplus 3 \Lm^3_{27} \oplus 2 V_{1,1} \oplus V_{0,2} \oplus V_{3,0},
\end{equation*}
where $V_{i,j}$ denotes the irreducible $\rmG_2$ module with highest weight $(i,j)$. This was first computed by Bryant in \cite{Bryant06someremarks}*{(4.6),(4.7)}. Of course, $V_{1}(\mathfrak{g}_2)$ is just the space of intrinsic torsion and is spanned by $\tau_0,\tau_1,\tau_2,\tau_3$. Similarly, one can show:

\begin{theorem}[Second order $\rmG_2$-invariants]\label{thm: second order invariants}\ \\
    The $\mathbb{R}$ module in $V_{2}(\mathfrak{g}_2)$ is generated by:
    \begin{itemize}
        \item $\delta \tau_1$
    \end{itemize}
    The 2 copies of $\Lm^1_7$ in $V_{2}(\mathfrak{g}_2)$ are generated by:
    \begin{itemize}
        \item $d\tau_0$.
        \item $*_\vp(d\tau_1 \w *_\vp\vp)$
    \end{itemize}
    The $\Lm^2_{14}$ module in $V_{2}(\mathfrak{g}_2)$ is generated by
    \begin{itemize}
        \item $(d\tau_1)^2_{14}$.
    \end{itemize}
    The 3 copies of $\Lm^3_{27}$ in $V_{2}(\mathfrak{g}_2)$ are generated by:
    \begin{itemize}
        \item $\big(d(*_\vp(\tau_1 \w *_\vp \vp ))\big)^3_{27}$.
        \item $(d\tau_2)^3_{27}$.
        \item $(d\tau_3)^4_{27}$.
    \end{itemize}
\end{theorem}
\begin{proof}
    Given that we already know the irreducible components of $V_2(\mathfrak{g}_2)$, the proof essentially amounts to showing that the stated differential forms are linearly independent. This is done by a careful inspection.
\end{proof}
Next we shall need the quadratic terms in the intrinsic torsion. A straightforward computation shows that
\begin{equation}
    S^2(V_1(\mathfrak{g}_2)) \cong 4\R \oplus 4\Lm^1_7 \oplus 3\Lm^2_{14} \oplus 8 \Lm^3_{27} \oplus 4V_{1,1}
    \oplus V_{2,1}
    \oplus 2V_{0,2}
    \oplus 2 V_{3,0}
    \oplus V_{4,0}\label{equ: sym2 v1g2}
\end{equation}
As before we have:
\begin{theorem}[Quadratic first order $\rmG_2$-invariants]\label{thm: first order invariants}\ \\
    The 4 copies of $\mathbb{R}$ in $S^2(V_1(\mathfrak{g}_2))$ are generated by:
    \begin{itemize}
        \item $\tau_0^2$, $|\tau_1|^2$, $|\tau_2|^2$, $|\tau_3|^2$.
    \end{itemize}
    The 4 copies of $\Lm^1_7$ in $V_{2}(\mathfrak{g}_2)$ are generated by:
    \begin{itemize}
        \item $\tau_0\tau_1$.
        \item $*_\vp(\tau_1 \w \tau_2 \w \vp)$.
        \item $*_\vp(*_\vp(\tau_1 \w \tau_3) \w \vp)$.
        \item $*_\vp(*_\vp(\tau_2 \w \tau_3) \w *_\vp\vp)$
    \end{itemize}
    The 3 copies of $\Lm^2_{14}$ in $V_{2}(\mathfrak{g}_2)$ are generated by
    \begin{itemize}
        \item $\tau_0 \tau_2$.
        \item $\big(*_\vp(\tau_2 \w \tau_3)\big)^2_{14}$.
        \item $\big(*_\vp(\tau_1 \w *_\vp \tau_3)\big)^2_{14}$.
    \end{itemize}
    The 8 copies of $\Lm^3_{27}$ in $V_{2}(\mathfrak{g}_2)$ are generated by:
    \begin{itemize}
        \item $\tau_0 \tau_3$.
        \item $(\tau_1 \w \tau_3)^4_{27}$.
        \item $(\tau_1 \w \tau_2)^3_{27}$.
        \item $(\tau_1 \w *_\vp (\tau_1 \w *_\vp \vp))^3_{27}$.
        \item $(\tau_2 \w \tau_2)^4_{27}$.
        \item $\textup{\textsf{Q}}(\tau_3,\tau_3)$, where $\textup{\textsf{Q}}$ is a bilinear quadratic pairing defined by
        \begin{equation*}
        \textup{\textsf{Q}}(\al,\beta)= *_\vp \big( (*_\vp\vp)_{ijkl} (e_j\ip e_i\ip *_\vp \alpha) \w (e_l\ip e_k\ip *_\vp \beta) \big).
        \end{equation*}
        \item $\widehat{\textup{\textsf{Q}}}(\tau_2,\tau_3)$, where $\widehat{\textup{\textsf{Q}}}$ is a bilinear quadratic pairing defined by
        \begin{equation*}
        \widehat{\textup{\textsf{Q}}}(\al,\beta)= *_\vp \big( (e_i \ip \alpha) \w (e_i \ip \beta) \big).
        \end{equation*}
        \item $\widehat{\textup{\textsf{Q}}}(\tau_3,\tau_3)$.
    \end{itemize}
\end{theorem}
\begin{proof}
Use the same argument as before.
\end{proof}

It follows that any second order $\rmG_2$ invariant is necessarily given by a linear combination of the above basis elements. Indeed this method was used by Bryant to derive a formula for the Ricci curvature of general $\rmG_2$-structures, see \cite{Bryant06someremarks}*{(4.28),(4.30)}. The same strategy was employed by Bedulli-Vezzoni in \cite{Bedulli2007} to derive analogous formulae for $\SU(3)$-structures. Of course, this method applies to more general second order invariants (not just Ricci curvature), but surprisingly this approach has not been utilised much in the literature to the best of our knowledge, see also \cite{FowdarSaEarp2024}.

\bibliography{biblioG}
\bibliographystyle{numeric}

\end{document}